\documentclass[11pt,letterpaper]{amsart}
\usepackage[margin=1.25in]{geometry}

\usepackage{mathrsfs}
\usepackage{amsfonts}
\usepackage{latexsym,epsfig}
\usepackage{amsmath, amscd, amsthm,amssymb}
\usepackage{mathabx, stmaryrd}

\usepackage{multicol}

\usepackage[pdftex]{color}
\usepackage{comment}
\usepackage{marginnote}
\usepackage[colorlinks,citecolor=cyan,linkcolor=magenta]{hyperref}
\usepackage[nameinlink]{cleveref}
\usepackage{enumitem}
\usepackage[mathscr]{euscript}
\usepackage[normalem]{ulem}
\usepackage{hyperref}

\newcommand{\DD}{\mathbb{D}}
\newcommand{\HM}{{\scriptscriptstyle{\mathrm{HM}}}}

\newcommand{\ZZ}[0]{\mathbb{Z}}

\newcommand{\R}{\mathbb{R}}

\newcommand{\wt}{\widetilde}
\newcommand{\mc}{\mathcal}
\newcommand{\ms}{\mathscr}

\newcommand{\ol}{\overline}
\newcommand{\wh}{\widehat}

\newcommand{\del}{\partial}

\newcommand{\hbs}{\tau^{(2)}}

\newcommand{\cW}{\mathcal{W}}
\newcommand{\db}{\partial^2}

\newcommand{\FF}{\mathcal{F}}

\newcommand{\C}{\mathcal{C}}

\newcommand{\orb}{\mathcal{O}}
\newcommand{\define}[1]{\textbf{#1}}

\newcommand{\union}{\cup}
\newcommand{\intersect}{\cap}
\newcommand{\boundary}{\partial}

\newcommand\tsim{\kern-.4em\sim}

\newcommand\ssm{\smallsetminus}

\newcommand{\cM}{\mathring{M}}

\newcommand{\ent}{\mathrm{ent}}
\newcommand{\gr}{\mathrm{gr}}
\newcommand{\cut}{\! \bbslash \!}

\renewcommand{\d}{\downarrow}
\renewcommand{\phi}{\varphi}
\renewcommand{\epsilon}{\varepsilon}

\DeclareMathOperator{\intr}{int}

\DeclareMathOperator{\lcm}{lcm}

\newtheorem{thm}{Theorem}

\newtheorem{theorem}{Theorem}[section]
\newtheorem{lemma}[theorem]{Lemma}
\newtheorem{proposition}[theorem]{Proposition}
\newtheorem{corollary}[theorem]{Corollary}

\newtheorem{claim}[theorem]{Claim}
\newtheorem{definition}[theorem]{Definition}

\theoremstyle{definition}

\newtheorem{remark}[theorem]{Remark}

\begin{document}

\title{Endperiodic maps via pseudo-Anosov flows}

\author[M.P. Landry]{Michael P. Landry}
\address{Department of Mathematics\\
Washington University in Saint Louis }
\email{\href{mailto:mlandry@wustl.edu}{mlandry@wustl.edu}}
\author[Y.N. Minsky]{Yair N. Minsky}
\address{Department of Mathematics\\ 
Yale University}
\email{\href{mailto:yair.minsky@yale.edu}{yair.minsky@yale.edu}}
\author[S.J. Taylor]{Samuel J. Taylor}
\address{Department of Mathematics\\ 
Temple University}
\email{\href{mailto:samuel.taylor@temple.edu}{samuel.taylor@temple.edu}}
\thanks{This work was partially supported by the NSF postdoctoral fellowship DMS-2013073, NSF grants DMS-2005328,
DMS-2102018, and the Sloan Foundation.}
\maketitle

\begin{abstract}
We show that every atoroidal endperiodic map of an infinite-type surface can be obtained from a depth one foliation in a fibered hyperbolic 3-manifold, reversing a well-known construction of Thurston. This can be done almost-transversely to the canonical suspension flow, and as a consequence we recover the Handel--Miller laminations of such a map directly from the fibered structure. We also generalize from the finite-genus case the relation between topological entropy, growth rates of periodic points, and growth rates of intersection numbers of curves. Fixing the manifold and varying the depth one foliations, we obtain a description of the Cantwell--Conlon foliation cones and a proof that the entropy function on these cones is continuous and convex. 
\end{abstract}

\setcounter{tocdepth}{1}
\tableofcontents

\section{Introduction}
\label{sec:intro}

Let $L$ be an infinite-type surface with finitely many ends and without boundary.
A homeomorphism $g\colon L \to L$ is \define{endperiodic} if each end of $L$ is 
either attracting or repelling under a power of $g$, and
\define{atoroidal} if it fixes no finite essential
multicurve up to isotopy (see \Cref{sec:endperiodic} for precise definitions).

Such maps appear naturally in Thurston's work on fibered compact 3-manifolds: 
A fiber $S$ can be ``spun'' 
around a sufficiently nice surface $\Sigma$
yielding a foliation in which $\Sigma$ is a
compact leaf and its complement is fibered by parallel copies of a noncompact surface $L$, so that the
monodromy of this fibering is an endperiodic map of $L$, which must be atoroidal when $M$ is hyperbolic.

In this paper we reverse this process, obtaining any atoroidal endperiodic map 
from some fibration by a spinning operation in a suitable hyperbolic fibered 3-manifold.
More importantly, the resulting foliation
is transverse to a pseudo-Anosov flow (as in Fried \cite{Fri79}),
and the stable and unstable foliations of this flow induce a similar structure on $L$ (see \Cref{th:intro_spA} and \Cref{th:intro_spA+}). We call
the return map of such a construction a {\em spun pseudo-Anosov (spA) map}. 
\smallskip

This construction has several consequences:

\begin{itemize}
\item We recover, directly from the pseudo-Anosov structure, the dynamical laminations of
  Handel--Miller. See \Cref{th:intro_HM} and 
\cite{fenley1992asymptotic, Fen97,cantwell1999foliation,CCF19}.

 \item We identify dynamical growth rates of the spun pseudo-Anosov  map: the spA map minimizes the exponential growth rate of periodic points among all homotopic endperiodic maps. Further, this rate is equal to the exponential growth rate of intersection numbers of curves under iteration and its log is the topological entropy (suitably defined) of the spA map. See \Cref{th:intro_stretch}.
  
\item 
  The compactified mapping torus of the endperiodic map is a manifold $N$ with boundary, which
  can admit a variety of depth one foliations whose compact leaves are $\boundary N$. These
  foliations are parameterized by the {\em foliation cones} of Cantwell--Conlon,
  which are analogous to the cones on fibered faces of Thurston's norm. This analogy can
  be made explicit by the spinning construction, and we show that the foliation cones are
  exactly the pullbacks of Thurston fibered cones by the inclusion of
  $N$ into a certain fibered manifold $M$ (\Cref{th:intro_spun}). 
 From this we can show that
  topological entropy defines a continuous, convex function on each foliation cone
  (\Cref{th:intro_entropy}).   
  \end{itemize}

\subsection{Motivation from fibered face theory}
\label{sec:motivation}
Let $M$ be a closed hyperbolic manifold that fibers over the circle.
Each fiber $S$ in a fibration 
of $M$ over the circle determines (and is determined by)
its homology class, which is Poincar\'e dual to an integral element of $H^1(M;\R)$
called a \emph{fibered class}.
The various fibrations of $M$ are organized by the \emph{Thurston norm}. This is a norm $x$ on the vector space $H^1(M;\R)$ whose unit ball $B_x = B_x(M)$ is a finite sided polyhedron having the following property:
there is a set of open cones over top-dimensional faces of $B_x$, called \emph{fibered cones}, so that each primitive integral point in the fibered cone $\C$ 
is a fibered class, and hence corresponds to a fiber in some fibration of $M$ over the circle. Conversely, each fibered class lies interior to a fibered cone \cite{thurston1986norm}.

Fried later reinterpreted Thurston's theory in terms of the \define{pseudo-Anosov suspension flow}
 $\phi = \phi_{\mc C}$ that is canonically associated to the fibered cone $\mc C$, up to isotopy and reparametrization \cite{Fri79}. Each fiber surface $S$ dual to a class in $\mc C$ is isotopic to a \emph{cross section} of the flow $\phi$, i.e. it is transverse and meets each orbit infinitely often,  and the associated first return map to $S$ is the pseudo-Anosov representative of its monodromy. 

Since any fibered class $\alpha$ in $\mc C$ determines a pseudo-Anosov first return map $f \colon S \to S$  on the associated cross section $S$ of $\phi$, we can assign to $\alpha$ the logarithm of the stretch factor of $f$, i.e. its \emph{entropy}, and 
denote this quantity by $\ent(\alpha)$. Fried proves that this assignment extends to a function $\ent \colon \mc C \to (0, \infty)$ that is continuous, convex, and blows up at the boundary of $\mc C$ \cite[Theorem E]{fried1982flow}. 
McMullen extends Fried's result by showing that $\ent$ is additionally real analytic and strictly convex \cite[Corollary 5.4]{mcmullen2000polynomial}. Understanding properties of the entropy function has since figured prominently into both the study of fibered cones as well as pseudo-Anosov stretch factors \cite{farb-leininger-margalit, kin2013minimal, hironaka2010small, leininger2013number, mcmullen2015entropy}.

\subsubsection*{At the boundary}
The structure discussed above applies only to fiber surfaces, i.e. those representing classes in the interior of the fibered cone. However, in the article where his norm is introduced, Thurston explains the geometric significance of a 
taut surface $\Sigma$ representing a class in the \emph{boundary} of the fibered cone $\C$. Recall that $\Sigma$ is \define{taut} if it has no nullhomologous components and it is norm minimizing, that is $x([\Sigma]) = -\chi(S)$.
According to \cite[Remark, p. 121]{thurston1986norm}, there are foliations $\mc F$ of $M$ having $\Sigma$ as the collection of compact leaves such that the restriction of the foliation to $M \ssm \Sigma$ is a fibration over $S^1$. In particular, $\mc F$ is a taut, cooriented, depth one foliation of $M$.

\smallskip
Our work here is partially motivated by understanding the interaction between Thurston's foliations and the canonical pseudo-Anosov suspension flow $\varphi_\C$. First, after replacing $\phi_\C$ with a dynamic blowup $\phi$ (see \Cref{sec:blowup}), which modifies $\phi_\C$ only at its singular orbits, we may isotope $\Sigma$ so that it is positively transverse to $\phi$. This is a consequence of our Strong Transverse Surface Theorem (\Cref{th:stst}), which strengthens a result of Mosher that established the existence of such a transverse surface in the homology class $[\Sigma]$.

Next, for any cross section $S$ of $\phi$ (i.e. any fiber surface in the associated fibered cone, up to isotopy) 
we can spin $S$ about $\Sigma$ to obtain a infinite-type surface $L$ that is transverse to $\varphi$, meets every flow line infinitely often in the forwards and backwards direction, and accumulates only on $\Sigma$; see \Cref{sec:spinning}. Hence, there is a well-defined first return map $f \colon L \to L$ along the flow $\phi$; we call any such map obtained by this construction \define{spun pseudo-Anosov} (see \Cref{def:spA} for a more formal definition). It follows immediately that $M\ssm \Sigma$ can be identified with the mapping torus of $f$ and hence inherits a foliation by $L$--fibers. The foliation $\mc F$ of $M$ is then obtained by adjoining the components of $\Sigma$ as leaves. By construction, the foliation $\mc F$ is transverse to $\phi$.

\smallskip
The spun pseudo-Anosov map $f \colon L \to L$ inherits additional structure from the flow $\phi$. 
Since $\phi$ has invariant unstable/stable singular foliations $W^{u/s}$, their intersections $\cW^{u/s} = L \cap W^{u/s}$ are a pair of $f$--invariant singular foliations of $L$. Moreover, the expanding/contracting dynamics of $\phi$ along its periodic orbits impose corresponding dynamics at the periodic points of $f$ (see \Cref{sec:sink-source} for details).

This structure is the basis for the properties of spA maps listed above (and more formally detailed below).

\subsubsection*{Motivation from big mapping class groups}
There has been recent interest in the study of big mapping class groups, i.e. mapping class groups of infinite-type surfaces.  A significant component of this is understanding the extent to which there is a Nielsen--Thurston---like classification of elements of big mapping class groups (see e.g. \cite[Problem 1.1]{MCG_questions}). From this point of view, spun pseudo-Anosov maps of infinite-type surfaces, at least for surfaces with finitely many ends, offer such a normal-form representative for atoroidal endperiodic maps. For such maps, the second item of \Cref{th:intro_stretch} answers \cite[Problem 1.5]{MCG_questions}. 

We also remark that atoroidal, endperiodic maps are central to recent work that relates the hyperbolic geometry of their mapping tori to combinatorial invariants \cite{field2021end, FKLL}.

\subsection{Main results}
\label{sec:intro_results}
Inspired by Nielsen--Thurston theory for finite-type mapping classes, one could ask which homeomorphisms $L \to L$ are isotopic to spA maps. Our answer is that the obvious necessary conditions are also sufficient:

\begin{thm}
\label{th:intro_spA}
An endperiodic homeomorphism $g \colon L \to L$ is isotopic to an spA representative $f \colon L \to L$ if and only if it is atoroidal.
\end{thm}

See \Cref{th:exists} and \Cref{th:spA_construction}.

The proof shows how to construct, given the compactified mapping torus $N$ of $g$, a closed hyperbolic manifold $M$ with a pseudo-Anosov suspension flow $\phi$ and an embedding $N \to M$ that maps $L$ to a leaf of a depth one foliation $\mc F_L$ of $M$ that is transverse to a dynamic blowup of $\phi$. The required spA map $f \colon L \to L$ is then the first return map to the leaf $L$ of $\mc F_L$.

The construction of $M$ and $\phi$ are fairly flexible and this leads to various
strengthenings of \Cref{th:intro_spA}. In particular (from \Cref{th:no_blow}), 

\begin{thm}\label{th:intro_spA+}
  An atoroidal endperiodic homeomorphism $g \colon L \to L$ is isotopic to an spA
  representative $f \colon L \to L$ which is honestly transverse to the pseudo-Anosov suspension flow.
\end{thm}

That is, one can vary the construction so that no dynamic blowup is needed. We call such a map ``spA$^+$'' -- see \Cref{def:spA+}.

\medskip

For any homeomorphism $g \colon L \to L$, we define its \define{growth rate}  
$\lambda(g)$ as the exponential growth rate of its periodic points:
\[
\lambda(g) = \limsup_{n \to \infty} \sqrt[n]{\# \mathrm{Fix}(g^n)},
\]
where $\mathrm{Fix}(g)$ denotes the set of fixed points of $g$. The following theorem gives various characterizations of the growth rate of an spA map, generalizing what is known about pseudo-Anosov homeomorphisms of finite-type surfaces. In its statement, $i(\alpha,\beta)$ denotes the geometric intersection number of curves $\alpha$ and $\beta$.

\begin{thm}[Characterizing stretch factors]
\label{th:intro_stretch}
Let $f \colon L \to L$ be spA. Then

\begin{enumerate}
\item $\lambda(f) = \inf_{g \simeq f} \lambda(g)$, over all homotopic endperiodic maps $g$.

\item $\lambda(f)  = \max_{\alpha,\beta} \limsup_{n\to\infty} \sqrt[n]{i(\alpha,f^n(\beta))}$, 
where $\alpha,\beta$ are essential simple closed curves on $L$.

\item $\log \lambda(f)$ is the topological entropy of the restriction of $f$ to the (unique) largest invariant compact set.
\end{enumerate}
\end{thm}

See \Cref{th:ent_spA}, \Cref{cor:spA_growth}, and \Cref{th:intersect}. 

Because of the parallels drawn by \Cref{th:intro_stretch} between $\lambda(f)$ and the stretch factor of a pseudo-Anosov surface homeomorphism, we call $\lambda(f)$ the \define{stretch factor} of the spA map $f$. 

\smallskip

We note here that the first item in \Cref{th:intro_stretch} follows from a stronger result: if $g\colon L\to L$ is endperiodic and isotopic to the spA map $f$, then for any $f$--periodic point of period $n$ there is a Nielsen equivalent $g$--periodic point also of period $n$.
See \Cref{sec:growth} for details. In this sense, spA maps have the tightest dynamics among isotopic endperiodic maps.

\subsubsection*{Foliation cones and entropy}
Just as a closed manifold $M$ might fiber in various ways, so too might a sutured manifold $N$---in an appropriate sense. Representing these foliations in $H^1(N)$ leads to the foliation cones of Cantwell--Conlon.  In more detail, a class in $H^1(N)$ is \define{foliated} if it is dual to a fibration $N \ssm \partial N \to S^1$ whose foliation by fibers extends to $N$ by adjoining $\partial N$. Such a foliation is called a \emph{depth one foliation suited to $N$}; see \Cref{sec:endperiodic}.

\smallskip 

Cantwell and Conlon prove that the foliated classes in $H^1(N)$ comprise the integer points of
a union of finitely many open rational polyhedral cones, each of which is called a \define{foliation cone} of $N$ \cite{cantwell1999foliation, CCF19}. 
This can also be recovered as a consequence of our next result:

\begin{thm}[All foliations are spun]
\label{th:intro_spun}
There is closed hyperbolic $3$--manifold $M$ and an embedding $i \colon N \to M$ such that the morphism $i^* \colon H^1(M) \to H^1(N)$ maps each fibered cone of $M$ whose boundary contains $[\partial_+N]$ onto a foliation cone of $N$, and each foliation cone of $N$ is obtained in this manner.

Consequently, each depth one foliation suited to $N$ is obtained by spinning fibers of $M$ around $\partial_\pm N$.  

Moreover, for a fixed foliation cone $\mc C$ of $N$ there is a semiflow $\phi_\C$, obtained by restricting a pseudo-Anosov suspension flow on $M$, so that every depth one foliation $\mc F$ suited to $N$ with $[\mc F] \in \mc C$ is transverse to $\phi_\C$.
\end{thm}

See \Cref{th:foliation_cones} and \Cref{cor:dual_semiflow}.

Having organized the various ways that $N$ can fiber into a finite collection of polyhedral cones and having defined stretch factors associated to each primitive integral point in these cones, we turn to explain how the stretch factors vary as the foliation is deformed (see \Cref{th:entropy_cone}):

\begin{thm}[Entropy]
\label{th:intro_entropy}
For each foliation cone $\mc C$, there is a continuous, convex function $\ent \colon \mc C \to [0,\infty)$, 
such that for any foliated class $[L] \in \mc C$, 
\[
\ent([L]) = \log \left (\text{stretch factor of the spun pseudo-Anosov } f_L \right ). 
\]
\end{thm}

As previously remarked, in the classical setting of a closed fibered manifold, the associated entropy function has the additional features that it is real analytic and blows up at the boundary of the cone. Interestingly, neither of these stronger properties need to hold here since essential annuli in $N$ (which arise from invariant curves/lines of the monodromy $L \to L$)  can create obstructions.

By combining \Cref{th:intro_entropy} with previous work of the authors \cite{LMT21}, one also sees additional connections between the entropy functions on a fibered cone and the entropy function on the foliation cone obtained by cutting along a transverse surface. 
For details, see \Cref{sec:Lein}. 

\subsubsection*{Invariant laminations and Handel--Miller theory}

We conclude by explaining the connection to a previous approach to the study of endperiodic maps. In the early 1990s, Handel and Miller developed a theory to understand endperiodic maps using representatives that fix a canonical  pair of invariant geodesic laminations $\Lambda_\HM^\pm$ with respect to a fixed hyperbolic metric, analogous to the Casson--Bleiler \cite{casson-bleiler} approach to pseudo-Anosov theory in the finite-type setting. Although this was not written down by the authors, various expositions can be found in work of Fenley \cite{fenley1992asymptotic, Fen97} and Cantwell--Conlon \cite{cantwell1999foliation}, with the definitive treatment appearing in \cite{CCF19}. See \Cref{sec:HM} for more background.

The last theorem states that for any spA map $f \colon L \to L$, singular versions of the Handel--Miller laminations $\Lambda^\pm_\HM$ appear as sublaminations of its invariant foliations $\cW^{u/s}$ and in this sense every spA map is a `singular' Handel--Miller representative (see \Cref{th:lam_relation}).

\begin{thm}[spAs are HM]
\label{th:intro_HM}
The invariant singular foliations $\cW^{u/s}$ of an spA map $f \colon L \to L$ contain invariant singular sublaminations $\Lambda^\pm$ which determine the same endpoints in the hyperbolic boundary of $\wt L$ as the Handel--Miller laminations $\Lambda_\HM^\pm$. 
\end{thm}

The sublaminations $\Lambda^\pm$ of $\cW^{u/s}$ are topological invariants of the isotopy class of $f$ and are important for computing the topological entropy of $f$, as in \Cref{sec:lam_entropy}. For a more precise statement of \Cref{th:intro_HM} and further details, see \Cref{sec:HM}. In this sense, spA maps give an alternative, independent structure theory of atoroidal endperiodic maps that relies on the study of pseudo-Anosov flows rather than Handel--Miller theory.

\subsection{Acknowledgments}
We thank Chris Leininger for asking a question that led us to consider limits of pseudo-Anosov stretch factors and their topological significance. We also thank John Cantwell and Larry Conlon for helpful discussions on foliation cones and endperiodic maps, and Chi Cheuk Tsang and Marissa Loving for comments on an earlier draft.
M.L. thanks the Mathematisches Forschungsinstitut Oberwolfach for providing an ideal working environment during the final stages of this project.

\section{Background}
Here we briefly collect some background needed for working with endperiodic maps and pseudo-Anosov flows.

\subsection{Endperiodic maps, depth one foliations, and sutured manifolds}
\label{sec:endperiodic}

Let $L$ be an orientable, connected, infinite-type surface with finitely many ends and no boundary. If $g \colon L \to L$ is a homeomorphism, an end $e$ of $L$ is \define{attracting} (or positive) if it has a neighborhood $U\subset L$ such that for some $n\ge1$, $g^n(U) \subset U$ and $\bigcap_{i\ge1}g^{in}(U) = \emptyset$. The end $e$ is \define{repelling} (or negative) if it is attracting for $g^{-1}$. We say that $g$ is \define{endperiodic} if each of its ends is either attracting or repelling. Throughout this paper, we also require surfaces to have no planar ends.

As in \Cref{sec:intro}, if $g$ has no invariant essential finite multicurve up to isotopy,
then it is said to be \define{atoroidal}. 
The following lemma implies a simple property of atoroidal maps that will be useful later.

\begin{lemma}
\label{lem:uncontainable}
Let $g\colon L\to L$ be an endperiodic map, and let $\alpha$ be an essential closed curve
in $L$. Suppose that each element of $\{\alpha, g(\alpha), g^2(\alpha), \dots\}$ is
homotopic into a fixed compact subset of $L$. Then $g$ is not atoroidal.
\end{lemma}

\begin{proof}
Fix a hyperbolic metric on $L$ and let $\alpha_i$ be the geodesic tightening of $g^i(\alpha)$. 
Say that a compact subsurface $S$ of $L$ with geodesic boundary is a ``$g$-sink for
$\alpha$" if all $\alpha_i$ beyond some point lie in $S$. The hypothesis guarantees such a
surface exists, and it is evident that if $S$ is a $g$-sink then so is $g(S)$ after tightening its boundary to geodesics.

The intersection of two $g$-sinks, after tightening the boundary, is a $g$-sink.
It follows that a minimal $g$-sink exists, which we call $K$. Since $K$ is minimal,
$K\subset g(K)$ up to isotopy. But since they are homeomorphic they are isotopic. The
boundary of $K$ then provides the invariant multicurve that shows $g$ is not atoroidal. 
\end{proof}

\subsubsection*{Compactified mapping tori of endperiodic maps}

For our purposes in this article, a \textbf{sutured manifold} 
is a compact, oriented 3-manifold $N$ whose boundary components each have a coorientation 
(this would be known elsewhere as a sutured manifold $(M,\gamma)$ for which $\gamma =\emptyset$). 
We write $\partial N = \partial_+ N \sqcup \partial_-N$, where $\partial_+N$ denotes the boundary components that are cooriented out of $N$ and $\partial_-N$ denotes the boundary components that are cooriented in to $N$.
The terminology of sutured manifolds will be convenient for us, but we will not use much from the theory.

Let $N$ be a sutured manifold, let $\mc F$ be a cooriented codimension one foliation of $N$, and let $\FF_0$ 
denote its set of compact leaves. The foliation $\mc F$ has \textbf{depth one} if $\mc
F|_{N \ssm \FF_0}$ defines a fibration of $N \ssm \FF_0$ over $S^1$. When $\partial N \neq
\emptyset$, all foliations of $N$ in this article will be \define{suited to} $N$ in the
sense that $\partial N = \FF_0$, as cooriented surfaces. In general, every depth one
foliation we will consider is \define{taut} in the sense that each leaf is met by a
compatibly oriented transverse curve or properly embedded arc.

The following theorem appears as \cite[Proposition 6.21]{cantwell2017cones}, see also \cite[Theorem 1.1]{cantwell1993isotopy}. For the statement, note that a depth one foliation $\mc F$ suited to a sutured manifold $N$ induces a \emph{dual class} in $H^1(N)$, i.e. the class determined by the fibration $N \ssm \partial N \to S^1$.
\begin{theorem}[Cantwell--Conlon]
\label{th:class_determines}
Let $\mc F$ and $\mc F'$ be taut, depth one foliations suited to the sutured manifold $N$ inducing the same class on $H^1(N)$. Then $\mc F$ and $\mc F'$ are isotopic via a continuous isotopy that is smooth on $N \ssm \partial_\pm N$ and constant on $\partial_\pm N$. 
\end{theorem}

There is an important correspondence between depth one foliations suited to sutured manifolds and \emph{compactified} mapping tori of endperiodic maps that we now describe.
For more see \cite[Section 3]{Fen97}, \cite[Lemma 12.5]{CCF19}, or for an even more detailed treatment \cite[Section 3]{field2021end}.

Let $g\colon L\to L$ be endperiodic. The mapping torus of $g$ is $L\times[0,1]/(x,1)\sim(f(x),0)$. This mapping torus comes equipped with an oriented 1-dimensional foliation, called the \emph{suspension flow}, induced by the $\{p\}\times[0,1]$ foliation on $L\times[0,1]$; we refer to the leaves of this foliation as orbits of the suspension flow. 

The mapping torus is noncompact, but it has a natural compactification obtained by appending an ideal point to each end of an orbit of the suspension flow that escapes compact sets. It follows from endperiodicity that $g$ acts properly discontinuously on the sets of points in $L$ that escape compact sets under iteration of $g$, and similarly for iterations of $g^{-1}$. From this one can see that the union of the ideal points is naturally a disconnected surface with one component for each $g$-orbit of an end of $L$.
After gluing on this surface we obtain a compact manifold with boundary called the \textbf{compactified mapping torus} of $g$, which we denote by $N_g$.  

The components of $\del N_g$ come with natural coorientations: components corresponding to negative ends of orbits are cooriented inward, and those corresponding to positive ends of orbits are cooriented outward. The unions of outward and inward boundary components are denoted $\del_+N_g$ and $\del_-N_g$. 

An important point for us will be that $g$ is atoroidal if and only if $N_g$ is atoroidal 
\cite[Lemma 3.4]{field2021end}. Since no end of $L$ is planar, each component of $\partial_\pm N_g$ has genus at least $2$.

\smallskip
There is a natural cooriented depth one foliation $\mc F_L$ on $N_g$ whose leaves are the noncompact (i.e. depth one) surfaces $L\times\{t\}$ as well as the components of $\del_\pm N_g$. Informally, the positive ends of $L$ spiral around $\partial_+ N_g$ and the negative ends spiral around $\partial_-N_g$.
This gives $N_g$ the structure of a sutured manifold. 
This foliation is taut \cite[Lemma 3.3]{field2021end}, so in particular $\del_+N_g$ and $\del_- N_g$ are taut in the sense that they are 
Thurston norm-minimizing in $H_2(N_g)$ \cite{thurston1986norm}.

\subsection{Pseudo-Anosov flows and dynamic blowups}
\label{sec:blowup}
In this article, we consider only pseudo-Anosov flows $\phi$ on a $3$-manifold $M$ that are \define{circular}, i.e. equal to the suspension flow of a pseudo-Anosov homeomorphism up to reparametrization. 
One advantage here is that all of the structure we need concerning $\phi$ follows from the well-known structure of pseudo-Anosov homeomorphisms. For example, the unstable/stable invariant foliations $W^{u/s}$ of $\phi$ are simply the suspensions of the unstable/stable foliations of the pseudo-Anosov monodromy.

\smallskip

When dealing with a pseudo-Anosov flow $\phi$ it can be useful to slightly weaken the notion of what it means for an oriented surface $\Sigma$ to be positively transverse to $\phi$, obtaining the concept called ``almost transversality." 
For this we briefly discuss dynamically blowing up singular orbits. Dynamic blowups and almost transversality were introduced by Mosher in \cite{Mosher90correction}. 

A flow $\phi^\sharp$ is a \textbf{dynamic blowup} of $\phi$ if it is obtained by
perturbing $\phi$ in a neighborhood of some of the singular orbits of $\phi$, in the following way. We replace a singular orbit $\gamma$ by the suspension $A$ of a homeomorphism $g$ of a finite tree $T$, i.e. $A=T\times I/g$. We can identify $T$ with the intersection of $A$ with a cross-sectional disk $D$. For an edge $e$ with period $n$, $g^n$ should fix the endpoints of $e$ and act without fixed points on $\intr(e)$. The edges of $T$, together with the intersection with $D$ of the singular leaves containing $\gamma$, give a larger tree $\ol T$ in $D$. We orient each edge $e$ in $\ol T$ according to the direction in which points are moved upon first return to $e$. We require that around each vertex of $\ol T$, these orientations alternate between outward and inward. See \Cref{fig:dynamicblowup}.

The suspension $A$ is a $\phi^\sharp$-invariant annulus complex, and $\phi^\sharp$ is semiconjugate to $\phi$ by a map that collapses $A$ to $\gamma$ and is otherwise one-to-one. 
The vector fields generating $\phi$ and $\phi^\sharp$ differ only inside a small neighborhood of the orbits that were blown up. Each annulus in the complex $A$ is called a \textbf{blown annulus}. The two boundary components of a blown annulus are periodic orbits of $\phi^\sharp$. Inside the annulus, orbits of $\phi^\sharp$ spiral from one boundary component to the other.

The \textbf{stable foliation} of $\phi^\sharp$ is a singular foliation whose leaves are exactly the preimages of leaves of the stable foliation of $\phi$ under the semiconjugacy collapsing the various annulus complexes. Similarly for the \textbf{unstable foliation} of $\phi^\sharp$. Thus if $H$ is a stable leaf of $\phi$ containing a singular orbit $\gamma$ of $\phi$ which is blown up to obtain $\phi^\sharp$, then the stable leaf of $\phi^\sharp$ corresponding to $H$ will contain every blown annulus collapsing to $\gamma$. As such the stable and unstable leaves of $\phi^\sharp$ affected by the dynamic blowup will be tangent along shared blown annuli (see \Cref{fig:dynamicblowup}).

\begin{figure}[h]
\centering
\includegraphics{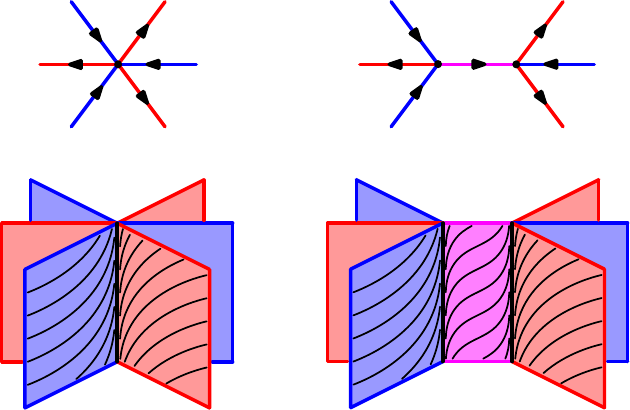}
\caption{A simple dynamic blowup of a 3-pronged singular orbit. On the top we see the intersection with a transverse disk before and after the blowup. On the bottom we see the intersection with a 3-ball before and after the blowup. In this case the associated annulus complex is a single annulus, shown in pink, which belongs to both the stable and the unstable leaf shown. Dynamic blowups can be more combinatorially complicated than this; see for example \Cref{fig:almosttransverse}.}
\label{fig:dynamicblowup}
\end{figure}

\smallskip

We say that $\Sigma$ is \textbf{almost transverse} to $\phi$ if there exists a dynamic blowup $\phi^\sharp$ of $\phi$ such that $\Sigma$ is positively transverse to $\phi^\sharp$. 
We also say that $\phi^\sharp$ is an \textbf{almost pseudo-Anosov flow}. Hence if $\Sigma$ is transverse to $\phi^\sharp$, then $\Sigma$ is almost transverse to $\phi$.

Suppose that $\Sigma$ is a compact surface positively transverse to an almost
pseudo-Anosov flow $\phi^\sharp$, obtained by dynamically blowing up a pseudo-Anosov flow
$\phi$. Let $a$ be a blown annulus of $\phi^\sharp$. Since $\Sigma \cap a$ is compact it
consists of either a collection of arcs between components of $\del a$ or a collection of
circles in $\intr(a)$ (see \Cref{fig:almosttransverse}). If $\Sigma \cap a$ consists of
arcs or is empty, then $a$ can be collapsed to obtain a ``less blown up" flow to which
$\Sigma $ is also transverse. Hence we say that $\phi^\sharp$ is \textbf{minimally blown
  up} with respect to $\Sigma$ if for every blown annulus $a$, the intersection of
$\Sigma$ with $a$ is a non-empty union of circles.

Similarly, if $\mc F$ is a depth one foliation transverse to $\phi^\sharp$, 
we say $\phi^\sharp$ is \textbf{minimally blown up} with respect to $\mc F$ if there is no
blown annulus $a$ of $\phi^\sharp$ such that $\mc F|_a$ is a foliation by properly
embedded arcs. We note that $\phi^\sharp$ is minimally blown up for $\FF$ if and only if it
is minimally blown up for $\FF_0$. Indeed, if $a$ is a blown annulus and $\FF\cap a$ is not a
foliation by arcs, then it 
 contains a closed leaf, which must be a component of $\FF_0 \cap a$, since $\FF$ is depth one.

\begin{figure}[h]
    \centering
    \includegraphics{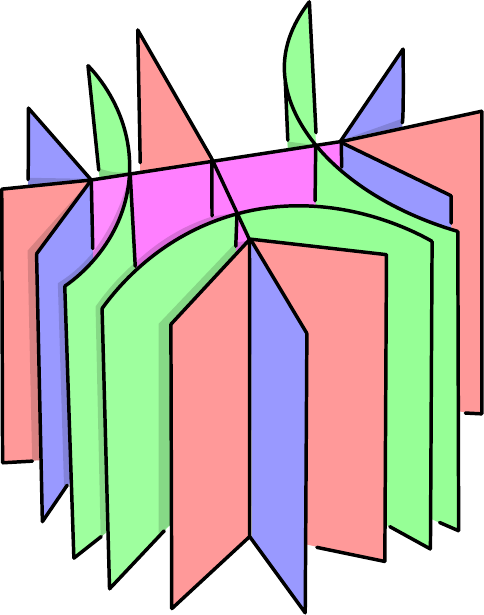}
    \caption{A local picture of a dynamically blown up 5-pronged singular orbit, together with part of a surface (green) transverse to the blown up flow. In this case the intersection of the surface with the blown annuli will be 3 circles.}
    \label{fig:almosttransverse}
\end{figure}

A key step in our construction of spA maps uses the following theorem, which is a special case of the main result in \cite{LMT_stst}.

\begin{theorem}[Strong transverse surface theorem]
\label{th:stst}
Let $M$ be a closed oriented $3$--manifold with a circular pseudo-Anosov flow $\varphi$.
For an oriented surface $\Sigma \subset M$, $\Sigma$ is 
almost transverse to $\varphi$, up to isotopy, if and only if $\Sigma$ is taut
and has nonnegative intersection number with every closed orbit of $\varphi$.
\end{theorem}

The ``if" statement of \Cref{th:stst} is the more difficult one and is proved as follows. There is a \emph{veering triangulation} associated to $\phi$, whose 2-skeleton is a cooriented branched surface in $M-U$, where $U$ is a small regular neighborhood of the singular orbits of $\phi$.
This is a \emph{partial branched surface in $M$}, in the sense of \cite{Landry_norm} (see \Cref{sec:veering}). If $\Sigma$ is a taut surface pairing nonnegatively with the closed orbits of $\phi$, then one can apply the techniques of that paper to isotope $\Sigma$ to be carried by the partial branched surface; this means it is carried  in $M-U$ (in the normal sense of branched surfaces), and intersects $U$ in a controlled way. The partial branched surface is positively transverse to $\phi$ (\cite[Theorem 5.1]{LMT21}) and its intersection with $\del U$ is easily understandable. This then allows us, given the data of $\Sigma\cap U$, to find a dynamic blowup of $\phi$ transverse to $\Sigma$.

\subsection{The spinning construction}
\label{sec:spinning}

We describe a special case of a standard construction in foliations which appears as \cite[Example 4.8]{CalegariFoliationBook}. We will refer to this operation as \textbf{spinning} (note that Calegari refers to it as ``spiraling").

Let $M$ be a closed, oriented 3-manifold equipped with a flow $\phi$. Suppose that
$S\subset M$ is a cross section to $\phi$, meaning that it is a closed oriented surface positively transverse to $\phi$ and intersecting each orbit. A consequence of this is that $M\cut S$ is homeomorphic to $S\times I$, where we can take the $I$-fibers to be the foliation on $M\cut S$ induced by the orbits of $\phi$. Further suppose that $\Sigma\subset M$ is another closed oriented surface which is transverse to both $S$ and $\phi$ such that no component of $\Sigma$ is isotopic to a component of $S$.

Let $\Sigma\times [-1,1]$ be a small tubular neighborhood of $\Sigma$ in $M$, where each $\{p\}\times [-1,1]$ fiber is an orbit segment of $\phi$. Let
\[
\ol \Sigma=\bigcup_{|n|>0}\Sigma\times \left\{\frac{1}{n}\right\}.
\]
Let $L$ be the oriented cut and paste sum of $S\ssm\Sigma$ with $\ol \Sigma$, smoothed so as to be transverse to $\phi$. 

Then $L$ is a noncompact surface. Moreover, $(M\ssm \Sigma)\cut L$ is homeomorphic to $L\times[0,1]$ where each $\{p\}\times[0,1]$ is an orbit segment of $\phi$. Hence we can fill in $M-(\Sigma \cup L)$ with a product foliation $L\times (0,1)$ to obtain a foliation $\mc F$ of $M$. The fact that $S$ is a cross section with no component isotopic to a component of $\Sigma$ implies that $S\cap \Sigma$ is homologically nontrivial in each component of $\Sigma$. This in turn implies that $L$ has ends that ``spiral" around the components of $\Sigma$,
the components of $\Sigma$ are the only compact leaves of $\mc F$, and the noncompact leaves of $\mc F$ define a fibration of $M\ssm \Sigma$ over $S^1$. Hence $\mc F$ is a taut depth one foliation of $M$.

Note that $N = M \cut \Sigma$ is naturally a sutured manifold, with the components of $\partial_\pm N$ cooriented by $\phi$, and $\mc F$ induces a taut, depth one foliation suited to $N$. We often continue to refer to this foliation of $N$ by $\mc F$ and also say that it is the result of spinning $N \cap S$ about $\partial_\pm N$.

\section{Existence of spA representatives}
\label{sec:spA_exist}

We begin by formally defining the main object of the paper.
\begin{definition}[spun pseudo-Anosov]
\label{def:spA}
An endperiodic map $f\colon L\to L$ is \textbf{spun pseudo-Anosov} (or \define{spA}) if
there exists a depth one foliation $\FF$ of a hyperbolic $3$-manifold $M$
and a transverse, circular almost pseudo-Anosov flow $\phi$ that is 
minimally blown up with respect to $\mc F$
such that $L$ is a 
leaf of $\mc F$ and $f$ is a power of the first return map induced by $\phi$.
\end{definition}

\Cref{def:spA} includes the maps described in \Cref{sec:motivation} that were
obtained by spinning a pseudo-Anosov homeomorphism of a
finite-type surface, and the spA maps we produce in this section will come from a generalization of this construction. See, e.g.,
\Cref{rmk:spin_in_proof}.

\smallskip

The main theorem of this section  gives the ``if'' direction of \Cref{th:intro_spA}:

\begin{theorem}[spA representatives exist] \label{th:exists}
Each atoroidal, endperiodic map is isotopic to a spun pseudo-Anosov map. 
\end{theorem}

\smallskip

\Cref{th:exists} will follow immediately from a more detailed result (\Cref{th:spA_construction}), which we now turn to state.
Let $g \colon L \to L$ be an endperiodic map and let $N_g$ be its compactified mapping torus. Recall that by construction $N_g$ comes equipped with a depth one foliation $\mc F_L$ whose compact leaves are $\partial_\pm N_g$ and whose depth one leaves are parallel to $L$. We call $\mc F_L$ the \define{depth one foliation associated to $L$}.

\begin{theorem} 
\label{th:spA_construction}
Let $g \colon L \to L$ be an atoroidal, endperiodic map with compactified mapping torus $N = N_g$. Then there exists
\begin{itemize}
\item a hyperbolic fibered $3$-manifold $M$,
\item a pseudo-Anosov suspension flow $\phi^\flat$ on $M$, and 
\item an embedding $N \to M$ and a dynamic blowup $\phi$ of $\phi^\flat$ 
so that the depth one foliation $\mc F_L$ extends to a depth one foliation of $M$ with respect
to which $\phi$ is transverse and minimally blown up.
\end{itemize}
Hence, the first return map $f \colon L \to L$ determined by $\phi$ is a spun pseudo-Anosov map isotopic to $g$.
\end{theorem}

The spA map $f \colon L \to L$ given by \Cref{th:spA_construction} is called an \define{spA representative} of $g$. 
We note that the constructed manifold $M$ and flow $\phi^\flat$ is quite flexible and this is one strength of the theory. For example, 
let us also give a natural strengthening of \Cref{def:spA}:

\begin{definition}[spA$^{+}$]\label{def:spA+}
An endperiodic map $f\colon L\to L$ is \define{spA$^+$} if there exists a circular pseudo-Anosov flow $\phi$ on a hyperbolic $3$-manifold $M$ and a depth one foliation $\FF$ transverse to $\phi$ such that $L$ is a depth one
leaf of $\mc F$ and $f$ is a power of the first return map induced by $\phi$.
\end{definition}

That is, a map $f\colon L \to L$ is spA$^+$ 
if no dynamic blowup is needed to make $\mc F$ positively transverse to $\phi$. We will prove in \Cref{th:no_blow} that every atoroidal endperiodic map is isotopic to an spA$^+$ map by showing that $M$ in \Cref{th:spA_construction} (or in \Cref{def:spA}) can be taken so that $\mc F$ is transverse to a circular (honest) pseudo-Anosov flow.

\begin{remark}[The spA package] \label{rmk:package}
Each spA map $f \colon L \to L$ comes equipped with the manifold $M$, foliation $\mc F$, and flow $\phi$ as in \Cref{def:spA}. Moreover, the compactified mapping torus $N = N_f$ can also be recovered (as in \Cref{th:spA_construction}) as the component of $M \cut \mc F_0$ containing the leaf $L$, where $\mc F_0$ denotes the compact (depth zero) leaves of $\mc F$. In general, the resulting map $N \to M$ could identify components of $\partial_\pm N$, but one can modify the foliation $\mc F$ by replacing its compact leaves with standardly foliated $I$-bundles. After this modification, $N \to M$ is an embedding. We always assume that this is the case and call the data: $M,\mc F, \phi, N$ an \define{spA package} for $f \colon L \to L$.
\end{remark}

\subsection{Juncture classes and spiraling neighborhoods}\label{sec:junctureclasses}
Fix an endperiodic map $g \colon L \to L$.
Let $N = N_g$ be its compactified mapping torus with the associated taut depth one foliation $\FF = \FF_L$. 
Let $\mc L = \mc L_g$ be the transverse $1$-dimensional oriented foliation on $N$ induced by the suspension flow of $g$.
We will sometimes refer to $\mc L$ as a semiflow on $N$. Note that $\mc F$ is cooriented by $\mc L$.

The foliation $\FF$ uniquely determines 
 its dual class $\xi \in H^1(N)$ 
that assigns to each oriented loop its signed intersection number with $L$. 
The pullback of $\xi$ to $H^1(\partial N)$ is called the \define{juncture class} $j\in H^1(\partial N)$ of $\partial N$ associated to $\FF$, and the restriction of $j$ to each 
component $\Sigma$ of $\partial N$ is called the juncture class of $\Sigma$.
A \define{realization} of the juncture class of $\Sigma$ is a cooriented collection of essential
simple curves in $\Sigma$ that represent the juncture class in $H^1$ having the property
that no subset 
is the boundary of a complementary subsurface of $\Sigma$ with either the
outward or inward pointing coorientation.

\begin{remark}
Since $\Sigma$ is closed,  a juncture class (or any cohomology class) can be realized by a single cooriented curve.
\end{remark}

A \define{spiraling neighborhood} $U$ of a component $\Sigma$ of $\partial_\pm N$ is a
collar neighborhood of $\Sigma$, foliated by arcs of $\mc L$,  whose outer boundary component is $\Sigma$ and whose inner boundary component is a surface $\Sigma_U$ that is transverse to both $\mc L$ and $L$. See, for example, \cite[Section 3]{cantwell1993isotopy} or \cite[Section 4]{fenley1992asymptotic}.
Note that $\Sigma_U \cap L$ represents the juncture class on $\Sigma$ after identifying $\Sigma_U$ with $\Sigma$ along the $\mc L$ fibers and coorienting using the coorientation on $L$. Moreover, after an isotopy of $\Sigma_U$ to remove product regions between $\Sigma_U$ and $L$, we may assume that $\Sigma_U \cap L$ realizes the juncture class on $\Sigma$. 
In general, a spiraling neighborhood is a disjoint collection of spiraling neighborhoods of each component of $\partial_\pm N$. 

Given a spiraling neighborhood $U$ of $\partial_\pm N$ we can collapse it along fibers of $\mc L$ so that $L \ssm \mathrm{int}(U)$ is sent to a properly embedded, cooriented surface $L_U \subset N$ whose boundary, with its induced coorientation, represents the juncture class in each component of $\partial_\pm N$.  More precisely, $L_U$ is obtained from $L \ssm \mathrm{int}(U)$ by flowing $L \cap \Sigma_U$ into $\partial_\pm N$ within $U$ and slightly isotoping the resulting surface rel boundary towards $\partial_\pm N$ so that its interior is transverse to $\mc L$.
See \Cref{fig:L-LU}.

\begin{figure}[htbp]
\begin{center}
\includegraphics[]{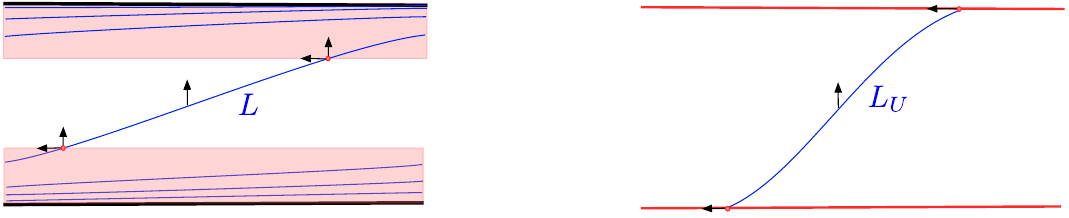}
\caption{Producing $L_U$ by pushing the spiraling of $L$ into $\partial_\pm N$.}
\label{fig:L-LU}
\end{center}
\end{figure}

By construction, $L_U$ is dual to $\xi$ and its interior is positively transverse to $\mc L$. One should think of $L_U$ as obtained from $L$ by `pushing the spiraling into $\partial_\pm N$.' The spiraling neighborhood $U$ determines how much of the spiraling is pushed into $\partial_\pm N$ and the following well-known lemma says that by choosing $U$ appropriately, any realization of the juncture class $j$ occurs as the boundary of $L_U$. See for example \cite[Lemma 0.6]{gabai1987foliations} where a more general statement is proven.

\begin{lemma} \label{lem:spiral}
For any realization $m$ of the juncture class $j\in H^1(\partial_\pm N)$, there is a spiraling neighborhood $U$ so that
$\partial L_U = m \subset \partial_\pm N$.
\end{lemma}

Finally, we say that the properly embedded surface $L_U$ is a \define{prefiber}
 if its interior meets each orbit of the semiflow $\mc L$. It is easy to see that for any $L_U$ there is a spiraling neighborhood $U' \subset U$ so that $\partial  L_{U'} = \partial  L_U$ and $L_{U'}$ is a prefiber. Informally, $L_{U'}$ can be obtained by ``peeling off'' another layer of $L$ from each component of $\partial_\pm N$.

\subsection{Extensions to the $h$--double}
\label{sec:ext}
Let $h\colon \partial_\pm N \to \partial_\pm N$ be a component-wise homeomorphism. The \define{$h$--double} of $N$ is the manifold $M = M(N, h)$ obtained by taking two copies of $N$ and gluing their boundaries together by $h$. When considering the 
triple $(N, \FF, \mc L)$ as above, 
$M$ comes
equipped with an induced taut depth one foliation and transverse, oriented $1$-dimensional foliation, still denoted by $\FF$ and $\mc L$ respectively, that extend the associated foliations on the two copies of $N$. 

Let us be more precise about the construction of the $h$--double.  Let $N^\d$ denote the sutured manifold $N$ with the coorientation on $\partial_\pm N$ reversed.  Hence, there are identifications $\sigma_\pm \colon \partial N_\pm \to \partial N^\d_\mp$ induced by the identity $N \to N^\d$. The sutured manifold $N^\d$ comes equipped with a depth one foliation $\FF^\d$ and oriented one dimensional foliation $\mc L^\d$ whose coorientation and orientation, respectively, have been reversed. We then glue $N$ to $N^\d$ via the map $\sigma_\pm \circ h \colon \partial_\pm N \to \partial N^\d_\mp$ and call the resulting manifold $M = M(N, h)$. By construction, the foliations $\FF$ and $\FF^\d$ combine to give a taut, cooriented depth one foliation on $M$, which we continue to denote by $\FF$, that is positively transverse to the oriented one dimensional foliation $\mc L \cup \mc L^\d$, which we continue to denote by $\mc L$. We can smooth $\mc L$ in a neighborhood of $\partial_\pm N$ maintaining transversality to $\FF$ and parameterize to obtain a flow on $M$ that is positively transverse to $\FF$.

\begin{remark}
The $h$--double $M = M(N,h)$ depends only on the sutured manifold $N$ and $h$. Hence, any
depth one foliation suited to $N$ (and transverse semiflow) automatically extend to $M$ by the construction.
\end{remark}

We identify $N$ with its image in $M$ coming from the construction. 
\begin{proposition} \label{prop:double_to_fiber}
If $h\colon \partial_\pm N \to \partial_\pm N$ reverses the sign of each juncture class 
(i.e. $h^*j = - j$ in $H^1(\partial_\pm N)$), then for any spiraling
neighborhood $U$ of $\partial_\pm N$ for which $L_U$ is a prefiber, there is a closed,
cooriented surface $S$ in $M = M(N,h)$ so that
\begin{enumerate}
\item $S$ has positive transverse intersection with every orbit of the extended flow $\mc L$ on $M$,
\item $S$ transversely intersects $\partial_\pm N \subset M$, and 
\item $S \cap N$ is properly isotopic to $L_U$. 
\end{enumerate}
\end{proposition}

\begin{remark}\label{rmk:hyp}
For \Cref{prop:double_to_fiber}, one can take $h$ to be any orientation preserving homeomorphism whose restriction to each component of $\partial_\pm N$ acts by $-I$ on $H_1$.
In particular a hyperelliptic involution has this property.
\end{remark}

\begin{proof}
Fix a spiraling neighborhood $U$ of $\partial_\pm N$ and consider the realization $m =
\partial L_U$ of the juncture class $j$ in $\partial_\pm N$.
Observe first that
$$
-j = \sigma_\pm^*(j^\d)
$$
which follows from the fact that 
$\FF^\d$ is $\FF$ with its coorientation reversed, i.e. under the identification $N \cong N^\d$ the
class in $H^1$ determined by $\FF^\d$ is exactly $-j$.
Now the hypothesis on $h$ implies that 
$$ h^* \circ \sigma_\pm^*(j^\d) = h^*(-j) = j.$$
This implies that the cooriented image
$m^\d = \sigma_\pm \circ h(m)$ is a realization of
the juncture class $j^\d$.

Now by \Cref{lem:spiral} there is a spiral neighborhood $V$ of $\partial_\pm N^\d$ so that the boundary of the properly embedded surface $L^\d_V$ in $\partial N^\d_\pm$ is equal to $m^\d$ and so that $L_V$ is a prefiber.
Hence in the $h$--double $M$, the boundaries of $L_U$ and $L^\d_V$ are identified with compatible coorientations. We define $S$ to be the unions $L_U \bigcup L^\d_V$. Since the interiors of $L_U$ and $L^\d_V$ are each positively transverse to the extended flow $\mc L$ and their coorientations agree across their boundary, we see that $S$ can be smoothed in a neighborhood of $\partial_\pm N$ to be positively transverse to $\mc L$ and $\partial_\pm N$. 
See \Cref{fig:glue_match}. Since $L_U$ is a prefiber, $S$ intersects each orbit of the extended flow $\mc L$ on $M$.
\end{proof}

\begin{figure}[htbp]
\begin{center}
\includegraphics[width = 1 \textwidth]{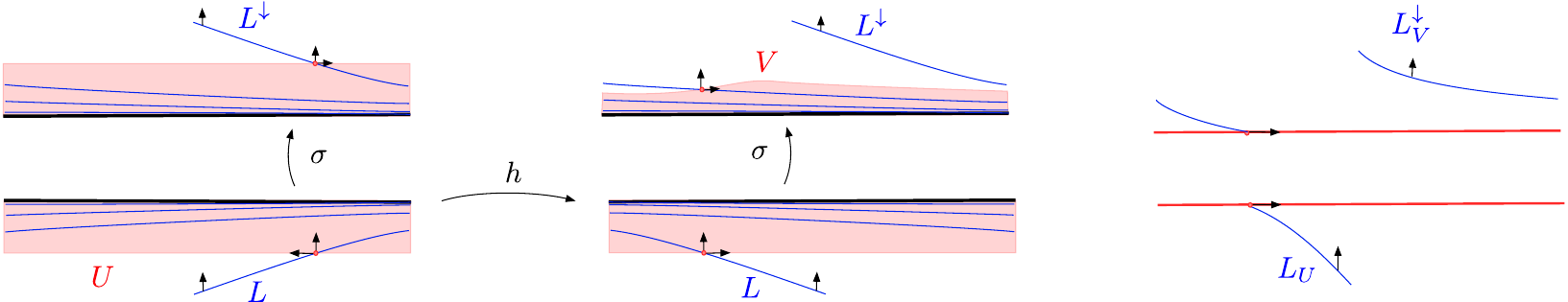}
\caption{Producing the transverse surface $S = L_U \bigcup L^\d_V$.}
\label{fig:glue_match}
\end{center}
\end{figure}

Since $S$ has positive transverse intersection with each orbit of $\mc L$, the manifold $M\cut S$ is a product $I$-bundle foliated by segments of $\mc L$. This is to say that 
there is a well-defined first return map $S \to S$ that determines a fibration of $M$ over $S^1$ with fiber $S$. 
\begin{remark}[Spinning $S$ back to $\mc F$]
\label{rmk:spinning_back}
By construction, spinning the surface $S$ obtained in \Cref{prop:double_to_fiber} about $\partial_\pm N$ reproduces the foliation $\mc F$.
\end{remark}

\subsection{Hyperbolic manifolds and circular pseudo-Anosov flows}
Recall that since $g \colon L  \to L$ is {atoroidal} its compactified mapping torus $N =N_g$ is atoroidal and each component of $\partial_\pm N$ is a closed surface of genus at least $2$.

\begin{lemma} \label{lem:hyp}
For $N$ as above, the gluing map $h$ (satisfying \Cref{prop:double_to_fiber}) can be
chosen so that $M(N,h)$ admits a hyperbolic structure.
\end{lemma}

\begin{proof}
We first consider the special case where $N = S \times [0,1]$. Here, if $h_1$ is the restriction of $h$ to $\partial_+ N = S \times 1$ and $h_0$ is the restriction of $h$ to $\partial_- N = S \times 0$, then $M(N,h)$ is the mapping torus of  $h_0h_1 \colon S \to S$, which is hyperbolic if and only if $h_0h_1$ is pseudo-Anosov \cite{thurston1998hyperbolic, otal2001hyperbolization}.

Otherwise, JSJ theory \cite{jaco1980lectures} gives proper essential subsurfaces $Y_+$ of $\partial_+N$ and $Y_-$ of $\partial_-N$ such that each essential annulus $A$ in $N$ can be isotoped so that $\partial A \cap \partial_+ N \subset Y_+$ and $\partial A \cap \partial_- N \subset Y_-$. Hence, if $h$ is chosen to that the image of any curve in $Y_\pm$ is not homotopic into $Y_\pm$, then $M(N,h)$ is atoroidal and hence hyperbolic by Thurston's hyperbolization theorem (see e.g. \cite{kapovich2001hyperbolic}).

From this, it is easy to produce a gluing map $h$ so that $M(N,h)$ is hyperbolic and for which \Cref{prop:double_to_fiber}. For example, $h$ can be taken so that on each component it is a hyperelliptic involution composed with a high power of a pseudo-Anosov homeomorphism that acts trivially on $H_1$.
\end{proof}

We now return to the context of \Cref{th:spA_construction}, where $g \colon L \to L$ is endperiodic and atoroidal and $N = N_g$ is its compactified mapping torus. Let $h \colon \partial_\pm N \to \partial_\pm N$ be any component-wise homeomorphism that satisfies the conditions of \Cref{prop:double_to_fiber} and \Cref{lem:hyp} and let $M = M(N,h)$ be the associated $h$--double. As in \Cref{sec:ext}, there is a fixed embedding $N \to M$ such that $\mc F = \mc F_L$ and $\mc L$ extend to $M$. 

With this structure fixed, we turn to the proof of \Cref{th:spA_construction}.

\begin{proof}[Proof of \Cref{th:spA_construction}]
Let $S$ be the properly embedded surface obtained from \Cref{prop:double_to_fiber}. Since $S$ has positive transverse intersection with each orbit of $\mc L$ on $M$, $S$ is a fiber in a fibration of $M$ over the circle and the class $[S]$ is contained in the interior of the cone $\C_{\mc L} \subset H^1(M)$ of classes that are nonnegative on homology directions of $\mc L$ (see \cite{Fri79} for details). Since $\partial_\pm N$ is also positively transverse to $\mc L$ by construction, $[\partial_\pm N]$ is also contained in  $\C_{\mc L}$. Indeed, $[\partial_\pm N]$ is contained in the boundary of this cone unless $N$ is itself a product. Note that since $\partial_\pm N$ is the union of compact leaves of the taut foliation $\mc F$, $\partial_\pm N$ is taut.

Since $M$ is hyperbolic, the monodromy $S\to S$ is isotopic to a pseudo-Anosov homeomorphism. By suspending the pseudo-Anosov representative of this monodromy we obtain a pseudo-Anosov flow $\phi$ on $M$, unique up to isotopy and reparameterization, which is transverse to $S$.

According to Fried \cite[Theorem 14.11]{Fri79}, the cone $\C_{\phi} \subset H^1(M)$ of
classes that are nonnegative on the homology directions for $\phi$  is equal to the
closure of the fibered cone containing $[S]$,
and $\C_{\mc L} \subset \C_{\phi}$.  
Hence, we also have that $[\partial_\pm N]$ is in the cone $\C_\phi$ and so $\partial_\pm N$ has nonnegative intersection number with each closed orbit of $\phi$. Since $\partial_\pm N$ is also taut, \Cref{th:stst} implies that $\partial_\pm N$ is almost transverse to $\phi$, up to isotopy. 

Let $\phi^\sharp$ be the associated minimal dynamic blowup suitably isotoped so that it is positively transverse to $\partial_\pm N$. This isotopy carries $S$ to a cross section $S'$ of $\phi^\sharp$. As in \Cref{rmk:spinning_back}, the foliation $\mc F$ of $M$ is obtained by spinning $S$ about $\partial_\pm N$. If we denote by $\mc F'$ the foliation of $M$ obtained by spinning the $\phi^\sharp$--cross section $S'$ about the $\phi^\sharp$--transverse surface $\partial_\pm N$, we obtain a transverse, taut, depth one foliation $\mc F'$ of $M$. The following lemma states that $\mc F$ and $\mc F'$ are isotopic. Once established, we can isotope $M$ so that $\mc F$ is positively transverse to the flow $\phi^\sharp$, thereby completing the proof of the theorem.
\end{proof}

It only remains to prove the following lemma, which follows easily from \Cref{th:class_determines}.
\begin{lemma}
Let $S$ and $S'$ be isotopic fibers of $M$ and suppose that $\Sigma$ is a taut surface in the boundary of the associated fibered cone. Then the foliations $\mc F$ and $\mc F'$ obtained by spinning $S$ and $S'$ around $\Sigma$ are isotopic in $M$.
\end{lemma}

\begin{proof}
Let $N_1, \ldots, N_k$ be the components of $M\cut \Sigma$, each of which is a sutured manifold that comes equipped with two taut, depth one foliations $\mc F_i,\mc F'_i$. Since $S'$ and $S$ are isotopic, the classes in $H^1(N_i)$ associated to $\mc F_i$ and $\mc F_i'$ are equal for each $i$. Hence by \Cref{th:class_determines}, $\mc F_i$ is isotopic in $N_i$ to $\mc F'_i$ via an isotopy that is constant on $\partial_\pm N_i$. Therefore, these isotopies glue together in $M$ to give an ambient isotopy taking $\mc F$ to $\mc F'$, completing the proof.
\end{proof}

\begin{remark}[Spinning and de-spinning]
\label{rmk:spin_in_proof}
The proof of \Cref{th:spA_construction} shows that the foliation $\mc F$ on the $h$--double $M$ is obtained by spinning the cross section
$S$ along $\partial_\pm N$, up to isotopy. 
Similarly, the restricted foliations $\mc F$ on $N$ is obtained by spinning the properly embedded surface $S \cap N$ about $\partial_\pm N$.

Reversing this process, the argument in \Cref{prop:double_to_fiber} shows by choosing a spiraling neighborhood of each compact leaf of $\mc F$ in $M$ (or, in other words, choosing spiraling neighborhoods of $\partial_\pm N$ in $N$ and $N^\downarrow$) one can produce a (nonunique) surface $S$ that is a cross section of $\varphi$. We call this process \define{de-spinning}. 

We note that it is not the case that every depth one foliation of an arbitrary closed manifold can be de-spun to a fibration because of a basic cohomological obstruction: each compact leaf has two associated juncture classes from leaves spiraling on it from either side and the foliation can be de-spun if and only if these classes agree.
\end{remark}


\section{Multi sink-source dynamics}
\label{sec:sink-source}

The goal of this section is to relate the periodic point behavior of an spA
map $f:L\to L$ to
the dynamics of the action of a lift $\wt f$ on the universal cover $\wt L$ and its
compactification $\DD = L \union \partial\wt L$, defined with a suitable hyperbolic
metric. The spA map comes with a pair of invariant foliations inherited from those of the
pseudo-Anosov suspension flow, whose expanding and contracting half-leaves interact with
the action on the circle at infinity (see below for complete definitions). 

A homeomorphism $h \colon S^1 \to S^1$ is said to have \define{multi sink-source dynamics}
if it has a finite number $k\ge 4$ of fixed points that alternate between attracting and repelling.

\begin{theorem}\label{thm:fix and boundary}
Let $f\colon L\to L$ be spA and let $\wt f:\wt L \to \wt L$ be a lift of $f$. Then $\wt f$
has a fixed point $\wt p\in \wt L$ if and only if it has a power $\wt f^n$ which acts on
$\partial\wt L$ with multi sink-source dynamics. Moreover when this happens, $\wt f^n$
fixes the half-leaves at $\wt p$ and its attracting/repelling points are equal to the
endpoints of expanding/contracting half-leaves.
\end{theorem}

This generalizes well-known properties of pseudo-Anosov maps on compact surfaces, however there
are a number of complications to deal with in our setting. Because the surface is only
almost-transverse to the pseudo-Anosov flow, the structure of the stable and unstable
foliations is harder to work with; in particular they are not everywhere transverse. The
local dynamics on the circle at infinity require more work to understand, especially
showing that sinks/sources on the circle are actually sinks/sources on the closed disk
(see \Cref{cor:gen_local_dynamics}). 

One application of this theorem is a proof of \Cref{th:intro_spA+}, which we restate here:
\begin{theorem} \label{th:no_blow}
Each atoroidal, endperiodic map is isotopic to an spA$^+$ map.
\end{theorem}

\subsubsection*{Summary of the section:} In \Cref{sec:annular-leaves} we analyze the behavior of
{\em periodic half-leaves} of the stable and unstable foliations of $L$, by explaining the possibilities for their suspensions in the two-dimensional foliations of $N$.
In \Cref{sec:fols} we recall results of Fenley and Cantwell-Conlon describing hyperbolic
metrics on $L$ and how they give the universal cover $\wt L$ a canonical circle at
infinity, and on the quasi-geodesic properties of the foliation leaves in this
metric.

In \Cref{sec:bdynamics} we prove \Cref{cor:gen_local_dynamics} and \Cref{lem:local_dynamics} which
give the sink/source properties at infinity for endpoints of periodic half-leaves. 

In \Cref{sec:global_dynamics} we prove \Cref{prop:fix_to_boundary}, which gives
one direction of \Cref{thm:fix and boundary}. 

In \Cref{sec:no_blow} we will apply what we have so far to prove
\Cref{th:no_blow} on spA$^+$ representatives. Finally,
we will apply this in \Cref{sec:sink-source  converse} 
to prove the other direction  of \Cref{thm:fix and boundary}, which will be stated in
\Cref{prop:boundary_to_fix}.

\medskip

Throughout this section we fix a spun pseudo-Anosov map $f \colon L \to L$,
together with an spA package (as in \Cref{rmk:package}) denoted by $M,\mc F, \phi, N$, where $M, \mc F, \phi$ are as in 
\Cref{def:spA} and $N \subset M$ is the compactified mapping torus of $f$. 
Moreover, we denote the induced semiflow on $N$ as $\phi_N$.
For simplicity, we often denote the closed surface $\partial_\pm N$ in this section by $\partial_\pm$, and note that it comprises a subset of the
 compact leaves $\mc F_0$ of $\mc F$.

Since $\phi$ is minimally blown up with respect to $\FF$, for each blown annulus $A$, $\FF_0 \cap A$ is nonempty and consists of finitely many curves homotopic in $M$ to the core of $A$. In particular, no closed orbit in the boundary of a blown annulus intersects $\FF_0$.
See \Cref{fig:decomp}.

\begin{figure}[htbp]
    \centering
    \includegraphics{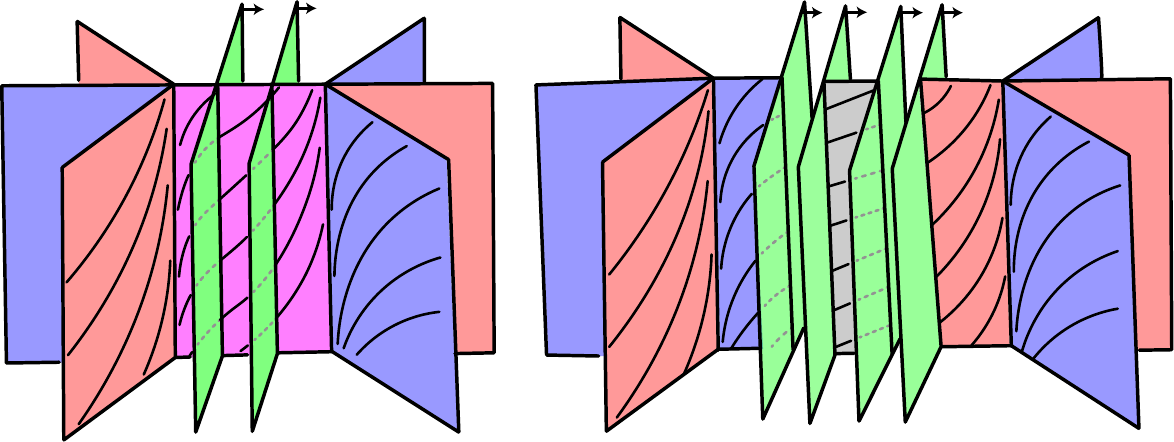}
    \caption{Left: two pieces of $\FF_0$ (green) intersecting a blown annulus (pink). Right: the result of cutting along $\FF_0$, with $\del_+(M\cut \FF_0)$ and $\del_-(M\cut \FF_0)$ indicated by the boundary coorientations. The blown annulus gives rise to two periodic half-leaves in $M\cut \FF_0$, both of which are technically stable \emph{and} unstable. However, one of the half-leaves is expanding (blue, touching $\del_+(M\cut \FF_0)$) and one is contracting (red, touching $\del_-(M\cut \FF_0)$). Cutting along $\FF_0$ also produces an annular leaf which touches $\del_+(M\cut \FF_0)$ and $\del_-(M\cut \FF_0)$ and is not periodic (gray).}
    \label{fig:decomp}
\end{figure}

\subsection{Periodic half-leaves and annuli}
\label{sec:annular-leaves}

The circular flow $\phi$ has invariant stable and unstable singular foliations which we denote $W^s$ and $W^u$, respectively. There is an induced semiflow $\phi_N$ on $N$; 
let $W^s_N$ and $W^u_N$ denote the foliations preserved by $\phi_N$ that arise by cutting
$W^s$ and $W^u$ along $\partial_\pm$. Finally we define $\mc W^s = L \cap W^s$ and $\mc
W^u = L \cap W^u$ to be the invariant \define{stable and unstable foliations} of $f \colon
L \to L$. By construction, $\mc W^{u}$ and $\mc W^s$ suspend in $f$'s compactified mapping
torus $N$ to be $W^u_N$ and $W^s_N$, respectively.

If $p$ is a periodic point of $f$ then the half-leaves of $\mc W^u$ and $\mc W^s$ emanating
from $p$ are fixed by a power of $f$, and we call them \define{periodic half-leaves} of
$f$. A periodic half-leaf contained in a blown up annulus is called a periodic \define{blown half-leaf}.
We remark that $\mc W^u$ and $\mc W^s$ are transverse on $L$ \emph{except} at blown
half-leaves, which are common to both foliations. 
Regardless, 
each periodic half-leaf $\ell$ of $f$ is either 
\define{contracting} or \define{expanding} depending on whether iterating positive or negative powers 
of $f$ attract points of $\ell$ to its periodic point $p$.

The main goal of this subsection is \Cref{lem:halfleafend}, which constrains the asymptotic
behavior of such half-leaves.

If $e$ is an end of $L$ we say that a subset $A\subset L$ \define{accumulates on $e$} if
every neighborhood of $e$ has non-empty intersection with $A$. 
    We say that $A$ \define{escapes $e$} if for every neighborhood $U$ of $e$ there is a compact
    $K\subset A$ such that $A\ssm K\subset U$.

\begin{lemma}\label{lem:halfleafend}
Every periodic half-leaf of $f$ accumulates on some end of $L$. Expanding half-leaves
accumulate on the positive ends, and contracting half-leaves
accumulate on the negative ends. Moreover, each periodic blown half-leaf escapes a unique end.
\end{lemma}

This lemma will follow easily once we develop the picture of the half-leaves
of the 2-dimensional foliations $W^{s/u}_N$ obtained as suspensions of the periodic
half-leaves of $\mc W^{s/u}$.

\medskip

We say that a leaf of $W^{s/u}$ or $W^{s/u}_N$ is \textbf{periodic} if it contains a periodic orbit of $\phi$ or $\phi_N$ respectively.  A \textbf{periodic half-leaf} of $W^{s/u}$ or $W^{s/u}_N$ is a component of $H_0\cut\{\gamma_i\}$, where $H_0$ is a periodic leaf of $W^{s/u}$ or $W^{s/u}_N$ respectively, and $\{\gamma_i\}$ is the collection of all periodic orbits contained in $H_0$. A periodic half-leaf of $W^{s/u}$ is compact if and only if it is a blown annulus.

Let $H_0$ be a periodic leaf of $W^{s/u}$. If $H_0$ contains no blown annuli then it
contains a unique periodic orbit $\gamma$. In this case  $H_0\cut\gamma$ 
is a union of noncompact periodic half-leaves, each a half-closed annulus adjacent to $\gamma$ so that the flow lines in it spiral toward or away from $\gamma$ if $H_0$ is in $W^s$ or $W^u$, respectively. If $H_0$ contains blown annuli, then they are attached along their boundaries in a tree pattern, and attached to this complex are
noncompact half-leaves (see \Cref{fig:almosttransverse}).

By definition, the periodic leaves of $W^{s/u}_N$ obtained from $H_0$ are the components of $H_0 \cut
\partial_\pm N$ which contain a periodic orbit.
Because $\phi$ is minimally blown up with respect to  $\mc F$ by assumption, 
each blown annulus that meets $N$
is cut by $\partial_\pm$ and hence each periodic leaf of
$W^{s/u}_N$ contains a unique periodic orbit. 

If $H$ is a periodic half-leaf of $W_N^{u/s}$ 
then we say that $H$ is \define{contracting} or \textbf{expanding} if every flow line in $H$ is asymptotic to $H$'s unique periodic orbit in the forward or backward direction, respectively. Periodic half-leaves that are contained in blown annuli (and hence are half-leaves of \emph{both} $W_N^u$ and $W_N^s$) can be expanding or contracting (see \Cref{fig:decomp}). 
However, if any other periodic half-leaf is contracting or expanding, then it is contained in leaf of $W^s$ or $W^u$, respectively. 

The next lemma describes the pieces obtained by cutting leaves of $W_N^{s/u}$ along periodic orbits, and shows in particular that each periodic half-leaf of $W_N^{s/u}$ is either expanding or contracting.

\begin{lemma}\label{lem:annulus leaf structure}
 Let $H_0$ be a periodic leaf of $W^u$ or $W^s$. Let $C$ be a component of $H_0$ cut along
  $\partial_\pm N$ and along any periodic orbits contained in $H_0$ such that $C\subset N$.
If $C$ is not a disk
 it is an annulus, and it is one of four types (illustrated in \Cref{fig:four-annuli}):
\begin{enumerate}
  \item (compact periodic) $C$ is an expanding or contracting periodic half-leaf of $W_N^{u/s}$ and has a unique closed orbit on its boundary.  Its other 
    boundary component is a single closed loop in $\partial_\pm$ to which the flow lines are
    transverse. 
    
    \item (noncompact periodic) $C$ is an expanding or contracting periodic half-leaf of $W_N^{u/s}$ and has a unique closed orbit on its boundary. The rest of $\partial C$ is a collection of properly embedded lines of $H_0\cap \partial_\pm N$, to which the flow lines are transverse. 

      \item (transient compact) $\partial C$ consists of two closed curves where the flow
        intersects both $\partial_+$ and $\partial_-$ transversely.

      \item (transient noncompact) $\partial C$
        has one closed component meeting one of $\partial_\pm$, and a collection of arcs meeting the other.

\end{enumerate}

If $C$ is contained in a blown annulus it has type (1) or (3). 

If $C$ is periodic and expanding (contracting), then all flow lines in $C$ are backward (forward) asymptotic to $\gamma$ and all other boundary components of $C$ lie in $\del_+N$ ($\del_-N$).
\end{lemma}

\begin{figure}[htbp]
    \centering{
    \includegraphics[width=4.8in]{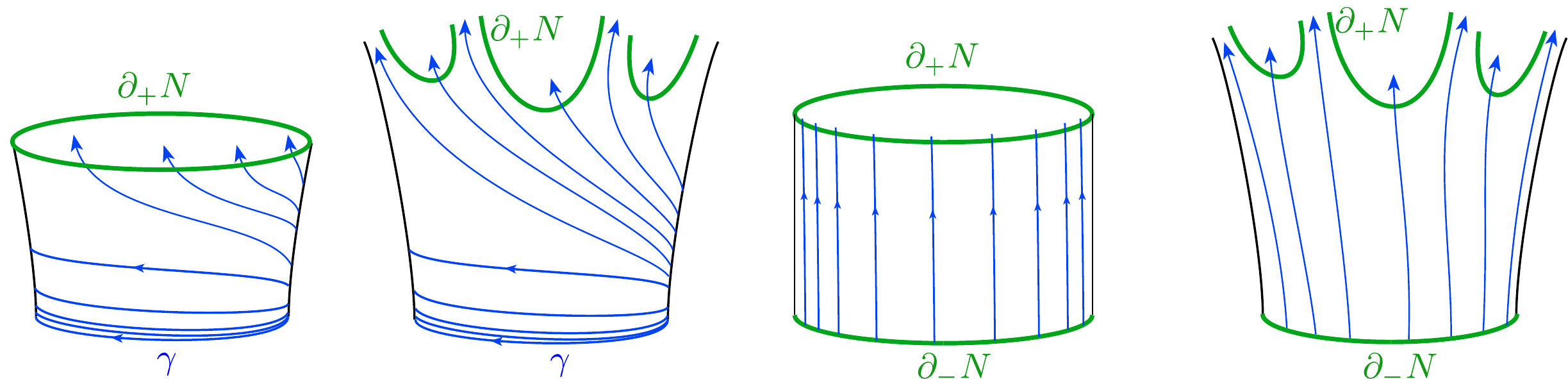}}
    \caption{The types of annular half-leaves of $W^{s/u}_N$. Cases $(1)$, $(2)$ and $(4)$ are shown in
      the expanding case.}
    \label{fig:four-annuli}
\end{figure}

\begin{proof}
We suppose $H_0$ lies in $W^u$; the stable case is analogous.

In each blown annulus of $H_0$, the flow $\phi$ is asymptotic to one boundary component in
each direction. Cutting along the core curves where the blown
annulus intersects $\partial_\pm$, we obtain annuli of types (1) and (3).

Now consider a noncompact half-leaf $A$ of $H_0$, with $\del A$ equal to a closed orbit $\gamma$. It is a standard fact that each end of a leaf of the stable or unstable foliation of a pseudo-Anosov diffeomorphism of a compact surface is dense in the surface, and this implies that $A$ is dense in $M$ (the fact that $\phi$ is obtained by dynamic blowup is not an issue here). Hence any subset of $A$ whose complement is compact has nonempty intersection with $\del_\pm$.

On $\intr(A)$, the flow is equivalent up to diffeomorphism to the vertical
flow on $S^1\times\R$ (where the backward time flow to $-\infty$ spirals toward $\gamma$). The intersection with $\partial_\pm$, by transversality, consists of closed curves or arcs which are graphs of functions from $S^1$ or a subinterval (respectively) to $\R$. Note this forces the arcs to be properly embedded and asymptotic to the $+\infty$ direction. 

From this description we see that every component of the complement of these curves of intersection is either a disk or an annulus, and that each annulus component $C$ either contains $\gamma$ in its boundary (giving cases (1) and (2)) or is bounded below by a single closed curve of $\partial_{\pm}$ (giving (3) and (4)). 

Since any such periodic half-leaf is contained in $N$, the flow lines enter through $\partial_-$ and exit through $\partial_+$; this implies the final statement of the lemma.
\end{proof}

Next, we want to describe how a periodic half-leaf of $W^{s/u}_N$ intersects the surface $L$, or
equivalently any of the noncompact leaves of the depth one foliation $\FF$ of $N$. 

\begin{lemma}\label{lem:annulus_fibering}
Let $H$ be a periodic half-leaf of $W_N^{u/s}$. The fibration $\mathrm{int}(N)\to S^1$ determined by $\mc F$ restricts to a fibration
of $\mathrm{int}(H)$, where each fiber is the intersection of $\mathrm{int}(H)$ with one of the leaves of $\FF$.
Moreover, $L\cap H$ is either:
\begin{enumerate}
  \item (compact periodic) A finite union of rays, each of which is transverse to the periodic boundary
    $\gamma$ and spirals onto the opposite boundary component.

  \item (noncompact periodic) A finite union of rays, each of which is transverse to the periodic boundary
    $\gamma$ and accumulates onto all of the other boundary components.
\end{enumerate}
In particular, the interior of $H$ is the 
suspension of a periodic half-leaf of the foliation $\mc W^{s/u}$ in $L$. When $f$ is
expanding on the half-leaf, its suspension is expanding, and when $f$ is contracting its suspension
is contracting.

Conversely, every suspension of a periodic half-leaf of $\mc W^{s/u}$ gives rise to such a 
periodic half-leaf of $W_N^{u/s}$. 
\end{lemma}

\begin{figure}[htbp]
    \centering{
    \includegraphics[width=3.5in]{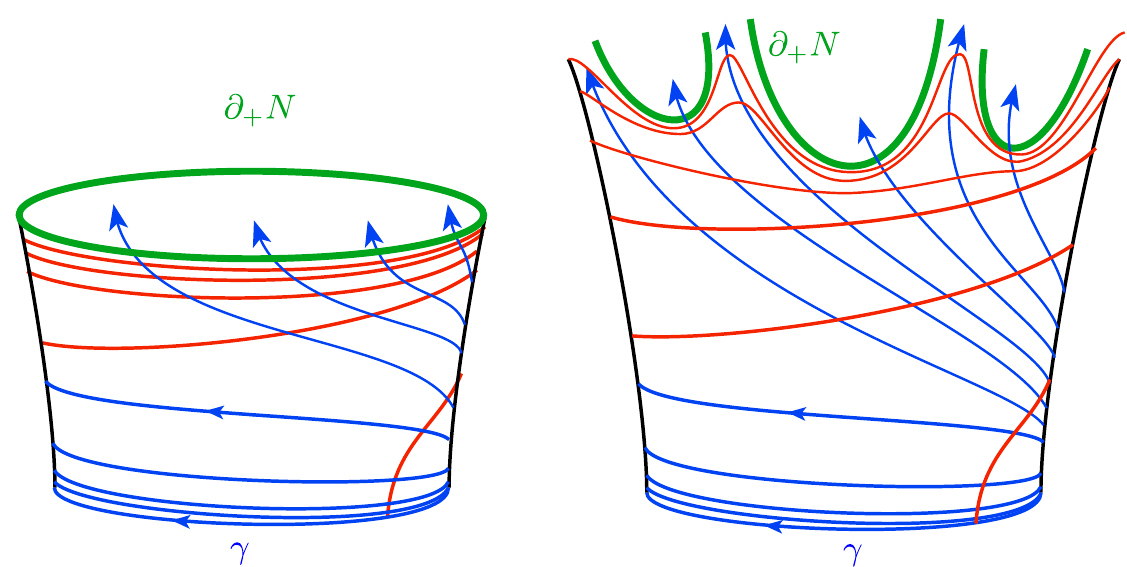}}
    \caption{Expanding periodic half-leaves of $W_N^{u/s}$ with their intersections with $L$ shown in red.}
    \label{fig:two-annuli-fibered}
\end{figure}

\begin{proof}
The fibration map $\mathrm{int}(N)\to S^1$, restricted to the half-leaf $H$, is a submersion since the flow
directions are transverse to the foliation $\FF$. Each leaf $F$ of $\FF$ in $\mathrm{int}(N)$ is the
preimage of a point in $S^1$, and so the same is true for $F\cap H$. Finally, $F\cap H$ is
a section of the flow in $H$, because every forward flow ray of $\phi_N$ eventually
returns to $F$. This implies that $H$ is fibered over $S^1$, with fibers $F\cap
H$ (which need not be connected).

Since $H$ is a periodic half-leaf, the fibration restricted to the (periodic) boundary component $\gamma$ is a
covering map to $S^1$, so that $L\intersect \gamma$ is a finite union of $k>0$ points. Thus
$L\intersect H$ has $k$ components, each of which must be a half-leaf emanating
from $\gamma$. Let $\eta$ be one such component. Then $H\cap \intr(N)$ is the suspension of $f^k$ restricted to $\eta$, so each orbit of $\phi_N|_{\intr(H)}$ returns infinitely often to $\eta$. This picture shows that $H$ is expanding/contracting if and only if $f^k$ is expanding/contracting on $\eta$.

Suppose $H$ is expanding. In this case $H\cap \del_-=\varnothing$ by \Cref{lem:annulus leaf structure}. Let $\ell$ be any component of $H\cap \del_+$, let $p\in \ell$, and let $\gamma_p$ be the orbit of $\phi_N$ terminating at $p$. Since $\gamma_p\cap \intr (N)$ intersects $\eta$ infinitely often in the forward direction, we conclude $\eta$ accumulates on $\ell$. The case of $H$ contracting is symmetric.

Conversely, if we suspend a half-leaf, we obtain an annulus which is properly embedded in
$N$ and hence must be a periodic annulus component as in \Cref{lem:annulus leaf
  structure}.
\end{proof}

\begin{proof}[Proof of \Cref{lem:halfleafend}]
By \Cref{lem:annulus_fibering}, any periodic half-leaf $\ell$ of $\cW^{u/s}$ suspends to a periodic half-leaf $H$ of $W_N^{u/s}$ as in \Cref{fig:two-annuli-fibered}. When $\ell$ is expanding, $H$ is expanding and so $\ell$ accumulates on $\partial_+$. In terms of the surface $L$, this implies that $\ell$ accumulates on a positive end of $L$. When $\ell$ is contracting, $H$ is contracting and $\ell$ accumulates on $\partial_-$.

If $H$ is contained in a blown annulus, it corresponds to case $(1)$ of
\Cref{fig:two-annuli-fibered} by \Cref{lem:annulus leaf structure}, and in that case the
$\ell$ spirals onto a single boundary component of $\partial_\pm$, which implies that it
escapes a unique end. 
\end{proof}

\subsection{Hyperbolic metrics, boundaries, and invariant foliations}
\label{sec:fols}
We now need to connect the foliations $\cW^{u/s}$ of $L$ with its hyperbolic geometry.

A \define{standard} hyperbolic metric on the surface $L$ is a complete hyperbolic metric that contains no embedded hyperbolic half spaces. Suppose that $L$ is given a standard hyperbolic metric so that its universal cover $\wt L$ is isometric to the hyperbolic plane; we denote its hyperbolic boundary by $\partial \wt L$ and the associated compactification by $\DD = \wt L \cup \partial \wt L$. 

The following facts, which we use without further comment, will be crucial: 
\begin{enumerate}
\item For any homeomorphism $g \colon L \to L$, any lift $\wt g \colon \wt L \to \wt L$ has a unique continuous extension to $\DD$ \cite[Theorem 2]{CaCo13} and we continue to denote by this homeomorphism and its restriction to $\partial \wt L$ by $\wt g$. 

\item  If $f, g \colon L \to L$ are homotopic maps, then they are isotopic \cite[Corollary 10]{CaCo13} and if $\wt f, \wt g \colon \wt L \to \wt L$ are obtained by lifting a homotopy to $\wt L$, then $\wt f$ and $\wt g$ agree on $\partial \wt L$ \cite[Corollary 5]{CaCo13}. 

\item If $\sigma_1,\sigma_2$ are two standard hyperbolic metrics on $L$, then any lift of the identity map on $L$ extends to a homeomorphism $\DD_1 \to \DD_2$, where $\DD_i$ is the hyperbolic compactification of $ \wt L$ with respect to $\sigma_i$ \cite[Lemma 10.1]{CCF19}.
\end{enumerate}

\smallskip
The following proposition implies the properties that we will need concerning the singular foliations $\cW^{u/s}$, most of which follow from work of Fenley \cite{Fen09}. We let $\wt{\cW}^{u/s}$ denote the lifts of $\cW^{u/s}$ to $\wt L$. 

\begin{proposition}[Foliations] 
\label{prop:foliation_facts}
Let $f \colon L \to L$ be spA. Then there exists a standard hyperbolic metric on $L$ such that the leaves of $\wt \cW^{u/s}$ are uniformly quasigeodesic. 

Moreover, lifting to the universal cover of $M$, the intersection of $\wt L$ with a leaf of $\wt W^{u/s}$ is connected and hence a leaf of $\wt \cW^{u/s}$.
\end{proposition}

\begin{proof}
First, the moreover statement is precisely \cite[Proposition 4.2]{Fen09} and follows easily from the basic structure of $\mc F$ and $\phi$. 

Next, we recall that since $M$ is hyperbolic, Candel proved that $M$ admits a leafwise
hyperbolic metric, i.e. a metric which varies continuously and for which 
each leaf of the foliation $\mc F$ has constant curvature $-1$
\cite{candel1993uniformization}.

Now $L$ spirals onto the boundary leaves, and each component of $\del_+N$ ($\del_-N$) contains a closed curve which comes via the semiflow from a closed curve in $L$ (namely, a curve in a fundamental domain of the ``endperiodic part" of the map $f$). By continuity of the metric, the forward (backward) $f$--orbit of this curve consists of bounded length curves. 
All points in $L$ are a bounded distance from one of these bounded length curves, so the injectivity radius in the hyperbolic metric on $L$ is bounded above, and in particular the metric is standard.

Finally, let $\phi^\flat$ be the circular pseudo-Anosov flow on $M$ obtained by blowing down $\phi$. It is well-known that the stable/unstable foliations of $\phi^\flat$ have Hausdorff leaf space (in fact, they are $\mathbb{R}$-trees). Since the blowup does not change the leaf space of the stable/unstable foliations, the same is true for $\wt W^{u/s}$; see the discussion at the end of Section $3$ of \cite{Fen09}. Since the foliations $\wt \cW^{u/s}$ are determined by the intersection $\wt L \cap \wt W^{u/s}$, the leaf spaces of $\wt \cW^{u/s}$ are also Hausdorff. Hence, we may apply \cite[Theorem C]{Fen09} to conclude that the leaves of $\wt \cW^{u/s}$ are uniformly quasigeodesic in $\wt L$.
\end{proof}

\begin{remark}
\label{rmk:quasi}
If $g \colon L \to L$ is endperiodic and $f_1, f_2$ are each spA maps isotopic to $g$,
then we can choose a standard hyperbolic metric on $L$ so that, lifting to the universal
cover, the leaves of the invariant foliations of {\em both} $f_1$ and $f_2$ are uniformly quasigeodesic. 

To see this, let $N = N_g$ be the compactified mapping torus for $g$ and let $N \to M_1$ and
$N \to M_2$ be the embeddings associated to $f_1$ and $f_2$ as in
\Cref{th:spA_construction}. As in the proof of \Cref{prop:foliation_facts},
each $M_i$ has a leafwise hyperbolic metric $g_i$ and these each pull back to a continuous leafwise hyperbolic metric on $N$, which we also denote $g_i$. By continuity and compactness of $N$, the ratio $g_1/g_2$ is uniformly bounded on $N$ and hence the induced hyperbolic metrics on the leaf $L$ are biLipschitz. In particular, the two metrics on $\wt L$ are quasi-isometric and so the leaves of both sets of invariant foliations are uniformly quasigeodesic in either metric. 
\end{remark}

We have the following immediate consequence, which is the primary place where we use that 
$\phi$ is minimally blown up with respect to $\FF$.

\begin{lemma}\label{lem:unique_fixed}
Let $f \colon L \to L$ be spA. Any lift $\wt f \colon \wt L \to \wt L$ has at most one fixed point.
\end{lemma}

Since powers of spA maps are spA by definition, the same holds for powers of $f$.

\begin{proof}
Suppose $p, q \in \wt L$ are distinct fixed points of $\wt f$. Then their respective $\wt
\phi$--orbits $\wt \gamma_p$ and $\wt \gamma_q$ project to homotopic closed $\phi$--orbits
$\gamma_p$ and $\gamma_q$ in $M$. If $\gamma_p =\gamma_q$, then $N$ contains an
an essential torus, a contradiction. Otherwise,
since distinct closed orbits of a circular pseudo-Anosov flow are never homotopic,
$\gamma_p$ and $\gamma_q$ are closed orbits of a blown leaf $\lambda$ shared by $W^u$ and
$W^s$. The lifted leaf $\wt \lambda$ contains $\wt \gamma_p$ and $\wt \gamma_q$ and by
\Cref{prop:foliation_facts} its intersection $\wt \ell$ with $\wt L$ is a (connected) leaf
of $\cW^{u/s}$ that contains both $p$ and $q$.
Let $\wt a$ be an arc in $\wt \ell$ from $p$ to $q$. Its image in $L$ suspends under $\phi_N$ to give an annulus $A\subset \lambda$ contained in the interior of $N$ cobounded by $\gamma_p$ and $\gamma_q$. In particular, $A$ is a union of blown annuli that do not meet the compact leaves of $\mc F$. This contradicts that $\phi$ is minimally blown up with respect to $\mc F$.
\end{proof}

\subsection{Local dynamics on $\partial \wt L$} \label{sec:bdynamics}
The first step in proving \Cref{thm:fix and boundary} is to analyze the {\em local
  dynamics} of lifts of endperiodic maps 
to $\wt L$. This lemma gives conditions for a fixed point  on $\partial \wt L$
to be a sink or a source, not just on the boundary but for points of the disk $\DD = \wt L
\union \partial \wt L$. 

\begin{lemma}
\label{cor:gen_local_dynamics}
Let $g \colon L \to L$ be an endperiodic map, fix $n\ge0$, and let $\xi \in \partial \wt L$ be fixed by 
a lift $\wt g^n$ of $g^n$. Suppose further that there is a quasigeodesic 
ray $\wt r$ in $\wt L$ with $\wt r_\infty = \xi$ such that its projection $r$ to $L$
accumulates on an attracting (repelling) end of $L$.
Then $\xi$ is a sink (source) for the action of $\wt g^n$ on $\DD$.
\end{lemma}

\smallskip
Before the proof we need a definition, following \cite{CCF19}. Let $g \colon L \to L$ be an endperiodic map and
let $e$ be an attracting end of $L$ with period $q$. A \textbf{$g$-juncture} is a compact
$1$-manifold $J$ which is the boundary of a neighborhood $U$ of an end $e$ of $L$ such
that
\begin{itemize}
\item $g^q(U)\subset U$
\item $\bigcap_{k\ge0} g^{kq}(U)=\varnothing$.
\end{itemize}
If $e$ is a repelling end, then a $g$-juncture for $e$ is defined to be a $g^{-1}$-juncture for $e$.
Note that since 
$L$ has no boundary, each end of $L$ has a \emph{connected} $g$-juncture.

\begin{proof}
We can begin the proof by replacing $\wt r$ with the geodesic ray with the same endpoints in $\DD$ since it will accumulate on the same ends when projected to $L$.

Let $e$ denote a positive end of $L$ so that $r$ accumulates on $e$
and suppose that $e$ has period $q$. 
Let $s=\lcm(n,q)$.
Choose a connected $g$-juncture $j_0$ for $e$,
 and let $j_k=g^{ks}(j_0)$. In this proof, we will use a star superscript to denote passage to the geodesic tightening in the fixed hyperbolic metric. For example $j_k^*$ is the geodesic tightening of $j_k$.

By definition, $\{j_k\}$ escapes to $e$ and so by \cite[Theorem 4.24]{CCF19}, $j_0^*,j_1^*, j_2^*,\dots$ also escapes to $e$. 
After fixing $m$ sufficiently large, we can assume that $j_m^*$ 
separates the initial point of $r$ from the end $e$, and hence
we can choose a lift $\wt j_m^*$ of $j_m^*$ which 
separates the initial point of $\wt r$ from $\xi = \wt r_\infty$. 

At this point, we redefine $j_0 := j_m^*$ and $\wt j_0=\wt j_m^*$ and, as before, set $j_k=g^{ks}(j_0)$. Note that $j_0 = j_0^*$ by 
definition but this is not necessarily the case for $k>0$.

Let $\wt j_k$ be the lift of $j_k$ to $\wt L$ defined by $\wt g^{ks}(\wt j_0)$. 
We note that each $\wt j_k$ also separates the initial point of $\wt r$ from
$\xi$. Indeed, if $U_0$ is the neighborhood of $e$ with boundary $j_0$ and $\wt U_0$ is
its lift with boundary $\wt j_0$, whose closure in $\DD$ contains $\xi$, then $\wt U_k =
\wt g^{ks}(\wt U_0)$ is the lift of $g^{ks}(U_0)$ with boundary $\wt j_k$ and has closure
in $\DD$ that also contains $\xi = \wt g^{ks}(\xi)$. Also, since $g^{ks}(U_0) \subset U_0$, we necessarily have $\wt U_k \subset \wt U_0$.

Since $\{j_k\}$ escapes to $e$, $\{\wt j_k\}$ must escape compact subsets of $\wt L$. 
Hence the sequence $\wt j_0, \wt j_1, \wt j_2,\dots$ Hausdorff-limits in $\DD$ to an interval
$I\subset \del \wt L$. Letting $I_k$ denote the intersection of the closure $\wh
U_k$ of $\wt U_k$
with $\del\wt L$, we must have $I\subset I_k$ because the $\wt U_k$ are nested.
But $I_k$ are also the intervals bounded by the endpoints of $\wt j_k$, which are the same
as the endpoints of $\wt j_k^*$. Since $\wt j_k^*$ also escape compact subsets of $\wt L$,
we must have that $I$ is a single point. 
In particular, $\bigcap_{k\ge0} \wt g^{ks}(\wh U_0) = \{\xi \}$ and so $\xi$ is a sink for
$\wt g^n$ with respect to the action on $\DD$, as required.
\end{proof}

We will primarily use \Cref{cor:gen_local_dynamics} in the following form:

\begin{corollary} \label{lem:local_dynamics}
Let $f:L\to L$ be spA and let $r$ be an expanding (contracting) half-leaf in $L$ of period $n$. Let $\wt r$ be a lift of $r$ to $\wt L$.
Let $\wt r_\infty$
denote the endpoint of $\wt r$ on $\del \wt L$, and let $\wt f^n$ 
be a lift of $f^n$ to $\wt L$ such that $\wt f^n(\wt r)=\wt r$.
Then 
\begin{enumerate}[label=(\alph*)]
\item $\wt r_\infty$ is a sink (source)
for the action of $\wt f^n$ on $\DD$, and
\item if $g\colon L\to L$ is an endperiodic map homotopic to $f$, and $\wt g^n$ is a lift of $g$ compatible with the lift $\wt f^n$, then $\wt r_\infty$ is also a sink (source) for the action of $\wt g^n$ on $\DD$.
\end{enumerate}
\end{corollary}

\begin{proof}
We equip $L$ with a standard hyperbolic metric so that the leaves of $\wt \cW^{u/s}$ are
uniformly quasigeodesic (\Cref{prop:foliation_facts}). By \Cref{lem:halfleafend},
the expanding (contracting) half-leaf $\wt r$ accumulates on a positive (negative) end of $L$. Now apply \Cref{cor:gen_local_dynamics} to both $f$ and $g$.
\end{proof}

\subsection{Global dynamics on the hyperbolic boundary}\label{sec:global_dynamics}
We are now ready to prove one direction of \Cref{thm:fix and boundary}, namely
that the dynamics of half-leaves at fixed points of an spA map determine the corresponding dynamics at infinity.
The other direction will be proved in \Cref{prop:boundary_to_fix}.

\begin{proposition} \label{prop:fix_to_boundary}
Let $f \colon L \to L$ be spA with fixed point $p$. Let $\wt p$ be a lift of $p$ to $\wt L$, let $\wt f$ be a lift that fixes $\wt p$, and let $n\ge 1$ be such that $\wt f^n$ fixes the half-leaves at $\wt p$. Then $\wt f^n$ has multi sink-source dynamics on $\partial \wt L$
with attracting/repelling points equal to the endpoints of expanding/contracting half-leaves.
\end{proposition}

Before giving the proof, we will require the following lemma. First, following Fenley, we say that a \define{slice leaf} of a singular foliation is a line that is the union of two half-leaves meeting at their common initial point. The slice leaf is a \define{leaf line} if the two half-leaves are adjacent in the circular ordering of half-leaves about their common initial point. Note that regular leaves are leaf lines and any nonregular slice leaves contain a singularity.

\begin{lemma} \label{claim:finding_leaves}
Let $\wt \cW$ be either $\wt \cW^u$ or $\wt \cW^s$ and suppose that $x,y \in \partial \wt L$ are not separated by any leaf of $\wt \cW$. Then there is a leaf $\ell \in \wt\cW$ with $\{x,y\} \subset \partial \ell $.
\end{lemma}

\begin{proof}
Let $A$ be one of the two closed intervals in $\partial \wt L$ with endpoints $x,y$.

Define a partial order on the leaf lines of $\wt \cW$ with endpoints in $A$ as follows:
$\ell_1 \prec_A \ell_2$ if the endpoints of $\ell_2$ determine a subinterval of $A$
containing the subinterval determined by the endpoints of $\ell_1$. 
Using \Cref{prop:foliation_facts},
it is clear that every linear chain has an upper bound and so there is a maximal element $\ell_A$ by Zorn's lemma. We claim that $\ell_A$ must join $x$ to $y$. 

Let $p$ be any point of $\ell_A$ and $p_i$ a sequence converging to $p$ such that each $p_i$ is contained in the complementary component of the closure of $\ell_A$ in $\DD$ that meets $A^c$. Let $\ell_i$ be a leaf line of $\wt\cW$ through $p_i$. Then the $\ell_i$ converge to a leaf line $\ell'$ through $p$. Note that the endpoints of $\ell'$ are not in the interior of $A$. Indeed, no endpoints of $\ell_i$ are  contained in $A$ because this would contradict either the maximality of $\ell_A$ or the assumption that no leaf of $\wt\cW$ separates $x$ and $y$.

We conclude that there exists a (possibly singular) leaf $\ell$ of $\wt \cW$ containing $\ell_A$ and $\ell'$. Suppose toward a contradiction that $\ell_A$ has an endpoint $c$ in the interior of $A$. If $\ell'$ has an endpoint {outside} $A$, then $\ell$ contains a slice leaf separating $x$ and $y$, a contradiction. On the other hand if $\ell'$ joins $x$ and $y$ then this contradicts the maximality of $\ell_A$. We conclude that $\ell_A$ has no endpoint in the interior of $A$, so $\partial \ell_A = \{x,y\}$.
\qedhere
\end{proof}

\begin{proof}[Proof of \Cref{prop:fix_to_boundary}]
After replacing $f$ by $f^n$, we may assume that $f$ fixes the half-leaves at $p$. We remark that if one half-leaf at $p$ is fixed, then they all are.

Let $r$ be an expanding half-leaf at $\wt p$. We will show that that if $r'$ is a contracting half-leaf that is adjacent to $r$ in the cyclic order around $\wt p$, then each point in the innermost interval $(r'_\infty,r_\infty) \subset \partial \wt L$ is attracted to $r_\infty$ under positive powers of $\wt f$ and to $r'_\infty$ under negative powers of $\wt f$.
The case where $r$ is a contracting half-leaf is symmetric and so this will complete the proof.

There are two cases:
\smallskip

\noindent \textbf{Case 1:} \textit{$r$  (or $r'$) is not a blown half-leaf}.

Note this is the case if $p$ is nonsingular. Assume that $r$ is not a blown half-leaf; the case for $r'$ is similar.

Since $r$ is expanding and not a blown half-leaf, it is a half-leaf of $\wt {\cW}^u$ and
its interior is transverse to $\wt {\cW}^s$. Let $r',r''$ be the contracting half-leaves
starting at $\wt p$ that are adjacent to $r$ in the cyclic ordering (when $p$ is not
singular, these are the only half-leaves of 
  $\wt{\cW}^s$ through $\wt p$), and 
let $\ell$ be any leaf line of $\wt \cW^s$ that crosses the interior of $r$. We show that 
the endpoints $\wt f ^k (\partial \ell)$ converge to $r_\infty$ 
as $ k \to \infty$ and to the two point set $\{r'_\infty,  r''_\infty\}$ as $ k \to -\infty$. 
We recall that by \Cref{prop:foliation_facts} the leaves of $\wt \cW^{u/s}$ are uniformly quasigeodesic.

First, as $ k \to -\infty$, the leaves $\wt f^k(\ell)$ meet $r$ along points that converge to $\wt p$. Since the interiors of $r'$ and $r''$ do not meet singularities of  $\wt \cW^s$, $f^k(\ell)$ must converge to the leaf line $r' \cup r''$ as $ k \to -\infty$ since this is the unique leaf line of $\wt \cW^s$ through $\wt p$ that meets the side of $r' \cup r''$ containing $r$. Hence, $\wt f ^k (\partial \ell)$ converges to $\{ r'_\infty,  r''_\infty\}$ as $ k \to -\infty$ as claimed.

Next, as $ k \to \infty$, the leaves $\wt f^k(\ell)$ meet $r$ along points that exit $r$
and hence converge to $r_\infty$. If $\wt f^k(\ell)$ exits compact sets of $\wt L$, then
the sequence of leaves converges to $r_\infty$ and hence $\wt f ^k (\partial \ell)$
converge to $r_\infty$ as required. Otherwise, as $ k \to \infty$, $\wt f^k(\ell)$
converges to a line in a leaf $\wh\ell$ of $\wt \cW^s$ 
which must have $r_\infty$ as an endpoint. Let $\ell'$
be the leaf of $\wt {\cW}^u$ containing $\wt p$ (and hence containing $r'$ 
and $r''$ as half-leaves).

Hence, $\wh \ell$ and $\ell'$
are distinct leaves of $\wt \cW^{s}$ that are invariant under $\wt f$.
But this contradicts the fact that the lift of an spA map fixes at most one leaf of $\wt
\cW^{s}$ (or $\wt \cW^{u}$), which we show as follows: 
The lift $\wt f$ extends to a deck translation $\tau$ of  $\wt N$, and
since $\pi_1 N \to \pi_1 M$ is injective, $\tau$ extends to a deck translation of $\wt M$. 
Now $\ell'$ and $\wh \ell$
suspend to give leaves $R$ and $E$, respectively, of the stable foliation $\wt W^s$ of $\wt \phi$
in $\wt M$ which are fixed under $\tau$, and this implies that their closed orbits are
homotopic. 
Since distinct homotopic closed orbits of $\phi$ in $M$ are contained in the same (necessarily blown) leaf of $W^s$,
we must have that $R=E$, and by
\Cref{prop:foliation_facts} we have that the intersection $\wt L \cap R$ is connected. But
since it contains both $\ell'$ and $\wh \ell$, this is a contradiction.

\smallskip

\noindent \textbf{Case 2:} \textit{Both $r$ and $r'$ are blown half-leaves}.

In this case, $r$ and $r'$ are half-leaves of both $\wt \cW^s$ and $\wt \cW^u$. Let $[r_\infty, r'_\infty] \subset \partial \wt L$ be the interval between their endpoints that does not contain the other endpoints of half-leaves based at $\wt p$.
By \Cref{lem:local_dynamics}, on the action of $\wt f$ on $[r_\infty, r'_\infty]$ is such that $r_\infty$ is locally an attractor and $r'_\infty$ is locally a repeller. Hence, it suffices to show that there are no other fixed points in $[r_\infty, r'_\infty]$. 

Suppose towards a contradiction that $(r_\infty, r'_\infty)$ contains a fixed point
of $\wt f$. We will show that there is an $\wt f$--invariant leaf line $\wh \ell$ with
endpoints in $[r_\infty,r'_\infty)$, which gives a contradiction exactly as at the end of Case
  $1$ (since $r$ and $r'$ are already fixed by $\wt f$, and cannot be in the same leaf with
  $\wh \ell$).

Let $q_1$ be the fixed point closest to $r_\infty$; this exists since the fixed point set
of $\wt f$ is closed and
any point sufficiently close to $r_\infty$ cannot be fixed. 
First suppose that there is a leaf line $\ell$ of $\wt \cW^s$ that separates $r_\infty$
from $q_1$. Then the limit $\wt f^k(\ell)$ as $ k \to \infty$ is an $\wt f$--invariant leaf
line $\wh \ell$ of $\wt \cW^s$ that joins $r_\infty$ with some fixed point $q$ in
$(r_\infty, r'_\infty)$ (here
we are using the fact that $r'_\infty$ is a local repeller). If no such separating
leaf line exists, then by
\Cref{claim:finding_leaves} we obtain a leaf line $\ell$ of $\wt \cW^s$ with endpoints
$r_\infty$ and $q_1$. If $\ell$ is not fixed by $\wt f$, consider the family of all leaf lines
joining $r_\infty$ to $q$. This set is closed and $\wt f$--invariant and so its boundary
lines are fixed, and we let $\wh \ell$ be one of them.   The contradiction now proceeds
as in Case 1. 
\end{proof}

As a consequence of
\Cref{lem:local_dynamics} and \Cref{prop:fix_to_boundary}
we can show that two isotopic spA maps have the same local dynamics. This will be essential in \Cref{sec:no_blow}.

\begin{lemma} \label{lem:same_dyn}
Suppose that $f,f' \colon L \to L$ are homotopic spA maps. Let $\wt f, \wt f'$ be corresponding lifts to $\wt L$ and suppose that $\wt f$ fixes a point $\wt p$ and its half-leaves. Then $f'$ also fixes a (unique) point $\wt q$ and its half-leaves. There is a bijective correspondence between the half-leaves at $\wt p$ and $\wt q$ induced by having the same limit point on $\partial \wt L$.

Moreover, a half-leaf $\wt r$ at $\wt p$ has image in $L$ that escapes to an end $e$ if and only if the corresponding leaf $\wt r'$ at $\wt q$ has the same property.
\end{lemma}

The proof will use the following well-known fact. The proof we give appears in Farb--Margalit \cite{FM}, where it is attributed to Handel \cite{handel1985global}.

\begin{lemma} \label{lem:fixed_points}
Let $g \colon \DD \to \DD$ be a homeomorphism with at least $4$ fixed points on $\partial \DD$ each of which has an attracting or repelling neighborhood in $\DD$. Then $g$ has a fixed point in $\mathrm{int}({\DD})$. 
\end{lemma}

\begin{proof}
Double the map to get a homeomorphism $Dg \colon S^2 \to S^2$. By assumption, this map has at least $4$ fixed points along the equator, each of which is attracting or repelling and hence has positive Lefschetz index. Since $\chi(S^2) = 2$, there must be a fixed point of negative index. Since this fixed point necessarily occurs off the equator, we have found a fixed point of $g$ in $\intr(\DD)$. 
\end{proof}

\begin{proof}[Proof of \Cref{lem:same_dyn}]
By part $(a)$ of \Cref{lem:local_dynamics}, each endpoint of each half-leaf at $\wt p$ has an attracting/repelling neighborhood in $\DD$. By part $(b)$, the same is true for the action of $\wt f'$. Hence, by \Cref{lem:fixed_points}, $\wt f'$ has a fixed point $\wt q$ in $\wt L$. We then apply \Cref{prop:fix_to_boundary} to $\wt f'$ at $\wt q$ to complete the proof of the first claim.

For the moreover statement, we apply \Cref{rmk:quasi} to choose a standard metric on $L$ so that the leaves of both sets of invariant foliations on $L$ have uniformly quasigeodesic lifts to $\wt L$. Hence, $\wt r$ and $\wt r'$ fellow travel in $\wt L$ and so their images fellow travel in $L$. In particular, one escapes $e$ if and only if the other does. This completes the proof.
\end{proof}

\subsection{Existence of spA$^+$ representatives}
\label{sec:no_blow}
We will now prove \Cref{th:no_blow} (\Cref{th:intro_spA+} from the introduction) which states that any atoroidal endperiodic map is in
fact isotopic to an spA$^+$ map.

\begin{proof}[Proof of \Cref{th:no_blow}]
The proof proceeds just as for \Cref{th:spA_construction}, where we must now show that for an appropriate $M = M(h,N)$, we have $\phi^\flat = \phi$, i.e. $\partial_\pm N$ is positively transverse to the circular pseudo-Anosov flow $\phi$. We assume, as we may, that $N$ is not a product.

First choose an auxiliary spA map $f_{a}$ as in \Cref{th:spA_construction}. Using the spA package associated to $f_a$, we can consider the collection $C_1$ of isotopy classes of curves in $\partial_\pm N$ that are closed leaves of the singular foliation $\partial_\pm N \cap W_a^{s/u}$ on $\partial_\pm N$. Also, let $C_2$ be the collection of curves in $\partial_\pm N$ that cobound essential annuli of $N$ from $\partial_-N$ to $\partial_+ N$. Let $C = C_1 \cup C_2$.

Now let $h\colon \partial_\pm N \to \partial_\pm N$ be a component-wise homeomorphism that
satisfies \Cref{prop:double_to_fiber} and for which no curve of $C$ is mapped into $C$, up
to isotopy. This can be achieved just as in the proof of \Cref{lem:hyp}.
Let $M =M(h,N)$ be the associated $h$--double (\Cref{sec:ext}). As in \Cref{lem:hyp}, $M$ is hyperbolic and we consider the spA map $f^\dagger$ produced in the proof of \Cref{th:spA_construction}.  

We show that $f^\dagger$ is spA$^+$. Referring to the proof of \Cref{th:spA_construction},
it suffices to show that $\phi$ has no blown annuli, i.e. that $\phi^\flat =
\phi$. Suppose otherwise that there is a blown annulus $A$. After cutting $A$ along
$\partial_\pm$ we obtain, as in \Cref{lem:annulus leaf structure} and \Cref{lem:annulus_fibering}, 
annuli corresponding to cases (1) and (3) of \Cref{fig:four-annuli}. Case (3) yields an essential
annulus in $N$, which intersects $\partial_\pm$ in curves of $C_2$. Case (1) yields
the suspension of an escaping (\Cref{lem:halfleafend}), expanding/contracting half-leaf of $f^\dagger$. We show
this latter case meets $\partial_\pm$ in curves of $C_1$. Note this is not immediate since $W^{u/s}$ and $W^{u/s}_a$ are foliations on different manifolds with different flows.

Let $r$ be an escaping, expanding/contracting half-leaf of $f^\dagger$. By \Cref{lem:same_dyn}, there is a corresponding escaping, periodic half-leaf $r_a$ of $f_a$. Moreover, from a spiraling neighborhood $U$ of $\partial_\pm N$, one sees that the suspensions $H$ and $H_a$ of $r$ and $r_a$ under their respective flows produce homotopic curves $H \cap \partial_\pm N$ and $H_a \cap \partial_\pm N$. Hence, we have show that all curves of $\partial_\pm N \cap A$ are contained in $C$. But from the construction of $M = M(h,N)$, we see that $h$ maps curves in $\partial_+ N \cap A$ to curves in $\partial_- N \cap A$ and hence curves in $C$ back into $C$. This contradicts our choice of $h$ and completes the proof.
\end{proof}

\subsection{Sink-source dynamics imply fixed points}
\label{sec:sink-source converse}
With \Cref{th:no_blow} in hand we can complete the proof of \Cref{thm:fix and
  boundary}. The remaining direction is stated in this proposition: 

\begin{proposition} \label{prop:boundary_to_fix}
Let $f \colon L \to L$ be spA. Suppose that $\wt f$ is a lift to $\wt L$ such that $\wt f^n$ acts with multi sink-source dynamics on $\del \wt L$
for some $n\ge 1$. Then $\wt f$ has a unique fixed point $\wt p$ in $\wt L$ whose half-leaves are fixed by $\wt f^n$ such that the endpoints of expanding/contracting half-leaves at $\wt p$ are exactly its attracting/repelling points in $\partial \wt L$.
\end{proposition}

\begin{proof}
It suffices to show the existence of the (necessarily unique by \Cref{lem:unique_fixed})
fixed point $\wt p$ since the rest then follows from \Cref{prop:fix_to_boundary}.
Note that we cannot directly use the Lefschetz argument of \Cref{lem:fixed_points} as before,
because our hypothesis does not give sink/source properties in the
disk $\DD$, only on its boundary. 

Further, if $\wt p$ is fixed by $\wt f^n$, then it must also be fixed by $\wt f$ since
otherwise $\wt f^n$ has multiple fixed points. 
Hence, it suffices to replace $\wt f$ with
$\wt f^n$ and we do so now.

Next, by \Cref{lem:same_dyn} we are free to replace $f$ with any isotopic spA map; let $f^\dagger$ be an isotopic spA$^+$ map whose existence is guaranteed by \Cref{th:no_blow}. 
So it remains to prove that if a lift $\wt f^\dagger$ has multi sink-source dynamics on $\partial \wt L$ then it has a fixed point in $\wt L$. The key place we use that $f^\dagger$ is spA$^+$ is that its invariant foliations $\cW^{u/s}$ are transverse away from singularities, i.e. there are no blown leaves.

 Let $x_1,x_2$ be attracting fixed points of $\wt f^\dagger$ on $\partial \wt L$. First
  suppose that they are joined by slice leaf $\ell$ of $\cW$, where $\cW$ denotes either
  $\cW^u$ or  $\cW^s$ (and let $\cW'$ denote the other one). 
As in the proof of \Cref{prop:fix_to_boundary}, by choosing $\ell$ to be a boundary leaf
of the family of all lines of $\cW$ joining $x_1,x_2$ we can assume that $\wt f^\dagger(\ell) = \ell$.
Let $\ell'$ be any regular leaf of, say, $\cW'$ that crosses $\ell$. 
 By the multi sink-source dynamics, the end points of $(\wt f^\dagger)^{-k} (\ell')$ converge to distinct repelling points $y_1,y_2$ in $\partial \wt L$,
as $k\to\infty$.
We conclude that there is an $\wt f^\dagger$-invariant leaf line $\ell''$ joining
$y_1,y_2$ and the intersection $\ell \cap \ell''$ in $\wt L$ is the required fixed point
$\wt p$.

Now suppose $x_1$ and $x_2$ are not joined by a slice leaf 
of either $\cW^{u/s}$. Then
according to \Cref{claim:finding_leaves} they are separated by a slice leaf $\ell$ of
$\cW^u$ (or of $\cW^s$). Because there are infinitely many choices for $\ell$ we can 
assume it has endpoints that are not fixed. Again considering the limit of 
$(\wt f^\dagger)^{-k} (\ell)$ as $k\to\infty$, the endpoints converge to distinct repelling
points $y_1,y_2\in \partial\wt L$,
which are therefore joined by a leaf of one of the
foliations. This then reduces to the previous case, with $\wt f^\dagger$ replaced by $(\wt f^\dagger)^{-1}$.
\end{proof}

\section{Invariant laminations and topological entropy}
\label{sec:lam_entropy}

Over the next two sections we establish the various characterizations of the stretch factor of an spA map as given in \Cref{th:intro_stretch}. Here, we focus on the topological entropy of spA maps, whereas the next section is devoted to the sense in which spA maps are dynamically optimal.

For this, recall from \Cref{sec:intro_results} that for any homeomorphism $g \colon L \to L$, its \define{growth rate}
$\lambda(g)$ is the exponential growth rate of its periodic points:
\[
\lambda(g) = \limsup_{n \to \infty} \sqrt[n]{\# \mathrm{Fix}(g^n)},
\]
where $\mathrm{Fix}(g)$ denotes the set of fixed points of $g$. Given an spA map $f \colon L \to L$, there is a largest compact invariant \emph{core dynamical system} $C_f \subset L$ defined in \Cref{sec:core_ent}, and the entropy $\ent(f)$ is defined to be the topological entropy of the restriction of $f$ to $C_f$. The main theorem of this section is 

\begin{theorem} \label{th:ent_spA}
If $f$ is spA, then $\ent(f) = \log \lambda(f)$.
\end{theorem}

Along the way, we define canonical invariant laminations $\Lambda^\pm$ 
associated to an spA map $f$ (\Cref{th:laminations}) and show that $C_f$ depends only on the homotopy class of the spA map: two homotopic spA maps have conjugate core dynamical systems (\Cref{lem:conjugate_cores}).

\subsection{Invariant laminations}
\label{sec:inv_lam}
Let $f\colon L \to L$ be an spA map. A point $x\in L$ is called \define{positive escaping} if $f^{n}(x)$ exits compact sets through positive ends of $L$ as $n\to \infty$ and called \define{negative escaping} if $f^n(x)$ exits compact sets through negative ends of $L$ as $n \to -\infty$.

\begin{definition}
Let $\Lambda^+$ be the union of all points in $L$ which are not negative escaping, and let $\Lambda^-$ be the union of all points in $L$ which are not positive escaping. \end{definition}

We will call $\Lambda^+$ and $\Lambda^-$ the \textbf{positive} and \textbf{negative invariant laminations} for $f$, respectively. We will justify this terminology by proving in \Cref{th:laminations} that they are the supports of (singular) sublaminations of $\mc W^u$ and $\mc W^s$, respectively. In general, a \emph{sublamination} $\Lambda$ of a singular foliation $\cW$ is a closed subset that is the union of
subleaves of $\mc W$, where a \emph{subleaf} of $\mc W$ is the union of at least $2$ half-leaves. Note that each leaf is itself a subleaf and a subleaf not containing a singularity of $\mc W$ is an entire (regular) leaf.
See \Cref{prop:1dlams_leafdef} for the precise description of how $\Lambda^\pm$ are obtained from $\mc W^{u/s}$.

\smallskip
We can now state the main theorem of this subsection.

\begin{theorem}[Invariant laminations]
\label{th:laminations}
For an spA map $f \colon L \to L$, the invariant laminations $\Lambda^\pm$ are a pair of $f$--invariant, transverse, nowhere dense (singular) laminations 
such that:
\begin{enumerate}
\item each periodic expanding half-leaf of $\cW^u$ (resp. periodic contracting half-leaf of $\cW^s$) is a half-leaf of $\Lambda^+$ (resp. $\Lambda^-$),
\item the collection of periodic leaves of $\Lambda^\pm$ is dense in $\Lambda^\pm$,
\item each half-leaf of $\Lambda^+$ (resp. $\Lambda^-$) accumulates on a positive (resp. negative) end of $L$.
\end{enumerate}
\end{theorem}

The proof of \Cref{th:laminations} will appear in \Cref{subsubsec:lamthmproof}.

We will observe in \Cref{rmk:inv_lams} that the laminations $\Lambda^\pm$ depend only on the isotopy class of $f$ rather than the specific choice of spA representative.

\subsubsection{Laminations in $N$}
For the proof of \Cref{th:laminations}, we return to the compactified mapping torus $N$ of $f$ and its semiflow $\varphi_N$. 
We refer the reader to \Cref{sec:fols} for notation surrounding the invariant foliations $W^{u/s}_N = W^{u/s} \cap N$. 

Let $W^+_N$ be the union of all $\phi_N$-orbits which do not meet $\del_-N$, and let $W^-_N$ be the union of all $\phi_N$-orbits which do not meet $\del_+N$.
For the next lemma, a \define{periodic blown half-leaf} of $W_N^{u/s}$ is the suspension of a periodic blown half-leaf of $\cW^{u/s}$. In other words, it is a component of the intersection of a leaf of $W_N^{u/s}$ with a blown annulus in $N$ that contains a closed orbit of $\phi_N$. See \Cref{fig:decomp}.

\begin{lemma} \label{lem:bad_leaves}
Let $W^\pm_N$ and $W^{u/s}_N$ be as above.
\begin{enumerate}[label=(\alph*)]
\item Let $H$ be a leaf of $W^{u}_N$. Then one of the following is true:
\begin{itemize}
\item $H$ does not contain a periodic blown half-leaf,
in which case $H$ meets $\del_- N$ if and only if the backward orbit of every point of $H$ meets $\del_-N$, or
\item $H$ contains at least one periodic blown half-leaf, in which case the points in $H$ whose backward orbits meet $\del_-N$ are exactly those points in the interior of contracting blown half-leaves.
\end{itemize}

\item The set $W^+_N$ is obtained from $W^u_N$ by removing
\begin{itemize}
\item every leaf $H$ with the property that for all $p\in H$, the backward $\phi_N$-orbit from $p$ meets $\del_-N$, and 
\item the open periodic blown half-leaves which are contracting.
\end{itemize}
\end{enumerate}
Statements (a) and (b) are true after replacing `$W^u_N$' with `$W^s_N$,' `backward with `forward,' `$\del_-N$' with `$\del_+N$,' and `contracting' with `expanding'.
\end{lemma}

\begin{proof}
Note that (a) implies (b), so we will simply prove (a). If $H$ is periodic,
this follows immediately from \Cref{lem:annulus leaf structure}, where each periodic half-leaf of $H$ is either of type (1) or (2).

Now suppose that $H$ is not periodic and consider the flow $\phi$ on $M$ with its unstable foliation $W^u$.
Since $H$ is a leaf of $W_N^u$ that is \emph{not} periodic, it does not contain a closed orbit of a blown annulus. 
As the closed orbits in the boundary of a blown annulus do not intersects $\FF_0$ (again see \Cref{fig:decomp}),
this implies that $H$ does not intersect any closed orbit of a blown annulus.
Hence all backwards $\phi$-orbits of points in $H$ are asymptotic \emph{in $M$}. The following claim implies that a single backward orbit in $H$ meets $\partial_-N$ if and only if they all do.

\begin{claim}\label{lem:allornone}
Suppose $\phi$ is a flow in a compact 3-manifold $M$ with universal cover $\wt M$, that $\wt \phi$ is the lift of $\phi$ to $\wt M$, and that $\Sigma$ is a surface positively transverse to $\phi$. Suppose $\wt x,\wt y\in \wt M$ both lie above a fixed lift $\wt \Sigma$ of $\Sigma$. If the backward $\wt \phi$-orbits from $\wt x$ and $\wt y$ are asymptotic, then $\wt x$ intersects $\wt \Sigma$ in backward time under $\wt \phi$ if and only if $\wt y$ does. 

The same holds after replacing `above' with `below,' and `backward' with `forward.'
\end{claim}

\begin{proof}[Proof of \Cref{lem:allornone}]
Let $\wt N_\epsilon$ be a regular neighborhood of $\wt \Sigma$ obtained by lifting a regular $\epsilon$-neighborhood of $\Sigma$, where $\epsilon$ is small enough so that $\wt x,\wt y\notin \wt N_\epsilon$ and the foliation of $\wt N_\epsilon$ by flowlines gives $\wt N_\epsilon$ the structure of an oriented interval bundle over $\wt \Sigma$.

Let $\wt \gamma_x$ and $\wt \gamma_y$ be the the backward orbits from $\wt x$ and $\wt y$ respectively. Suppose that $\wt \gamma_x$ meets $\wt \Sigma$. Since $N_\epsilon$ is an $I$-bundle, $\wt \gamma_x$ eventually lies below $\wt \Sigma$ and outside of $\wt N_\epsilon$. Since the forward orbits from $\wt x$ and $\wt y$ are forward asymptotic, $\wt\gamma_y$ must also eventually lie below $\wt \Sigma$. Since the roles of $\wt x$ and $\wt y$ are symmetric, we are done.
\end{proof}

Returning to our leaf $H$,  let $p,q\in H$ and suppose the backward orbit from $p$ meets
$\del_-N$. let $\wt H$ be a lift of $H$ to $\wt M$, and let $\wt p,\wt q\in \wt H$ be
lifts of $p$ and $q$ respectively. The backward $\wt \phi$-orbit from $\wt p$ intersects
some lift $\wt \Sigma$ of a component $\Sigma$ of $\del_-N$. Since the backward orbits
from $\wt p$ and $\wt q$ are asymptotic,
they must both intersect $\wt \Sigma$ by \Cref{lem:allornone}. Projecting to $M$, we see that the backward orbit from $q$ must intersect $\del_- N$.

The case where $H$ is a leaf of $W^s_N$ follows symmetrically, finishing the proof of \Cref{lem:bad_leaves}.
\end{proof}

\begin{proposition}[Laminations $W^\pm_N$]
\label{prop:lamsN}
The pair $W^\pm_N$ are nowhere dense (singular) sublaminations of $W^{u/s}_N$ such that each boundary leaf 
contains a closed orbit of $\varphi$.
\end{proposition}

\smallskip
In the proof, we will use the flow space of $\phi$, defined as follows. Let $\wt M$ denote
the universal cover of $M$. Collapsing orbits of the lifted flow defines a quotient map
$\Theta \colon \wt M \to \mc O$, and $\orb$ is called the \define{flow space} of
$\phi$. According to Fenley--Mosher \cite[Proposition 4.1, 4.2]{fenley2001quasigeodesic},
$\orb$ is homeomorphic to the plane. 
 Note that the images of $W^{u/s}$ in $\mc O$ are a pair of singular foliations.

\begin{proof}[Proof of \Cref{prop:lamsN}]
We will focus on $W^+_N$ since the case for $W^-_N$ is similar.

To show that $W^+_N$ is a sublamination of $W^u_N$, note that its complement is the set of points whose backward $\phi_N$-orbits meet $\del_-N$. This set is clearly open, so $W^+_N$ is closed. By \Cref{lem:bad_leaves}(b), $W^+_N$ is a union of subleaves of $W^u_N$, 
and hence is a (singular) sublamination. Here, we have used that fact that every singular orbit of $\varphi_N$ bounds at least $2$ contracting and $2$ expanding half-leaves.

That $W^+$ is nowhere dense follows from the fact $\phi$ has an orbit \underline{in $M$} that is dense in both the forward and backward direction. This orbit exists because every pseudo-Anosov flow on an atoroidal manifold is transitive \cite[Proposition 2.7]{Mos92}.

It remains to prove that each boundary leaf of $W^+_N$ contains a closed orbit of $\varphi$. For this, let $H$ be a boundary leaf of $W^+_N$. If we let $H_M$ be the leaf of $W^u$ that contains $H$, we claim that it suffices to show that $H_M$ contains a closed orbit of $\varphi$. Indeed, for any point $p \in H \subset H_M \cap N$, the backward $\phi$-orbit from $p$ limits on a closed orbit of $H_M$. However, since $H$ is a leaf of $W^+_N$ the backwards orbit through $p$ never meets $\partial_-N$ and hence stays in $N$. In particular, $H$ contains a closed orbit of $H_M$.

To prove that $H_M$ contains a closed orbit, we use the projection $\Theta \colon \wt M \to \mc O$ from the universal cover of $\wt M$ to the flow space $\mc O$ of $\phi$. First, let $\wt N$ be a component of the preimage of $N$ in $\wt M$ and let $\wt H$ be a component of the preimage of $H$ contained in $N$. Since $H$ is a boundary leaf, there is a component $\Sigma$ of $\partial_-N$ and a lift $\wt \Sigma$ to the boundary of $\wt N$ with the property that points arbitrarily close to $\wt H$ flow backwards to meet $\wt \Sigma$. 
Hence, the image $\Theta(\wt H)$ lies at the boundary of the $\Theta (\wt \Sigma)$. (For context, the topological boundary of $\Theta (\wt \Sigma)$ is a union of slices leaves of the projected foliations \cite[Theorem 4.1]{Fen09}.) But we prove in \cite{LMT_stst} that
each slice leaf in the boundary of $\Theta (\wt \Sigma)$ is periodic. (For $\varphi$  an honest pseudo-Anosov suspension flow, this was proved in \cite[Lemma 3.21]{cooper1994bundles}.)
This implies that $H_M$ contains a periodic orbit as required.
\end{proof}

\subsubsection{Proof of \Cref{th:laminations}}\label{subsubsec:lamthmproof}

It remains to translate properties of $W^\pm_N$ into properties of the laminations $\Lambda^\pm$.

\begin{corollary}\label{fact:escaping_leaf}
Let $\ell$ be a leaf of $\cW^u$. Then one of the following is true:
\begin{itemize}
\item $\ell$ does not contain a periodic blown half-leaf, in which case $\ell$ contains a negative escaping point if and only if every point in $\ell$ is negative escaping, or
\item $\ell$ contains at least one periodic blown half-leaf, in which case the negative escaping points in $\ell$ are exactly those contained in the interiors of contracting blown half-leaves.
\end{itemize}
The same hold for leaves of $\cW^s$ after replacing `negative' with `positive' and `contracting' with `expanding.'
\end{corollary}

\begin{proof}
Since $f \colon L \to L$ is the first return map to $L$ under the flow $\varphi$ and the positive (negative) ends of $L$ accumulate precisely on $\partial_+ N$ ($\partial_- N$), we have $\Lambda^\pm = W^\pm_N \cap L$. Moreover, leaves of $W^\pm_N$ are in bijective correspondence with $f$--orbits of leaves of $\Lambda^\pm$ and so the corollary follows from \Cref{lem:bad_leaves}.
\end{proof}
\Cref{fact:escaping_leaf} and the definition of $\Lambda^\pm$ now imply the following proposition:

\begin{proposition}\label{prop:1dlams_leafdef}
The positive invariant lamination $\Lambda^+$ is obtained from the foliation $\cW^u$ by removing 
\begin{itemize}
\item the leaves for which every point is negative escaping, and
\item the open periodic (blown) half-leaves that are contracting. 
\end{itemize}
The {negative invariant lamination} $\Lambda^-$ is characterized symmetrically by replacing `$\mc W^u$' with `$\mc W^s$,' `negative' with `positive,' and `contracting' with `expanding.'
\end{proposition}

\begin{proof}[Proof of \Cref{th:laminations}]
First, since $W^\pm_N$ are nowhere dense sublaminations of $W^{u/s}_N$ by \Cref{prop:lamsN}, $\Lambda^\pm = W^\pm_N \cap L$ are nowhere dense sublaminations of $\cW^{u/s}$. 

Let $\ell$ be an expanding periodic half-leaf of $\cW^u$.
By \Cref{lem:halfleafend},
$\ell$ does not accumulate on negative ends. Hence, $\ell$ is not negative escaping and therefore is a half-leaf of $\Lambda^+$. 
Since every periodic point of $L$ is the endpoint of at least two expanding periodic half-leaves of $\cW^u$, the same is true in $\Lambda^+$; in other words, there are no ``1-pronged singular leaves" of $\Lambda^+$.
The case where $\ell$ is contracting periodic leaf of $\cW^u$ is similar and thus item $(1)$ is proven.

Next we prove item (3). Let $r$ be a half-leaf of $\Lambda^+$. Note that $r$ does not accumulate on a negative ends of $L$, because it contains no negative escaping points. Suppose for a contradiction that $r$ does not accumulate on positive end. Then $r$ is contained in a compact set $K_1 \subset L$, and hence accumulates on a sublamination $L_1$ of $\Lambda^+$. Since $L_1 \subset \Lambda^+$, $f^{n}(L_1)$ stays in a fixed compact set $K_2$ as $n \to -\infty$ and so it limits to a compact $f$-invariant sublamination $L_2 \subset \Lambda^+$. The suspension of $L_2$ under $\varphi$ 
is contained in interior of $N \subset M$, contradicting the fact that all half-leaves of $W^{u/s}$ are dense in $M$ (as in the proof of \Cref{lem:halfleafend}).

It remains to prove that periodic leaves of $\Lambda^\pm$ are dense. Since $\Lambda^\pm$ are nowhere dense, boundary leaves are dense in each and so it suffices to know that each boundary leaf is $f$--periodic. Let $\ell$ be a boundary leaf of (say) $\Lambda^+$. Then the suspension of $\ell$ under $\phi_N$ is a boundary leaf $H$ of $W^+_N$. According to \Cref{prop:lamsN}, $H$ contains a periodic orbit $\gamma$ and so its intersection $\gamma \cap L$ includes an $f$--periodic point in $\ell$. This gives item (2), completing the proof.
\end{proof}

\begin{remark}[Periodic leaves of $\Lambda^\pm$]
\label{rmk:periodic}
From the argument above, we see that $f$--periodic leaves of $\Lambda^\pm$ contain $f$--periodic points. That is, if $f^n(\ell) = \ell$ for $\ell$ a leaf of $\Lambda^\pm$, then $\ell$ contains a (unique) fixed point of $f^n$. Indeed, the suspension of $\ell$ under the semiflow is a leaf $H$ of $W^\pm_N$ that is contained in a periodic leaf $H_M$ of $W^{u/s}$. As in the proof of \Cref{prop:lamsN}, if $H$ is a leaf of $W^+_N$ (resp. $W^-_N$) then following a point of $H$ backwards (resp. forwards) along the flow, we stay in $H \subset N$ and arrive at a periodic orbit of $\varphi$. Intersecting this orbit with $\ell$ produces the required $f$--periodic point.
\end{remark}

\begin{remark}[Invariant measures on $\Lambda^\pm$] \label{rmk:no_measure}
An example of Fenley \cite{Fen97}, together with \Cref{th:lam_relation}, shows that $\Lambda^\pm$ do not always admit projectively invariant transverse measures of full support. It would be interesting to characterize when there are such transverse measures. 
\end{remark}

\subsection{Topological entropy of spA maps}
\label{sec:core_ent}
We next associate to each spA map $f \colon L \to L$ its \define{core dynamical system} $C_f$:
\[
C_f  = \{x \in L : (f^n(x))_{n \in \ZZ} \text{ is bounded in $L$}\}.
\] 

Since $f$ is endperiodic, $C_f$ is exactly the set of points that are neither positive nor negative escaping. In symbols, $C_f = \Lambda^+ \cap \Lambda^-$.

\begin{lemma}\label{lem:core}
The core $C_f$ is $f$--invariant, compact, and contains any other  $f$--invariant compact set. 
\end{lemma}

We define the \define{entropy} of an spA map $f$, denoted $\ent(f)$, to be the topological entropy of the restriction $f|_{C_f}\colon C_f \to C_f$. For background on topological entropy in the related setting of pseudo-Anosov maps, see \cite[Expos\'e 10]{FLP}.

\begin{remark}
Topological entropy is usually (and unambigously) defined for maps on compact spaces, but there have been several generalizations. For example, C\'anovas and Rodr\'iguez define the topological entropy of a general map to be the supremum of topological entropies of the restriction to compact invariant subset \cite{canovas2005topological}. With this definition, $\ent(f)$ is unchanged by \Cref{lem:core}.
\end{remark}

An endperiodic map on $L$ is a \define{translation} if every point of $L$ is both positive and negative escaping. The following lemma characterizes translations among spA maps.

\begin{lemma}\label{lem:translation}
Let $f \colon L \to L$ be spA. The following are equivalent:
{\setlength{\multicolsep}{2pt}
\begin{multicols}{2}
\begin{enumerate}
\item one of $\Lambda^+$ or $\Lambda^-$ is empty,
\item both $\Lambda^+$ and $\Lambda^-$ are empty,
\item $f$ has no periodic points,
\item $C_f$ is empty,
\item $f$ is a translation.
\end{enumerate}
\end{multicols}}
\end{lemma}

\begin{proof}
Any periodic point has a bounded orbit, so (4) implies (3). Each periodic point is contained in a leaf of both $\Lambda^+$ and $\Lambda^-$, and by \Cref{th:laminations}  the periodic leaves of $\Lambda^\pm$ are dense in $\Lambda^\pm$, so (3) implies (2). Clearly (2) implies (1), and (1) implies (4) since $C_f=\Lambda^+\cap \Lambda^-$ as noted above. Hence (1) through (4) are equivalent.

To finish the proof, we note that (2) and (5) are equivalent because $\Lambda^+=\varnothing$ and $\Lambda^-=\varnothing$ if and only if each point of $L$ is both positive and negative escaping, i.e. $f$ is a translation.\end{proof}

With an eye toward proving \Cref{th:ent_spA}, we recall the fact, due to Bowen \cite[Theorem 17]{bowen1971entropy}, that if $S \colon X \to X$ and $T \colon Y \to Y$ are maps and $F \colon X \to Y$ is continuous, surjective, equivariant, and has finite fibers, then the topological entropies of $S$ and $T$ are equal. In particular, conjugate systems have the same topological entropy.

The next proposition shows that the core dynamical system of an spA map is naturally conjugate to the restriction of a pseudo-Anosov homeomorphism to a closed invariant subspace. 

\begin{proposition} 
\label{prop:pA_subsystem}
Suppose $f \colon L \to L$ is an spA map with associated closed manifold $M$ and circular flow $\varphi$ (\Cref{rmk:package}) so that there is a cross section $S$ of $\varphi$ obtained by de-spinning (\Cref{rmk:spin_in_proof}). If $F \colon S \to S$ is the first return map for $S$, then $f \colon C_f \to C_f$ is conjugate to the restriction of $F$ to some closed, invariant subspace.
\end{proposition}

\begin{proof}
Since $S$ is obtained by de-spinning, there is an isotopy along flow lines carrying $S \cap N$ to a compact subsurface $L'$ of $L$, such that the isotopy is supported in a spiraling neighborhood $U$ of $\partial_\pm N$ (see \Cref{sec:junctureclasses}). In the notation of \Cref{sec:spA_exist}, $S \cap N = L_U$ and $L' = L \ssm U$. Since $U$ is foliated by flow segments, 
the isotopy from $L' \subset L$ to $S \cap N \subset S$ along segments of the flow is supported away from $C_f$. With this setup, $C_f \subset S$, and $f$ agrees with $F$ since they are both the return map to $C_f$ along $\varphi$.
\end{proof}

Next we show that the core dynamical system does not depend on the choice of spA representative. In other words, it is canonically associated to the isotopy class of $f$.

\begin{proposition}
\label{lem:conjugate_cores}
If $f_1, f_2 \colon L \to L$ are isotopic spA maps, then their core dynamical systems are conjugate.
\end{proposition}

For the proof, we require a bit more structure associated to the core. First, we have the following lemma about the endpoints of the lifts $\wt \Lambda^\pm$ of $\Lambda^\pm$ to $\wt L$:

\begin{lemma}\label{lem:lam_inf}
Each slice leaf of $\wt \Lambda^\pm \subset \wt L$ is uniquely determined by its endpoints in $\partial \wt L$, and leaves of $\wt \Lambda^+$ and $\wt \Lambda^-$ do not share endpoints. 

Moreover, if $\ell_1$ and $\ell_2$ are leaves of (say) $\Lambda^+$ that share an endpoint at infinity, then they share a periodic half-leaf based at a singularity.
\end{lemma}

\begin{proof}
We begin by proving the moreover statement. Suppose that $\ell_1$ and $\ell_2$ are distinct leaves of $\wt \Lambda^+$ that share an endpoint $p \in \partial \wt L$. It suffices to show that they share a half-leaf since distinct leaves that share a half-leaf are singular and all singularities in $\Lambda^+$ are periodic: if $s$ is a singularity of $\Lambda^+$ then $f^{n}(s)$ stays in a compact set of $L$ as $n \to -\infty$ and hence becomes periodic. If $\ell_1$ and $\ell_2$ do not share a half-leaf, we obtain a contradiction as follows: since the laminations $\Lambda^+$ are nowhere dense, there is also a boundary leaf line $\ell$ of $\wt \Lambda^+$ with endpoint $p$. Since boundary leaves are periodic (\Cref{prop:lamsN}), $p$ itself is fixed under some $\wt f^k$. If both $\ell_1,\ell_2$ are also fixed under $\wt f^k$, then we are done since they both contains the unique fixed point of $\wt f^k$. If not, then iterating either $\ell_1,\ell_2$ under $\wt f^k$ would produce a leaf line whose image in $\Lambda^+ \subset L$ gets arbitrarily close to a negative end, contradicting \Cref{th:laminations}.

The remaining claims now follow easily. If $\ell_1$ and $\ell_2$ are slice leaves of $\wt \Lambda^+$ with the same endpoints, then $\ell_1 = \ell_2$ because they agree on two singular half-leaves and each leaf contains a most one (necessarily periodic) singularity (\Cref{lem:unique_fixed}). Finally, if a leaf $\ell^+$ of $\wt \Lambda$ shared an endpoint with a leaf $\ell^-$ of $\wt \Lambda$, then their projection to $L$ would fellow travel. But this again contradicts item $(3)$ of \Cref{th:laminations}.
\end{proof}

Next, recall that the \textbf{double boundary} $\partial^2 \wt L$ is defined as the space of distinct pairs of points in the circle $\partial \wt L$ modulo the involution $(x,y) \mapsto (y,x)$. The laminations $\Lambda^\pm$ determine closed, pairwise unlinked, $\pi_1(L)$--invariant subsets $\db(\Lambda^\pm) \subset \partial^2 \wt L$ obtained by taking the endpoints of \emph{leaf lines} (see \Cref{sec:global_dynamics})
in $\wt \Lambda^\pm \subset \wt L$.
In light of \Cref{lem:lam_inf}, we are free to blur the distinction between a leaf line of $\wt \Lambda$ and the corresponding point in $\db(\Lambda^\pm)$.
Since $\Lambda^\pm$ are $f$-invariant, $\db(\Lambda^\pm)$ are invariant under the homeomorphism induced by any lift of $\wt f$ of $f$, which we continue to denote by $\wt f$.

Define
\[
\wt \Lambda_\pitchfork =  \{(l^+,l^-) \in \db(\Lambda^+) \times \db(\Lambda^-) :  l^+ \text { and } l^- \text{ link in } \partial \wt L\},
\]
and note that $\wt \Lambda_\pitchfork$ is $\pi_1(Y)$ invariant and $\wt f(\wt \Lambda_\pitchfork) = \wt \Lambda_\pitchfork$ for each lift of $f$. In fact, $\wt \Lambda_\pitchfork$ admits a natural equivariant surjection to the preimage $\wt C_f$ of the core in $\wt L$. To see this, define the map $\wt \Lambda_\pitchfork \to \wt C_f$ that sends each $(l^+,l^-)$ to the unique point of intersection $l^+ \cap l^-$.  This map is surjective essentially by definition, and working in foliation charts of $\wt \cW^\pm$, we see that the map is also continuous. Hence, by equivariance, it descends to a continuous surjective map that conjugates the induced action of $f$ on $\Lambda_\pitchfork = \wt \Lambda_\pitchfork/ \pi_1(L)$ to the action of $f$ on $C_f$. 

In fact, the fibers of the map $\Lambda_\pitchfork \to C_f$ are easily determined (see e.g. Casson--Bleiler \cite[Lemma 6.2]{casson-bleiler}). In particular, once sees that the fibers of the map $\Lambda_\pitchfork \to C_f$ generate an equivalence relation $\sim$ that is entirely determined by $\Lambda_\pitchfork$ 
and that the the fibers have uniformly bounded size.

With this setup, we give the 
\begin{proof}[Proof of \Cref{lem:conjugate_cores}]
From the above discussion, each spA map $f_i$ determines a subspace $\db(\Lambda^\pm_{i}) \subset \partial^2 \wt L$ in which the periodic leaf lines are dense. However, by \Cref{lem:same_dyn}, the collection of periodic leaf lines of these laminations are equal and hence $\db(\Lambda^\pm):= \db(\Lambda^\pm_{1}) = \partial^2 (\Lambda^\pm_{2})$. Since the maps $f_1,f_2$ are homotopic, compatible lifts have the same action on $\partial \wt L$ and so $f_1, f_2$ have the same action on $\Lambda_\pitchfork := \Lambda_{1\pitchfork} = \Lambda_{2\pitchfork}$. Finally, the factor maps $\Lambda_\pitchfork \to C_{f_1}$ and $\Lambda_\pitchfork \to C_{f_2}$ make exactly the same identifications by above discussion. Hence, the cores $C_{f_1}$ and $C_{f_2}$ are conjugate, as required.
\end{proof}

\begin{remark}[Invariance of laminations]
\label{rmk:inv_lams}
We record the fact, which follows from the proof of  \Cref{lem:conjugate_cores}, that if $f_1$ and $f_2$ are isotopic spA maps, then their respective invariant laminations $\Lambda_1^\pm$ and $\Lambda_2^\pm$ are isotopic (they have the same geodesic tightenings and they differ from their tightenings by a controlled pinching).
\end{remark}

We now use \Cref{prop:pA_subsystem}, \Cref{lem:conjugate_cores}, and the symbolic dynamics of pseudo-Anosov homeomorphisms to give the 

\begin{proof}[Proof of \Cref{th:ent_spA}]
According to \Cref{lem:conjugate_cores}, it suffices to prove the theorem for any isotopic spA map and by \Cref{th:no_blow} we may choose $f$ to be spA$^+$. That is, 
if we consider the spA package associated with $f$, as in \Cref{rmk:package}, then the flow $\phi$ is an honest pseudo-Anosov suspension flow on $M$.  We further assume that $M$ is constructed as an $h$--double, as in the proof of \Cref{th:no_blow}, so that the foliation $\mc F$ can be de-spun to a cross section $S$ of $\phi$; see \Cref{rmk:spin_in_proof}. 
In particular,  the first return map $F \colon S \to S$ produced by  \Cref{prop:pA_subsystem} is an honest pseudo-Anosov homeomorphism. Let $C$ be the closed $F$--invariant subspace of $S$ that is conjugate to $C_f$; it suffices to show that the entropy of $F|_C$ is equal to the logarithm of the growth rate of its fixed points. 

Let $\sigma \colon \Sigma \to \Sigma$ be a subshift of finite type associated to a Markov partition for $F \colon S \to S$ as constructed in \cite[Expos\'e 10]{FLP}. The symbolic coding $\theta \colon \Sigma \to S$ is a semiconjugacy with the property that that its fibers have uniformly bounded size. Hence, the same is true for the restricted coding: 
$\theta|_{\theta^{-1}(C)} \colon \theta^{-1}(C) \to C$. Since the entropy of a subshift of finite type is given by the logarithm of the growth rate of its fixed points \cite{bowen1970topological}, it suffices to show that $\theta^{-1}(C) \subset \Sigma$ has finite type.

According to Fried \cite[p. 492]{fried1987finitely}, it suffices to show that the subsystem $C \subset S$ is \emph{isolated}, i.e. it has an open neighborhood $U$ such that $C = \bigcap_{i\in \mathbb{Z}}F^i(U)$. This is easy from our topological setup: First, reparameterize the flow $\phi$ so that $F\colon S \to S$ is the time $1$ return map and chose any metric on $M$. Next, note that $C \subset S$ is exactly the $F$--invariant closed set with the property that no flow line through its points ever intersects $\Sigma = \partial_\pm N$ (here we are referring to the notation of  \Cref{prop:pA_subsystem} and \Cref{rmk:package}). The $\phi$--transverse surface $\Sigma$ has an $\epsilon$--collar foliated by segments of the flow and we choose an open neighborhood $U \subset S$ of $C$ so that for each $u\in U$ there is a $c\in C$ having the property that $d(\phi_t(u),\phi_t(c)) \le \epsilon/2$ for $t \in [-1,1]$. To show $C = \bigcap_{i\in \mathbb{Z}}F^i(U)$ note that any $x \in \bigcap_{i\in \mathbb{Z}}F^i(U)$ has the property that $d(\phi_t(x), \phi_{t}(C))\le \epsilon/2$ for all $t \in \mathbb{R}$. But since $\phi_t(C) \subset N$ for all $t$, $\phi_t(C)$ never meets the $\epsilon$--collar of $\Sigma$ (since such points necessarily escape $N$). Hence, $\phi_t(u)$ never meets $\Sigma$ and hence $u \in C$ as required. This completes the proof.
\end{proof}


\section{Stretch factors of spun pseudo-Anosov maps}
\label{sec:strech_factors}

In this section we complete the proof of \Cref{th:intro_stretch} by establishing that spA maps minimize growth rate among all homotopic endperiodic maps (\Cref{cor:spA_growth}) and that for an spA map $f$, $\lambda(f)$ is equal to the growth rate of intersection numbers of curves under iteration (\Cref{th:intersect}).

Taken together, this is our justification for calling $\lambda(f)$ the \define{stretch factor} of the spA map $f \colon L\to L$.

\subsection{Growth rates of fixed points} \label{sec:growth}

Having developed the required structure in \Cref{sec:sink-source}, the proof of the following theorem can proceed just as in the case of a pseudo-Anosov homeomorphism. See, for example, \cite[Theorem 14.20]{FM}.
\begin{theorem}
Let $f \colon L \to L$ be a spun pseudo-Anosov map. Then for each $n\ge 1$, the homeomorphism $f$ has the minimum number of periodic points of period $n$ among all homotopic, endperiodic homeomorphisms.
\end{theorem}

\begin{proof}[Proof sketch]
Let $g$ be another endperiodic map that is homotopic to $f$. The proof works by showing that for any period $n$ periodic point $x$ for $f$ there is a \emph{Nielsen equivalent} periodic point $y$ of $g$ with the same period $n$. Here $x$ and $y$ are Nielsen equivalent if there are lifts $\wt x$ and $\wt y$ to $\wt L$ such that compatible lifts of $f^n$ and $g^n$ fix $\wt x$ and $\wt y$ respectively. Note that no two periodic points of $f$ are Nielsen equivalent since each lift of $\wt f^k$ fixes a most one point by \Cref{lem:unique_fixed}. Hence, from this fact it follows that Nielsen equivalence determines an injective map from period $n$ periodic points of $f$ to period $n$ periodic points of $g$.

To this end, suppose that $x$ is a periodic point of $f$ of period $n$. To simplify notation, replace $f$ (and $g$) with $f^n$ (and $g^n$, respectively) so that $x$ is now a fixed point. Let $\wt x$ be a lift of $x$ and let $\wt f$ be a lift of $f$ that fixes $\wt x$. By \Cref{prop:fix_to_boundary}, there is a $k\ge 1$ such that $\wt{f}^k$ acts with multi sink-source dynamics on $\partial L$; hence, so does the $k$th power of a compatible lift $\wt g$ of $g$. First suppose that $k=1$.
Then by \Cref{cor:gen_local_dynamics}, each fixed point of $\wt g$ on $\partial L$ is attracting/repelling for the action on $\DD$. So by \Cref{lem:fixed_points}, $\wt g$ has a fixed point $\wt y$ in $\wt L$. If $k\ge 2$, then we note that $\wt g$ does not fix any points on $\partial L$. Then, referring to the proof of \Cref{lem:fixed_points}, the homotopic maps $D\wt g$ and $D\wt f$ have no fixed points on the equator and the same Lefschetz index. But the unique fixed point of $\wt f$ has negative index and so $D\wt g$, and hence $\wt g$, must have a fixed $\wt y$ point in $L$. 

The proof that $x$ and $y$ have the same period now follows exactly as in \cite[Theorem 14.20]{FM}. This completes the proof.
\end{proof}

The following corollary is now immediate:

\begin{corollary} \label{cor:spA_growth}
If $f \colon L \to L$ is spA, then 
\[
\lambda(f) = \inf_{g \simeq f} \lambda(g),
\]
where the infimum is over all endperiodic maps homotopic to $f$.

In particular, $\lambda(f)$ for an spA map $f$ depends only on the homotopy class of $f$.
\end{corollary}

\begin{remark}
It is unclear whether \Cref{cor:spA_growth} remains true if $g$ is allowed to vary over all homotopic (not necessarily endperiodic) homeomorphisms.
\end{remark}

\subsection{Stretch factors and growth of intersection numbers}

Next, we turn to the following purely topological characterization of the stretch factor. 

\begin{theorem}
\label{th:intersect}
Let $f$ be a spun pseudo-Anosov map.
Then
\begin{align} \label{eq:intersect}
\lambda(f) = \max_{\alpha,\beta} \limsup_{n \to \infty} \sqrt[n]{ i(\beta, f^n(\alpha))},
\end{align}
where the maximum is taken over all isotopy classes of essential simple closed curves. 
\end{theorem}

Recall that if $\alpha$ and $\beta$ are simple closed curves in $L$, then $i(\alpha,\beta)$ denotes their geometric intersection number, i.e. the minimal number of intersection points as $\alpha$ and $\beta$ are varied within their isotopy class. In particular, it is immediate that the right hand side of \cref{eq:intersect} depends only on the homotopy class of $f$.

\begin{remark}
We will prove \Cref{th:intersect} using an invariant train track $\mc V$ for $f \colon L \to L$ (see \Cref{sec:invariant_track}). In the finite-type setting, the analogous intersection property for pseudo-Anosov maps can be proven using the singular flat metric associated to the invariant \textit{measured} foliations. While in our setting this structure does not necessarily exist (see \Cref{rmk:no_measure}), there could still be an interesting approach using similar ideas.
\end{remark}

The proof of \Cref{th:intersect} will occupy the next several subsections. We will assume,
as we may by \Cref{th:no_blow} and \Cref{cor:spA_growth}, that $f \colon L \to L$ is an spA$^+$ map as defined in \Cref{def:spA+}. Referring to the spA package from \Cref{rmk:package}, the foliation $\mc F$, which has $L$ as a depth one leaf, is transverse to the pseudo-Anosov flow $\phi$ on $M$. We also assume, as in the construction (see \Cref{rmk:spinning_back}), that $\mc F$ is obtained by spinning a cross section $S$ of $\phi$ about the transverse surface $\Sigma = \partial_\pm N$. Finally, we assume that there exists at least one periodic point of $f$ and hence $\lambda(f)\ge 1$; otherwise the equality in \Cref{th:intersect} holds by \Cref{lem:translation}.

\subsubsection{Some veering combinatorics} \label{sec:veering}
For the purposes of proving \Cref{th:intersect} we will need some combinatorial tools.

Recall that associated to $\phi$ is the veering triangulation $\tau$ of the manifold $\cM$ obtained by removing the singular orbits of $\phi$ from $M$ \cite{agol2011ideal, gueritaud}. The $2$--skeleton of $\tau$ is a branched surface in $\cM$ transverse to $\phi$ (\cite[Theorem 5.1]{LMT21}) and, after replacing $\cM$ with the complement of a tubular neighborhood $U$ of the singular orbits, we obtain the compact model of $\tau$ considered $\tau^{(2)}$ as a \emph{partial branched surface} in $M$. See \cite{Landry_norm, LMT_stst} for details. 

Also associated to $\tau$ is an \textbf{unstable branched surface}  which we denote $B^u$ (see \cite[Section 5.2.1]{LMT20} or \cite[Section 3.3]{Landry_norm}). This branched surface is characterized by the property that it is topologically dual to $\tau$, and its intersection with each $\tau$-face $F$ is a train track whose single switch points toward the unique edge of $F$ which is topmost for the $\tau$-tetrahedron immediately below $F$. As a consequence, the portion of $B^u$ lying in a single $\tau$-tetrahedron looks like the suspension of a single train track folding move. See \Cref{fig:Bu} where the one-skeleton of $B^u$ is indicated in red; note that there are two branch lines crossing in a single `triple point' in the center of the tetrahedron.

\begin{figure}
    \centering
    \includegraphics[height=1.5in]{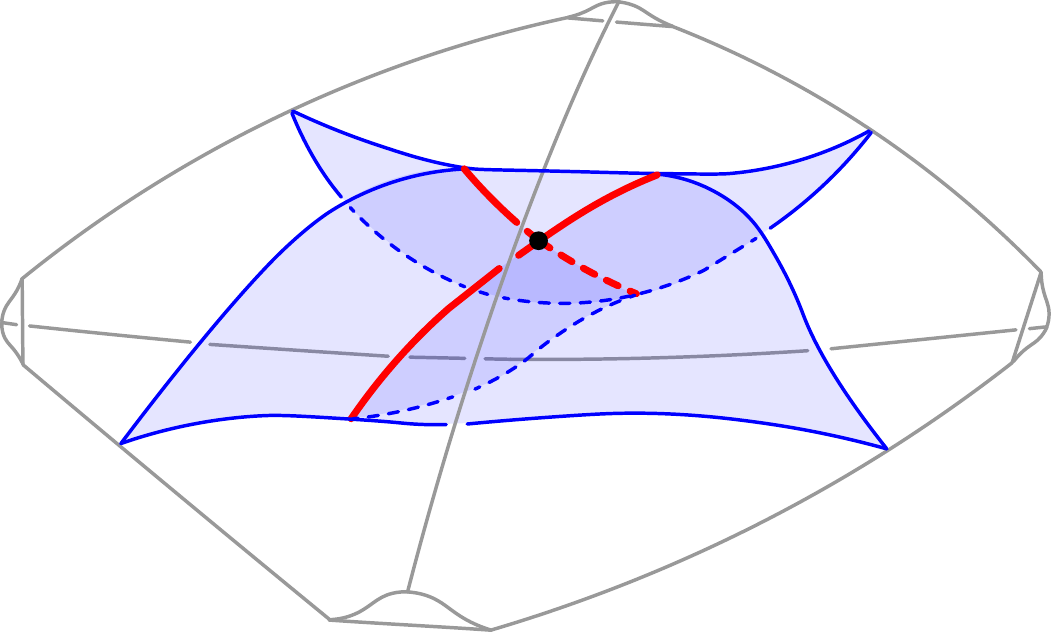}
    \caption{Part of the unstable branched surface $B^u$ lying in a single $\tau$-tetrahedron. The branch locus is indicated in red.}
    \label{fig:Bu}
\end{figure}

We now describe some combinatorics of $B^u$ that was developed in \cite[Section 4]{LMT20}. Strictly speaking the combinatorics were developed for a related branched surface $B^s$, but one can pass between the branched surfaces by reversing the coorientation on $\tau$-faces. 

The sectors of $B^u$ are diffeomorphic to rectangles. Let $s$ be a $B^u$-sector. Then the sides of $s$ inherit orientations from the \define{dual graph} $\Gamma$ of $\tau$; this is the directed graph dual to the faces of $\tau$, which is naturally identified with the $1$--skeleton of $B^u$. There is a unique corner of $s$ which is a source, which we call the \textbf{bottom point} of $s$, and a unique sink which we call the \textbf{top point} of $s$. The two remaining corners are called \textbf{side points}. The side points are linked in $\del s$ with the top and bottom points. The complementary components of the two side points containing the top and bottom points of $s$ are the \textbf{top} and \textbf{bottom} of $s$, respectively. The bottom point is the unique triple point in the bottom of $s$, while the top of $s$ may contain any number of triple points. See the lefthand side of \Cref{fig:Busector}.

\begin{figure}
    \centering
    \includegraphics{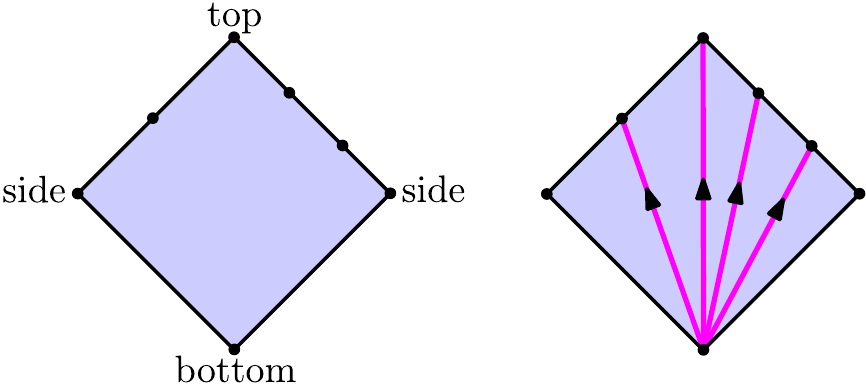}
    \caption{Left: a sector of $B^u$ with side, top, and bottom points labeled. Right: the corresponding part of the flow graph $\Phi$ in dual position.}
    \label{fig:Busector}
\end{figure}

The complementary regions of $B^u$ in $M$ are ``cusped solid tori." That is, each is diffeomorphic to the mapping torus of an $n$-cusped disk by a diffeomorphism. The number $n$ is equal to the number of prongs of the singular $\phi$-orbit at the core of the complementary region, so in particular $n\ge 3$.

The final object associated to $\tau$ is the \textbf{(unstable) flow graph} $\Phi$, which can be defined as follows: 
its vertices are the edges of $\tau$, and for each $\tau$-tetrahedron there is a directed $\Phi$-edge to the top $\tau$-edge from the bottom $\tau$-edge and the two equatorial $\tau$-edges that are involved with the associated train track folding move (see, e.g., \Cref{fig:Gfstep}).
When $\Phi$ is embedded in $M$ in this way we say it is in \textbf{standard position}. There is another position for $\Phi$ we will use called \textbf{dual position}. In dual position there is a vertex at each triple point of $B^u$ (i.e. vertex of $\Gamma$), and for each $B^u$-sector $s$ there is a directed edge in $s$ from the bottom triple point of $s$ to each triple point which is not a side vertex. See the righthand side of \Cref{fig:Busector}. In dual position, each $\Phi$-edge is positively transverse to the 2-skeleton $\hbs$. 

Starting with $\Phi$ in standard position, we can put $\Phi$ in dual position by homotoping it downward with respect to $\tau$ \cite[Section 4]{LMT20}.

\subsubsection{Carried position and the veering track on $\partial_+ N$}

Since $\Sigma$ is transverse to the pseudo-Anosov flow $\phi$, $\Sigma$ is `carried' by the partial branched surface $\tau^{(2)}$ in the following way: after an isotopy, $\Sigma$ intersects the tubular neighborhood $U$ in meridional disks transverse to $\phi$ and otherwise is contained in an $I$-fibered neighborhood of $\tau$ where it is transverse to the fibers, which can be taken to be segments of the flow. 
This follows from a strengthening of \Cref{th:stst} which is proven in \cite{LMT_stst}.

The  intersection of $\Sigma$ with the unstable branched surface $B^u$ is an (unstable) \define{veering train track} $\mc V_\Sigma$ on $\Sigma$. The complementary regions $\Sigma \cut \mc V$ are cusped $n$-gons with $n\ge 3$ because of the structure of the complementary regions of $B^u$ discussed earlier. 

The carried position of $\Sigma$ is not unique and as in \cite[Proposition 4.5]{landry2019stable} we can choose $\Sigma$ in a `lowest' position. 
In such a position the track $\mc V_\Sigma$ has no \textit{large branches}, which are branches whose switches both point inward---a large branch indicates a diagonal exchange that moves $\Sigma$ to a lower position. We fix $\Sigma$ once and for all in this position.

Since $\Sigma$ and $S$ are both carried by $\tau$, we can spin $S$ about $\Sigma$ to
produce the foliation $\mc F$ whose leaves are carried by $\tau$ in the above sense.
In particular, $L$ is carried by $\tau$ so we can define an associated unstable track $\mc
V$ on $L$  as we did for $\Sigma$.

\begin{claim}
\label{claim:trans_pos}
Any essential curve $c$ in $\Sigma$ can be isotoped to be transverse
to $\mc V_\Sigma$, such that there are no bigons in $\Sigma \cut \left( \mc V_\Sigma \cup
c \right)$.
\end{claim}

\begin{proof}[Proof of \Cref{claim:trans_pos}]
We sketch a proof of the claim for any train track $\tau$ whose complementary regions are $n$-gons  ($n\ge 3$) and which has no large branches. Such a train track carries finitely many closed curves which we will call the \emph{sinks} of $\tau$, and a unique geodesic lamination $\lambda$ with one closed leaf for each sink of $\tau$. The sinks are disjoint in $\tau$ by the no large branches condition. Each end of a noncompact leaf of $\lambda$ spirals on a closed leaf, and each closed leaf has noncompact leaves spiraling onto it from both sides in a consistent direction. We may isotope $\tau$ so that its sinks agree with the closed leaves of $\lambda$. Note that by spinning the sinks of $\tau$ in the direction of the spiraling by an isotopy supported in a small neighborhood of the sinks, we can approximate $\lambda$ arbitrarily well by $\tau$.

Given any simple closed curve $c$ which is not isotopic to a sink, its geodesic representative $c^*$ is transverse to $\lambda$ and all the complementary regions of $c^*\cup \lambda$ have nonpositive index. If $c$ is isotopic to a sink then we can push $c^*$ off the sink slightly in one direction to achieve the same end. Now by a spinning isotopy as described above, we may assume each complementary region of $c^*\cup \tau$ corresponds to a complementary region of $c^*\cup \lambda$ with the same index, ruling out bigons.
\end{proof}

\begin{remark}
\label{rmk:nobigons}
Since $L$ is obtained by spinning, the track $\mc V$ in the ends of $L$ semi-covers the track $\mc V_\Sigma$. An easy covering argument then implies that any curve $c$ far enough into the ends of $L$ can be isotoped to be transverse to $\mc V$ with no complementary bigons.
\end{remark}


\subsubsection*{Cutting the unstable flow graph along $\Sigma$}
Let $\Phi$ be the unstable flow graph of $\tau$ embedded in $M$ in dual position.

Let $\eta$ be the cohomology class dual to $\Sigma$. Let $\Phi \cut \Sigma$ denote $\Phi$
minus the edges which (in our embedding) cross faces of $\tau$ that carry
$\Sigma$ with positive weight. Then let $\Phi|\eta$ be the {\em dynamical core} of
$\Phi\cut\Sigma$, namely the subset consisting of edges that are contained in directed
cycles. We note that $\Phi|\eta$ carries all directed cycles of $\Phi$ whose pairing with
$\eta$ is $0$. 

Let $\Phi_N$ be the subgraph of $\Phi|\eta$ which is contained in $N$.

\subsubsection{The invariant veering track and growth rates}
\label{sec:invariant_track}
The argument will hinge on the following constructions, which we will justify below. 
\begin{enumerate}
  \item 
The map $f\colon L\to L$ determines a train track folding map $f_{\mc V}\colon \mc V\to \mc V$
by pushing $L$ through the tetrahedra in $\mathrm{int}(N)\cut L$ and noting that each step is a
folding move on the track, until the final track is mapped
back to $\mc V$ by $f$. 
\item
Let $G_f$ be the transition graph of the folding map $f_{\mc V}$. That is, the vertices of $G_f$
are the branches of $\mc V$, and for each branch $b$, the map $f_{\mc V}$ takes it to a sequence
of branches and $G_f$ contains a directed edge from $b$ to $c$ for each appearance of a branch $c$
in this sequence. When $f_{\mc V}$ maps $b$ homeomorphically to a single branch $c$, we get just
one edge  emanating from $b$, which terminates in $c$, and we call that a simple edge.

\item
  $G_f$ is equipped with an embedding $h_1:G_f \to N$. 
\item
 $h_1$ is homotopic to a map $h_2:G_f \to \Phi$ which is ``monotonic" in the sense that each edge of $G_f$ is either collapsed to a vertex of $\Phi$ or mapped in an orientation-preserving manner to a directed path in $\Phi$. 
\item
 $h_2$ gives a bijection between positive cycles in $G_f$ and positive cycles in
 $\Phi_N$, so that the number of edges in a cycle in $G_f$ is equal to the intersection number of its image with $L$.
\end{enumerate}

Once we have $G_f$, we define the
growth rate $\gr_{G_f}(1)$ 
as the exponential growth rate of directed cycles of length $n$, where length is
just number of edges (i.e each edge has length one). Our main claim about this is
\begin{claim}\label{claim:Gf growth is lambda} 
  $\lambda(f) = \gr_{G_f}(1).$
\end{claim}
We will prove this claim after describing points (1--5) of the construction. 

\medskip

We first describe the folding map $f_{\mc V}$. Because $M$ is fibered with fiber $S$, there is some multiple $mS$ which is fully carried by the 2-skeleton $\hbs$. Since $L$ is obtained by spinning $S$ around $\Sigma$, we obtain (possibly enlarging $m$) a cycle of parallel copies $L=L_0,L_1,\ldots,L_m=L$ of $L$, each carried by $\tau$, so that between two successive copies there is exactly one
tetrahedron. The passage from $L_i$ to $L_{i+1}$ (mod $m$) applies a folding move to the
train track $\mc V_i = L_i\cap B^u$. Composing these we obtain the map $f_{\mc V} : \mc V
\to \mc V$. This gives part (1).

Part (2) gives the definition of $G_f$. 

We obtain the map $h_1:G_f \to N$ of part (3) as follows:

Realize the union of the  $L_i$ in carried position in a fibered regular neighborhood $\mc N$ of
the $2$-skeleton of $\tau$ in $\cM$. For clarity we work in the punctured manifold, and
so each leaf $L_i$ is punctured by the singular orbits and admits an ideal triangulation: 
For each cell $c$ (face or edge) of the veering triangulation let
$\mc N(c)$ be the union of fibers of $\mc N$ passing through $c$. We then get an
ideal triangulation of $L_i$
whose cells are components of the intersection of $L$ with each $\mc N(c)$.

Between $L_i$ and $L_{i+1}$, there is one tetrahedron $t_i$ so that $L_i$ passes along its
bottom faces and $L_{i+1}$ passes along the top. The folding move is illustrated in
\Cref{fig:Gfstep}. We can build an intermediate graph $G_i$ connecting the triangulation
edges in $L_{i}$ to those in $L_{i+1}$. Outside of $t_i$, the directed edges of $G_i$ are
mapped along vertical fibers of $\mc N$ from a triangulation edge of $L_i$ to the parallel
copy above it. Within the tetrahedron we have such vertical edges together with additional
edges indicating the folding, as in the figure. 
\begin{figure}
    \centering
    \includegraphics{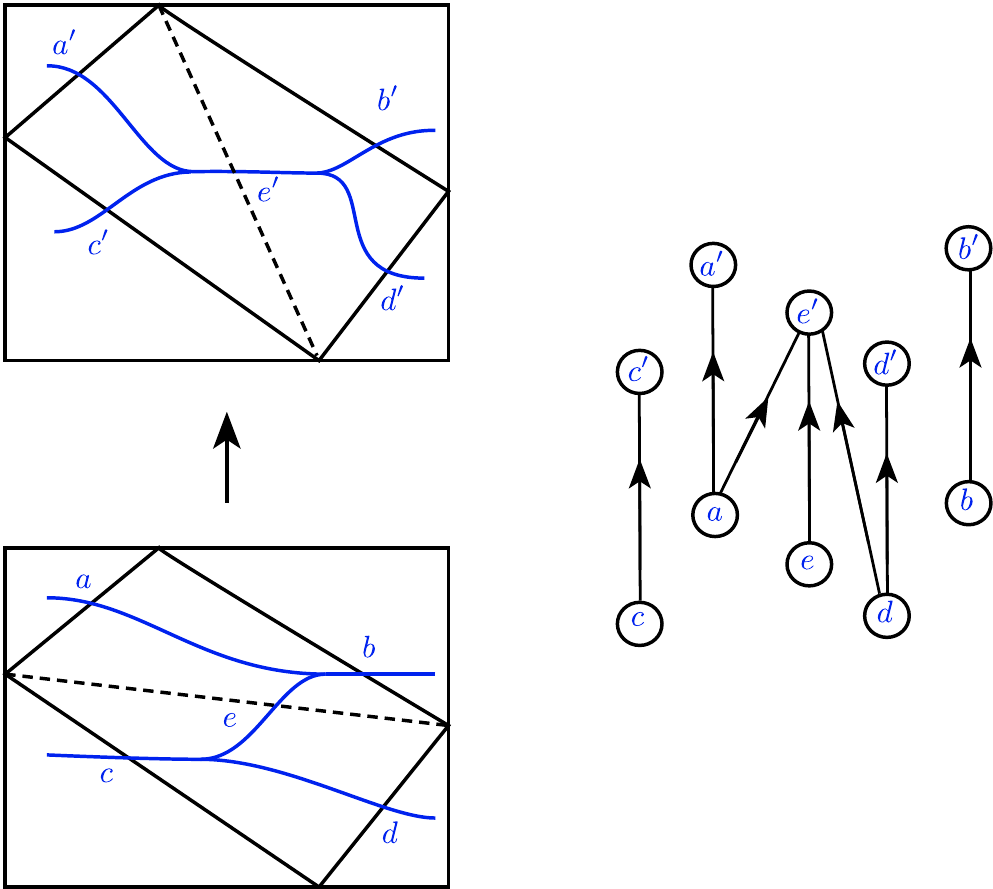}
    \caption{A folding map across a single tetrahedron produces edges in $G_{f}$}
    \label{fig:Gfstep}
\end{figure}

Now each edge of $G_f$ corresponds to a directed path in the union of the $G_i$, from
$L_0$ back to itself. This gives the map  $h_1\colon G_f\to N$ described in part (3). 

Now we observe that the three edges terminating in $e'$ in \Cref{fig:Gfstep} correspond
to the edges of the flow graph $\Phi$ for that tetrahedron, using the definition of $\Phi$ in standard position. Each of the vertical $G_i$ edges in
the figure (as well as any outside this tetrahedron) is contained in a product region $\mc N(e)$ for some edge $e$. 
Thus each edge of the intermediate graphs $G_i$ can either be mapped to an edge of $\Phi$ or collapsed to a vertex of $\Phi$. Following the map $h_1$ from edges of $G_f$ to paths along the $G_i$ with this map to $\Phi$, we obtain the map $h_2$. Note that, in the dual position embedding of $\Phi$, each vertex moves down from its corresponding $\tau$ edge to the interior of the tetrahedron below it. This is a homotopy, so we see that the final map $h_2$ is homotopic to $h_1$. This establishes part (4).

It remains to explain item (5), the correspondence between cycles of $G_f$ and cycles of
$\Phi$. 
By the construction of $h_2$, every directed cycle of $G_f$ maps to a directed cycle
of $\Phi$. Moreover, because $h_2$ and $h_1$ are homotopic, and $h_1$ maps to the
complement of $\Sigma$, the intersection number of $h_2(c)$ with $\Sigma$ is $0$ for each
cycle $c$ of $G_f$. Since $\Phi$ is positively transverse to $\Sigma$ this implies that $h_2(c)$
lies in $\Phi|\eta$. Finally each cycle in $G_f$ crosses $L$ non-trivially
which places the image cycle in $\Phi_N$.

Conversely given a directed cycle $c$ in $\Phi_N$ we must find a cycle in $G_f$ that maps
to it. Note that $c$ must have positive intersection with $L$ because
$N$ is fibered by $L$. We will obtain the cycle in $G_f$ by cutting $c$ along its
intersections with $L$, but first we consider the local picture.

Let $r$ be an edge of $\Phi_N$ that $c$ traverses. Then $r$ begins in the interior of a
tetrahedron $t_-$ whose top edge $e_-$ is identified with the initial vertex of $r$. 
It passes transversely through the 2-skeleton of $\tau$ to end in a tetrahedron $t_+$, for
which $e_-$ is either its bottom edge or one of two side edges. After a slight homotopy we
can arrange for $r$ to pass in $t_-$ directly through a small neighborhood $U$ of $e_-$ and
from there into $t_+$. See \Cref{fig:phihomotopy}.

\begin{figure}[h]
\begin{center}
\includegraphics{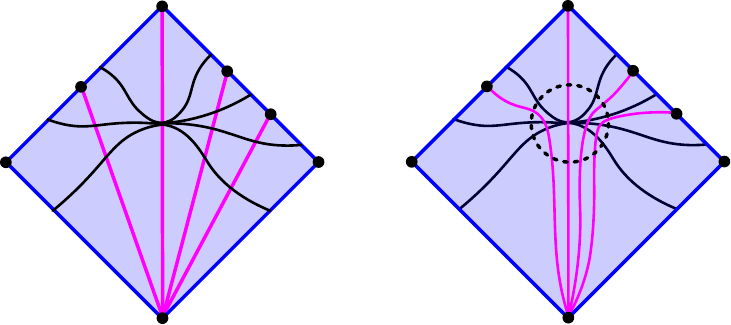}
\caption{In the proof of property (5), we homotope edges of the flow graph as shown. The dashed circle indicates the intersection of a small neighborhood of a $\tau$-edge with a sector.}
\label{fig:phihomotopy}
\end{center}
\end{figure}

Note that $\Sigma$, in its carried position, is not intersected by
$r$ since $r$ is an edge of $\Phi_N$. But $r$ might intersect a part of $L$ passing
through $U$, and this must happen for at least one such $r$ through which $c$ travels.
Thus, $r$ cut along its intersections with $L$ gives us a sequence of segments in $U$ that
are homotopic within $U$ to simple edges of the graph $G_f$ (in its $h_1$ embedding),
followed by one edge that we identify with the original $r$. 

Thus, if we now cut this modified $c$ along all its intersections with $L$, we obtain sequences of
simple $h_1(G_f)$ edges, interspersed with sequences of $\Phi_N$ edges that correspond
to non-simple edges of $h_1(G_f)$. This decomposition gives us a cycle $c'$ in $G_f$,
whose $h_2$ image is $c$. Moreover, the number of edges in $c'$ is equal to
the number of points of $c\cap L$. This establishes point (5).

\medskip

We can now give a proof of  \Cref{claim:Gf growth is lambda}. First,
from \cite[Theorem 7.1]{LMT21}
 we know that
\[
 \lambda(f) =\gr_{\Phi_N}(\xi), 
 \]
 where $\xi\in H^1(N)$ is the cohomology class dual to $L$. 
Indeed, the graph $\Phi$ encodes the flow \cite[Theorem 6.1]{LMT21} and the directed cycles of $\Phi|\eta$ are in
correspondence with the flow orbits that avoid $\Sigma$. Restricting to $\Phi_N$ gives
us the flow orbits that cross $L$, so that $\xi$ records size of the $f$-orbits associated
to each of them.

Thus \Cref{claim:Gf growth is lambda} follows from the equality
\[
\gr_{\Phi_N}(\xi)  = \gr_{G_f}(1),
\]
which is a consequence of item (5) above: it gives a 
bijection between cycles of $G_f$ and $\Phi_N$, so that the class
that counts  each edge of $G_f$ exactly once is taken to $\xi$.

\subsubsection{Finishing the proof of \Cref{th:intersect}}

Let $w^{(n)}_{ij}$ be the number of times the branch $e_i$ of $\mc V$ maps over the branch $e_j$ under $(f_{\mc V})^n$.
This quantity also counts the number of directed paths of length $n$ in $G_f$ from the
vertex $e_i$ to the vertex $e_j$.
Note that for each $i$ and $n\ge 1$, $\sum_{j}w^{(n)}_{ij}$ is finite and similarly for each
$j$ and $n\ge 1$, $\sum_{i}w^{(n)}_{ij}$ is finite.

Our next claim is straightforward given the following observation: the endperiodicity of $f$ implies that all but finitely many branches of $\mc V$ are mapped homeomorphically by $f_{\mc V}$ to other branches of $\mc V$. This implies that each end of $G_f$ has a neighborhood homeomorphic to a ray.

\begin{claim}
\label{claim:edge_growth}
\begin{align} \label{eq:gr}
 \gr_{G_f}(1) = \max_{i} \limsup_{n \to \infty} \left (\sum_{j}w^{(n)}_{ij}\right)^{1/n}.
\end{align}
\end{claim}

\begin{proof}[Sketch]
As observed above, there is a finite subgraph $G'$ of $G_f$ such that each complementary component of $G'$ is a directed ray. Then both sides of \Cref{eq:gr} are unchanged if $G_f$ is replaced by $G'$ (that is, when the right-hand side counts directed paths in $G'$). But the equality in \Cref{eq:gr} is well-known for finite graphs, see e.g. \cite[Lemma 3.1]{mcmullen2015entropy}.
\end{proof}

All the pieces are now in place for the proof of \Cref{th:intersect} which topologically characterizes the stretch factor.

\begin{proof}[Proof of \Cref{th:intersect}]
Let $\alpha$ and $\beta$ be any curves on $L$. Fix any representative $b$ of the homotopy class of $\beta$ on $L$ that is transverse to the branches of $\mc V$ and let $a$ be a representative of $\alpha$'s homotopy class contained in---not necessarily carried by---the track $\mc V$. That such a representative exists is consequence of  the fact that the complementary regions of $\mc V$ are $n$-gons. Let $a_n$ be the image of $a$ after taking $f^n(a) \subset f^n(\mc V)$ and folding it back into $\mc V$, and observe that $i(\beta, f^n(\alpha)) \le \#(b \cap a_n)$. 

After reindexing the branches of $\mc V$, we assume that $a$ traverses the branches $c_1 e_1, \ldots, c_ke_k$ where the $c_i$ are the positive integer multiplicities. Let $d_j$ be the number of intersections that $b$ has with $e_j$. Then 
\[
\#(b \cap a_n) =  \sum_j d_j \sum_{i=1}^k w^{(n)}_{ij}c_i.
\]

Using the above together with \Cref{claim:edge_growth} and \Cref{claim:Gf growth is lambda} we conclude that
\begin{align*}
\limsup_{n\to\infty} i(\beta, f^n(\alpha))^{1/n}&\le \limsup_{n\to\infty} \#(b\cap a_n)^{1/n}\\
&=\limsup_{n\to\infty}\left( \sum_{j} d_j\sum_{i=1}^k w^{(n)}_{ij}c_i\right)^{1/n}\\
&\le \max_i \limsup_{n\to\infty}\left(\sum_j w^{(n)}_{ij}\right)^{1/n}\\
&\overset{\ref{claim:edge_growth}}{=}\gr_{G_f}(1)\\
&\overset{\ref{claim:Gf growth is lambda}}{=}\lambda(f),
\end{align*}
and hence
\begin{equation}\label{eq:onedirection}
\max_{\alpha,\beta} \limsup_{n \to \infty} i(\beta, f^n(\alpha))^{1/n} \le \lambda(f).
\end{equation}

We now aim to prove the reverse of \Cref{eq:onedirection}.
Let $C$ be a recurrent component of $G_f$ with growth rate $\lambda(f)$. This subgraph $C$ exists because $\lambda(f)=\gr_{G_f}(1)$ by \Cref{claim:Gf growth is lambda}, and because the growth rate of a directed graph is always achieved by a recurrent component.

We can find a curve $\alpha$ carried on $\mc V$
that traverses every branch of $\mc V$ associated to a vertex of $C$ as follows. Let $e$ be a branch of $\mc V$ corresponding to a vertex of $C$.
Then the train paths $f^n(e)$ in $\mc V$ eventually cross every branch associated to $C$
(and possibly other branches as well). By surgery at a $\mc V$-branch corresponding to a
vertex of $C$ crossed multiple times by some $f^n(e)$, we can find a closed curve carried by $\mc V$ that traverses this $\mc V$-branch. This curve may not be embedded, but we can resolve any points of intersection to obtain an embedded multicurve. Applying $f$ sufficiently many times to a given component of this multicurve yields a simple closed curve traversing each branch corresponding to $C$. Let $\alpha$ be this curve.

Next, note that since $f^n(\alpha)$ is not eventually contained in any compact set of $L$ up to homotopy by \Cref{lem:uncontainable}, it eventually crosses a positive juncture $\beta$. From \Cref{claim:trans_pos} and \Cref{rmk:nobigons}, we can realize $\beta$ so that it is transverse to $\mc V$ and there are no complementary bigons. 
In particular, all of the carried curves $f^n(\alpha)$ are in minimal position with respect to $\beta$ and hence their geometric intersection number is given by counting their points of intersection. 

To complete the proof, reindex so that $e_1,\dots, e_\ell$ are the vertices of $C$, which we identify with the corresponding branches of $\mc V$, each of which is traversed by $\alpha$ with positive weights $c_1,\dots, c_\ell$. Also let $e_{\ell+1},\dots, e_m$ be the branches intersected by $\beta$, and $d_i$ be the number of intersections of $\beta$ with $e_i$ for $\ell<i\le m$. (Note that we do not assume that the $(e_i)$ lie in $C$ for $i>\ell$.) We will also use the symbol $e_i$ to denote the corresponding vertex of $G_f$.

The fact that $f^n(\alpha)$ intersects $\beta$ for some $n\ge1$  translates into the fact
that for some fixed $k>\ell$, there is a directed path $p$ in $G_f$ from a vertex of $C$
to $e_k$.
For this fixed $k$, each path of length $n$ from the vertex $e_i$ of $C$ to $e_k$ contributes $c_i d_k$ points of intersection to $i(\beta,f^n(\alpha))$. Then
\[
i(\beta, f^n(\alpha)) \ge \sum_{i=1}^\ell w^{(n)}_{ik}c_id_k,
\]
and so it suffices to apply the fact that
\[
\limsup_{n \to\infty} \left (\sum_i (A^n)_{ik}\right )^{\frac{1}{n}} \ge \lambda,
\]
whenever $C$ is a component with growth $\lambda$ of a directed graph with transition matrix $A$, $k$ is a fixed vertex of the graph reachable from a vertex of $C$, and the sum is over all vertices of $C$. As $C$ was chosen so that $\lambda = \lambda(f)$, the proof is complete.
\end{proof}


\section{Foliation cones and entropy functions}
\label{sec:foliation_cones}
\subsection{Foliation cones via fibered cones}

In \cite{cantwell1999foliation, cantwell2017cones}, Cantwell and Conlon show that if $N$ is a sutured manifold that admits a taut, depth one foliation suited to $N$, and each component of $\del_\pm N$ has negative Euler characteristic, then there exists a finite collection of disjoint open, convex, polyhedral cones in $H^1(N)$, called \define{foliation cones}, 
so that the integer points 
of each cone are exactly the foliated classes. See also \cite[Section 13]{CCF19}.
Here, we recall from \Cref{sec:intro_results} that a class in $H^1(N)$ is foliated if it
is dual to a depth one foliation suited to $N$; 
by \Cref{th:class_determines} each depth one foliation is determined by its cohomology class, up to isotopy. 

In this section, we show that for the atoroidal sutured manifolds considered here,
one can completely understand the foliation cones of $N$ as pullbacks of fibered cones of a closed $3$-manifold. We emphasize that this independently develops much of the Cantwell--Conlon foliation cone theory directly from the fibered face theory of Thurston \cite{thurston1986norm} and Fried \cite{Fri79}.

\begin{theorem}[Foliation cones]
\label{th:foliation_cones}
Let $N$ be an atoroidal sutured manifold such that each component of $\partial_\pm N$ is a closed surface of genus at least $2$.
Then there is a closed hyperbolic $3$-manifold $M$ and an embedding $i \colon N \to M$ with the property that a class in $H^1(N)$ is foliated if and only if it is the pullback under $i^*\colon H^1(M) \to H^1(N)$ of a fibered class that is contained in a $i(\partial_+N)$--adjacent fibered cone. 

Hence,
\[
\{ i^*\big(\C\big) \colon \C \text { is a fibered cone of } M \text{ containing } [i(\partial_+ N]) \big)\}
\]
is exactly the set of foliation cones of $H^1(N)$.
\end{theorem}

Note that from this it is immediate that there are finitely many foliation cones and each one is polyhedral and rational since this is the case for fibered cones \cite{thurston1986norm}. Also, as in \Cref{rmk:spin_in_proof}, each depth one foliation suited to $N$ is obtained from a fibration of $M$ by spinning. 

\medskip

Toward the proof of \Cref{th:foliation_cones}, we construct the manifold $M$ 
using the results from \Cref{sec:ext}.
Let $h \colon \partial_\pm N \to \partial_\pm N$ be a component-wise homeomorphism that acts by $-1$ on 
$H^1(\partial_\pm N)$ (or at least on the cone spanned by all  juncture classes) 
and such that the $h$--double $M = M(N,h)$ is hyperbolic (see \Cref{lem:hyp}). 
As before, we identify $N$ with its image in $M$ and let $\C_1, \ldots, \C_k$ be the fibered cones of $M$ that contain $[\partial_+ N]$. (Note that $[\partial_+N] = [\partial_- N]$ in $H^1(M)$.)
Each cone $\C_i$ is associated to a pseudo-Anosov flow $\varphi_i$ of $M$ for which $\partial_\pm N$ is almost transverse. We replace each $\varphi_i$ with its minimal dynamic blowup so that it is transverse to $\partial_\pm N$.

\begin{proof} 
First observe that if $\xi$ is a fiber class contained in $\C_i$, then it is represented by a fiber surface $S$ that is a cross section of the flow $\varphi_i$. Just as in the proof of \Cref{th:spA_construction}, $S$ can be spun about $\partial_\pm N$ to produce a taut, depth one foliation $\mc F$ that is transverse to $\varphi_i$. Hence, the restriction of $\mc F$ to $N$ represents the foliated class $i^*(\xi)$. 

Conversely, let $\alpha$ be foliated class in $H^1(N)$ represented by a depth one foliation $\mc F$. For $M$ fixed as above, \Cref{prop:double_to_fiber} applies to the foliation $\mc F$ and therefore it along with \Cref{rmk:spin_in_proof} gives that there is a $\phi_i$--cross section $S$ of $M$, whose fibered cone contains $[\partial_+N]$ in its boundary, such that spinning $S$ about $\partial_\pm N$ produces the foliation $\mc F$, up to isotopy. 
Hence, $\alpha$ is the pullback of the class dual to $S$ in $M$.

We note that the above shows that there exists a taut, depth one foliation suited to $N$ if and only if there exists a fibered cone of $M$ containing $[\partial_+ N]$ in its boundary.

To complete the proof, it remains to show $i^*$ takes each fibered cone onto a foliation
cone. If not, since the images of fibered cones cover the foliation cones, there are two distinct fibered cones whose pullbacks overlap in $H^1(N)$. In this case, there is a foliation of $N$, and hence $M$, which is transverse to both $\varphi_i$ and $\varphi_j$ for $i\ne j$. But then we can de-spin this foliation as above (again see \Cref{rmk:spin_in_proof}) to produce a surface that is positively transverse to $\varphi_i$ and $\varphi_j$ and meets every orbit, a contradiction.
\end{proof}

The following corollary characterizes integral classes of a foliation cone as those represented by foliations that are transverse to the restriction of an (honest) pseudo-Anosov flow.

\begin{corollary}
\label{cor:dual_semiflow}
The closed manifold $M$ in \Cref{th:foliation_cones} can be constructed so that the
pseudo-Anosov flow $\varphi$ associated to a fibered cone $\C \subset H^1(M)$ containing $[i(\partial_+ N)]$ is transverse to $i(\partial_\pm N)$. 
The integral points of the foliation cone $i^*(\C)$ are exactly the points represented by depth one foliations that are transverse to the restricted semiflow $\varphi_N$.
\end{corollary}

\begin{proof}
This was essentially established in the proof of \Cref{th:foliation_cones}, except that $\varphi$ was the dynamic blow up of a pseudo-Anosov flow. 
However, by applying the proof of \Cref{th:no_blow}, we see that 
there is a choice of $h \colon \partial_\pm N \to \partial_\pm N$
so that the flow $\phi$ on $M = M(N,h)$ transverse to $\partial_\pm N$
is an honest pseudo-Anosov flow. In fact, this can be achieved by taking $h$ on each component to be a hyperelliptic involution composed with a sufficiently high power of a fixed pseudo-Anosov homeomorphism that acts trivially on $H_1$ (see \Cref{lem:hyp}). Taking a larger power if necessary, we see that $M$ can be chosen so that each
$\phi_i$ is transverse to $\partial_\pm N$ for each of the (finitely many) fibered cones containing $[i(\partial_+ N)]$ in its boundary.
\end{proof}

\subsection{Entropy functions of foliation cones}
Fix $N$ as in the statement of \Cref{th:foliation_cones} and let $\C \subset H^1(N)$ be one of its foliation cones. For each depth one foliation $\mc F$ suited to $N$, we let $h_\mc F \colon L \to L$ be its monodromy, defined up to isotopy, where $L$ is a depth one leaf of $\mc F$. Let $\lambda([\mc F])$ denote the stretch factor of any spA representative of $h_{\mc F}$. By \Cref{cor:spA_growth} and \Cref{th:class_determines}, this depends only on the class $[\mc F] \in H^1(N)$.

\begin{theorem}
\label{th:entropy_cone}
Let $N$ be a sutured manifold such that each component of $\partial_\pm N$ is a closed surface of genus at least $2$, and let $\C \subset H^1(N)$ be a foliation cone of $N$. 
Then the function
\[
[\mc F] \mapsto \log \lambda([\mc F])
\]
that assigns to a depth one foliation the logarithm of the stretch factor of its monodromy extends to a convex, continuous function $\ent \colon \C \to [0,\infty)$.
\end{theorem}

\begin{proof}
Let $M$ and $\varphi$ be as in \Cref{cor:dual_semiflow}.
 Note that for any $\mc F$ with $\xi = [\mc F] \in \mc C$ with depth one leaf $L$, the first return map to $L$ under $\phi_N$ is an spA map $f \colon L \to L$. Hence, $\lambda([\mc F]) = \lambda(f)$ is exactly $\gr_{\phi_N}(\xi)$, the growth rate of closed orbits of $\phi_N$ with respect to class $\xi$. In the terminology of \cite{LMT21}, $\xi$ is strongly positive since it is the pullback of a fibered class in $M$ (\Cref{th:foliation_cones}). Hence, \cite[Theorems 8.1, 9.1]{LMT21} and \cite[Remark 9.2]{LMT21} imply that $\ent_{\phi_N}(\xi) = \log (\gr_{\phi_N}(\xi))$ defines a continuous, convex function on $\xi \in \C$.
\end{proof}

\subsection{spA maps as limits of pseudo-Anosovs and accumulations of stretch factors}
\label{sec:Lein}

Let $M$ be a closed, hyperbolic fibered manifold and let $\mc C \subset H^1(M)$ be a fibered cone with associated pseudo-Anosov suspension flow $\phi$. 

Let $\alpha_i$ be a sequence of primitive integral classes in the interior of $\mc C$, each one representing a cross section $S_i$ with first return map $f_i$; we set $\lambda(\alpha_i) = \lambda(f_i)$ to be the stretch factor of the pseudo-Anosov $f_i$. 

In \cite[Theorem 9.10]{LMT21}, we characterize the limit set of all stretch factors arising from a single fibered cone $\mc C$, answering a question of Chris Leininger. Here we show that all such limits are themselves stretch factors of spun pseudo-Anosov maps:

\begin{theorem}
With the setup above, if $\lambda(\alpha_i) \to \lambda >1$, then $\lambda$ is the stretch factor of an spA map $f \colon L \to L$. 

Moreover, there is a face $\mc C' \subset \partial \mc C$ such that for any primitive integral $\eta$ in the relative interior of $\mc C'$, $L$ can be chosen to be a leaf of a depth one foliation of $M$ that is obtained by spinning a fiber surface of $M$ about a surface $\Sigma$ that is dual to $\eta$ and almost transverse to $\phi$. 
\end{theorem}

Note that if $\eta$ and $\alpha$ are primitive integral classes in $\partial \C$ and $\mathrm{int}(\mc C)$, respectively, then $\alpha_i = \alpha + i \eta$ for $i\in \mathbb N$ satisfy the above conditions so along as there are infinitely many closed primitive orbits of $\varphi$ that are $\eta$--null (see \cite[Corollary 9.8 and Lemma 9.11]{LMT21}). Here, an orbit $\gamma$ is $\eta$--null if $\eta(\gamma) = 0$.

\begin{proof}
For $\alpha \in \mathrm{int} (\C)$ and $\eta \in \C$, let $\gr_{\varphi}(\alpha; \eta)$ be the growth rate of $\eta$--null closed orbits of $\varphi$ with respect to $\alpha$. 

The proof of Claim 9.11 in \cite{LMT21} established that, after passing to a subsequence, there exists a subface $\mc C'$ in the boundary of $\mc C$ such that for any primitive integral $\eta$ in the relative interior of $\mc C'$, the sequence $\gr_{\varphi}(\alpha_i; \eta)$ is eventually constant and that
\[
\lambda(\alpha_i) = \gr_\varphi(\alpha_i) \to \gr_{\varphi}(\alpha; \eta),
\]
where $\alpha$ is equal to some $\alpha_i$ for $i$ sufficiently large. We remark that the proof \cite[Claim 9.11]{LMT21} is only explicitly given for the manifold obtained from $M$ by removing the singular orbits of $\phi$, but essentially the same proof applies here.

Now let $\Sigma$ be a surface dual to $\eta$ that is almost transverse to $\phi$ (\Cref{th:stst}) and let $\phi^\sharp$ be a dynamic blowup of $\phi$ that is transverse to $\Sigma$. Also let $S$ be a cross section of $\phi$ (and hence $\phi^\sharp$) that is dual to $\alpha$, and let $\mc F$ be the $\phi^\sharp$--transverse foliation of $M$ obtained by spinning $S$ about $\Sigma$. If we cut $M$ along $\Sigma$, we obtain a disjoint union $N_1, \ldots, N_k$ of manifolds with induced foliations $\mc F_1, \ldots, \mc F_k$ transverse to semiflows $\phi_{N_1}, \ldots, \phi_{N_k}$. 

Since the $\eta$--null orbits of $\varphi^\sharp$ are exactly the orbits missing $\Sigma$, and hence contained in some $N_i$, we have
\begin{align*}
 \gr_{\varphi}(\alpha; \eta) &=  \gr_{\varphi^\sharp}(\alpha; \eta) \\
 &= \max_i  \gr_{\varphi_{N_i}}(\alpha)\\
 & = \max_i \lambda(f_i),
\end{align*}
where $f_i$ is the spA return map to a depth one leaf $L_i$ of $\mc F_i$. This completes the proof.
\end{proof}

\section{Connection to Handel--Miller theory}
\label{sec:HM}

In this section, we relate the structure of spun pseudo-Anosov maps developed here to Handel--Miller theory.

\subsection{Handel--Miller basics}\label{sec:HMbasics}

Let $g\colon L\to L$ be an endperiodic map. 
Fix a standard hyperbolic metric on $L$, and choose systems of $g$-junctures (see \Cref{sec:bdynamics}) for all $g$-cycles of ends of $L$. In \cite{CCF19} it is proven that the geodesic tightenings of the negative $g$-junctures limit on a geodesic lamination $\Lambda^+_\HM$ and symmetrically the geodesic representatives of the positive $g$-junctures limit on a geodesic lamination $\Lambda^-_\HM$. The laminations $\Lambda^+_\HM$ and $\Lambda^-_\HM$ are called the positive and negative \textbf{Handel--Miller laminations}, respectively. These laminations are mutually transverse. Furthermore by \cite[Corollary 4.72]{CCF19} $\Lambda^\pm_\HM$ are independent of the choice of junctures, and by \cite[Corollary 10.16]{CCF19}, $\Lambda^\pm_\HM$ are independent of the choice of standard hyperbolic metric up to an ambient isotopy of $L$.

One property of $\Lambda^\pm_\HM$ we will use below is that each leaf of say $\Lambda^+_\HM$ accumulates on a positive end and hence intersects positive $g$-junctures.

By \cite[Theorem 4.54]{CCF19}, $g$ is isotopic to an endperiodic map $h\colon L\to L$ which permutes the set of tightened $g$-junctures and leaves invariant $\Lambda^\pm_\HM$. 
This map $h$ is called a \textbf{Handel--Miller map}, or sometimes a \textbf{Handel--Miller representative} of the isotopy class of $g$. While $h$ is not uniquely determined, it is uniquely determined on $\Lambda^+_\HM\cap \Lambda^-_\HM\subset L$.

Let $\ms U_+$ be the set consisting of all points in $L$ that escape to the positive ends of $L$ under iteration of $h$, and $\ms U_-$ the set consisting of points that escape to the negative ends of $L$ under iteration of $h^{-1}$. We call $\ms U_+$ and $\ms U_-$ the positive and negative \textbf{escaping sets}, respectively.
By \cite[Lemma 4.71]{CCF19}, $\Lambda^+_\HM=\del \ms U_-$ and $\Lambda^-_\HM=\del \ms U_+$.

Let $\mc P_+=L-(\Lambda^+_\HM\cup \ms U_-)$, and symmetrically $\mc P_-=L-(\Lambda^-_\HM\cup \ms U_+)$. A component of $\mc P_+$ or $\mc P_-$ is called a positive or negative \textbf{principal region}, respectively. A negative principal region is a positive principal region of $h^{-1}$ and vice versa.

The structure of principal regions is developed in detail in \cite[Sections 5.3,
  6.1-6.4]{CCF19}.  In the case of interest to us,
when $L$ is without boundary and $h$ is atoroidal, this  is particularly simple
and directly analogous to what happens for pseudo-Anosov maps of closed surfaces. 
We will sketch the structure with these hypotheses, which we will be assuming from now
on.

Each positive principal region $P_+$ is homeomorphic to an open disk, and any lift $\wt P_+$ to the
universal cover $\wt L$ is the interior of a finite-sided polygon with vertices in $\partial
\wt L$. This lift has a {\em dual} lift $\wt P_-$ of a negative principal region, whose
vertices on $\partial \wt L$ alternate with those of $\wt P_+$ (see
\Cref{fig:principalregions}). 
A suitable power of a lift of $h$ preserves the vertices of both dual
polygons.
The dynamics of this lift on the circle at infinity are as
follows:

\begin{lemma}\label{lem:principalboundaryaction}
Let $h\colon L\to L$ be an Handel-Miller atoroidal endperiodic map. Let $\wt P_+$ be the
lift of a positive principal region of $h$ with dual lifted negative principal region $\wt
P_-$. Suppose $\wt{h^p}$ is a lift of a power of $h$ to $\wt L$ that preserves $P_+$, and
hence $P_-$, as well as all their vertices. Then the action of $\wt{ h^p}$ on $\del \wt L$ is
multi sink-source, where the sources are exactly the vertices of $\wt P_-$ and
the sinks are the vertices of $\wt P_+$. 
\end{lemma}

\Cref{lem:principalboundaryaction} follows from Proposition 6.9(ii) and Corollary 6.13 \cite{CCF19} since in our setting lifts of principal regions have finitely many ideal vertices on the circle at infinity.

\begin{figure}
    \centering
    \includegraphics[height=2in]{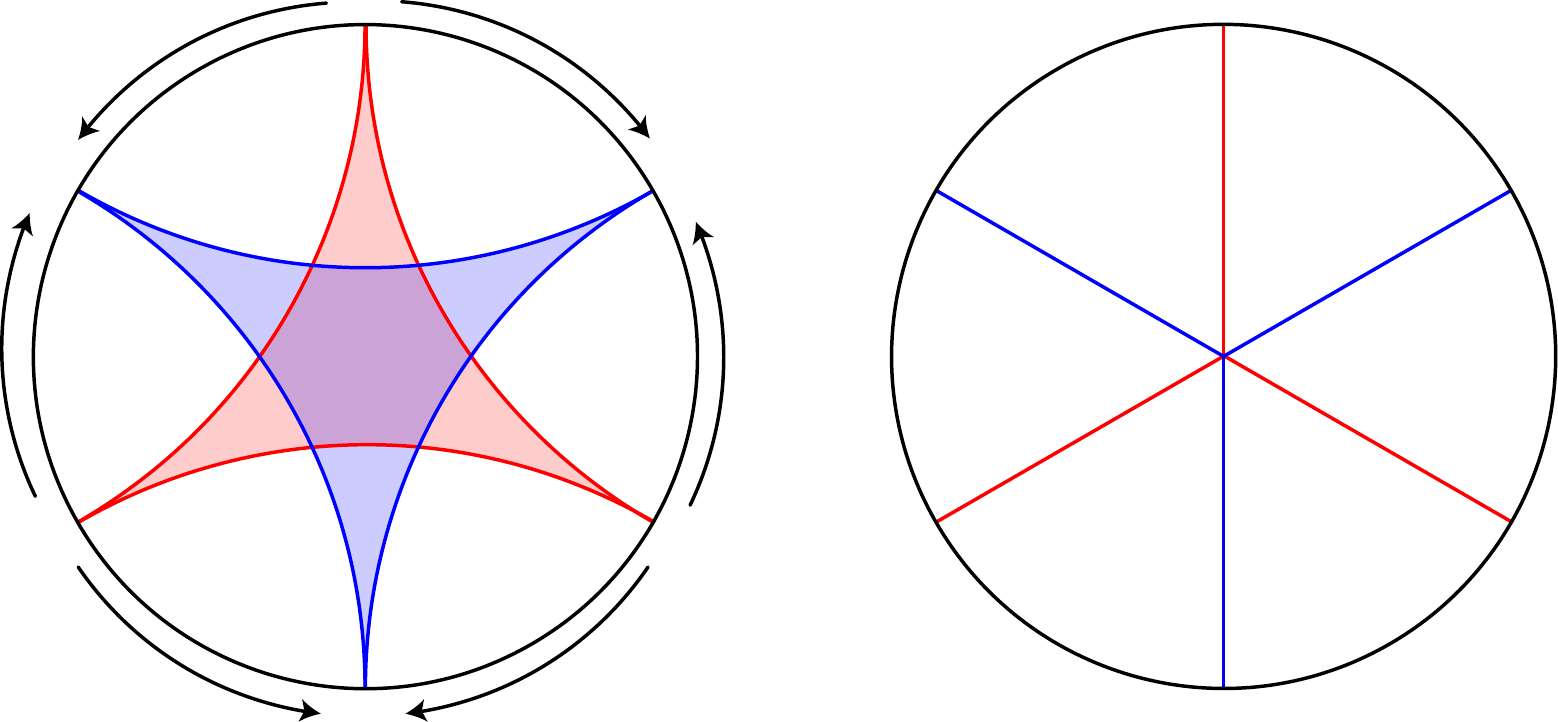}
    \caption{Left: lifts of two dual principal regions to $\wt L$, as well as the boundary dynamics of the action of a lift of a power of $f$ fixing the ends of two lifted principal regions. Right: the corresponding lifts of singular leaves of $\Lambda^\pm$.}
    \label{fig:principalregions}
\end{figure}

\subsection{Spun pseudo-Anosov maps are Handel--Miller}
We collect the remaining needed properties of the Handel--Miller map in this proposition:
\begin{proposition}
\label{prop:HM}
Let $h \colon L \to L$ be a Handel--Miller map with invariant laminations $\Lambda_\HM^\pm$. Then the collection of periodic leaves of $\Lambda_\HM^\pm$ is dense in $\Lambda_\HM^\pm$. 

If $\wt h\colon \wt L \to \wt L$ is any lift of a power of $h$ that fixes a point $\wt q$ of $\wt L$, then either $\wt q$ is contained in the closure of a lift of a principal region or the following holds:
$\wt q$ is the unique fixed point of $\wt h$ in $\wt L$, it is the point of intersection between two leaves $\ell^\pm \subset \wt \Lambda^\pm_\HM$, and $\partial \ell^\pm \subset \partial \wt L$ are exactly the fixed points of $\wt h^p$ on $\partial \wt L$, where $p$ is the smallest natural number such that $\wt h^p$ fixes the ends of $\ell^\pm$.
\end{proposition}

\begin{proof}
By \cite[Proposition 5.12]{CCF19}, the leaves of $\Lambda^\pm_\HM$ which border $\ms U_\mp$ are dense in $\Lambda^\pm_\HM$. By \cite[Theorem 6.5]{CCF19}, each such leaf is periodic. This proves the first claim.

For the second claim, note that the projection of $\wt q$ cannot lie, in particular, in
the negative escaping set. By the definition of principal regions we have $L=\ms
P_+\cup\ms U_-\cup \Lambda^+_\HM$,
so we conclude that $\wt q$ must lie in the closure of a lift of a principal region or in the lift of a leaf of $\Lambda^+_\HM$. We remark that the two cases are not mutually exclusive since positive principal regions are bordered by leaves of $\Lambda^+_\HM$.

Thus it suffices to consider only the case when $\wt q$ lies in a lifted leaf $\ell^+$ of $\wt \Lambda^+_\HM$ which does \emph{not} lie on the boundary of a lifted principal region. 

Let $p$ be the least natural number such that $\wt h^p$ fixes both ends of $\ell^+$. Let $\wt\sigma$ be the lift of a positive $h$-juncture in $L$ intersecting $\ell^+$. Since $h$ is  Handel--Miller and $\wt h^p$ fixes the ends of $\ell^+$, under iteration of $\wt h^{-p}$ we have that $\wt \sigma$ must converge to a leaf $\ell^-$ of $\wt \Lambda^-_\HM$ from one of its sides. Under iteration of $\wt h^p$, $\wt \sigma$ must converge to a point in $\del \wt L$ since its image in $L$ escapes to the positive ends of $L$ under iteration of $h$. In fact, the sequence converges to the boundary point of $\ell^+$ that lies to the same side of $\ell^+$ intersecting $\wt \sigma$.

Let $\wt\sigma'$ be another lifted $h$-juncture intersecting $\ell^+$ on the opposite side
of $\ell^-$ as $\wt \sigma$. Then under iteration of $\wt h^{-p}$, $\wt \sigma'$ converges
to a leaf $\ell^-_1$ of $\wt \Lambda^-_\HM$. By \cite[Corollary 6.15]{CCF19}, we must have that $\ell^-=\ell^-_1$ because otherwise $\ell^+$ would bound the lift of a principal region. This forces $\wt q$ to be the unique point in $\ell^+\cap \ell^-$. 

Furthermore, the picture of lifted $h$-junctures developed above shows that every point in $\wt L$ limits under iteration of $\wt h^{-p}$ to  $\wt q$ or a boundary point of $\ell^-$ in $\partial L$. It follows that $\wt q$ is the unique fixed point of $\wt h^p$. The statement about dynamics on $\del \wt L$ also follows.
\end{proof}

\begin{remark}
\Cref{prop:HM} is true without the assumption that $h$ is atoroidal, but in that case the lifts of principal regions are more complicated than the picture developed in \Cref{sec:HMbasics}.
\end{remark}

Recall from \Cref{sec:core_ent}, that if $\Lambda$ is a (possibly singular) lamination on $L$, then ${\partial}^2(\Lambda)$ denotes the 
closed, pairwise unlinked, $\pi_1(L)$--invariant subset of $\partial^2 \wt L$ obtained by taking the boundaries of \emph{leaf lines} in $\wt \Lambda \subset \wt L$. Note that if the lamination is nonsingular, then every leaf is regular and hence a leaf line.

\begin{theorem}\label{th:lam_relation}
Suppose that $f \colon L \to L$ is an spA map and $h$ is a homotopic Handel--Miller map. Then as $\pi_1(L)$--invariant subspaces of $\partial^2 \wt L$,
\[
\db(\Lambda^\pm) = \db(\Lambda_\HM^\pm).
\]
\end{theorem}

\begin{proof}
Combining \Cref{prop:HM} and \Cref{th:laminations} it suffices to prove that
$\db(\Lambda^\pm)$ and $\db(\Lambda_\HM^\pm)$ have the same set of periodic points
(i.e. points fixed under a lift of some power of the homotopy class of $f$).

Let $\ell$ be a periodic leaf line of either $\wt \Lambda^+$ or $\wt \Lambda^-$. By \Cref{rmk:periodic}, there is an $n\ge1$, a lift $\wt f^n$ of $f^n$, and a point $\wt p \in \ell$ such that $\wt f^n$ fixes $p$ and preserves $\ell$.  
By \Cref{prop:fix_to_boundary}, $\wt f^n$ has multi sink-source dynamics on $\del \wt L$. Let $\wt h$ be the corresponding lift of $h$ so that $\wt h$ and $\wt f$ have the same action on $\partial \wt L$. \Cref{lem:local_dynamics} and \Cref{lem:fixed_points} then imply that $\wt h^n$ fixes a point $\wt q$ in $\wt L$.
According to \Cref{prop:HM}, either $\wt q$ is the unique such point and we see that the
end points of $\ell$ must agree with those of the leaf of $\Lambda^\pm_\HM$ through $\wt
q$, or $\wt q$ is in the closure of the lift of a principal region. In the latter case,
\Cref{lem:principalboundaryaction} shows that the endpoints of $\ell$ are shared by a leaf
of $\wt \Lambda^\pm_\HM$ which is fixed under $\wt h^n$.
This shows that $\db(\Lambda^\pm) \subset \db(\Lambda_\HM^\pm)$.

If, instead, $\ell$ is a periodic leaf of (say) $\wt \Lambda_\HM^+$, then \Cref{prop:HM} furnishes a periodic point $\wt q$ contained in $\ell$ and also contained in a periodic leaf $\ell^-$ of $\wt \Lambda_\HM^-$, so that $\ell \cap \ell^- = \wt q$. Indeed, if $\ell$ does not bound a lifted principal region then there is nothing to say, and if $\ell$ does bound a lifted principal region then there are two possible choices of $\wt q$ and $\ell^-$.

Let $\wt h$ be a lift of $h$ and $n\ge 1$ so that $\wt h^n$ fixes both $\ell$ and $\ell^-$. 
Then \Cref{cor:gen_local_dynamics}
implies that the endpoints of $\ell$ are sinks and the endpoints of $\ell^-$ are sources, both for the action of $\wt h^n$ on $\DD$ and of $\wt f^n$ on $\DD$. 
So we may again apply \Cref{lem:fixed_points} to conclude that $\wt f^n$ has a 
fixed point $\wt p$, where $\wt f$ is the $\wt h$--compatible lift of $f$.
Note that $\wt f^n$ must fix the half-leaves of $\wt \Lambda^\pm$ from $\wt p$ because $\wt f^n$ has a nonempty fixed point set in $\del \wt L$. Hence by \Cref{prop:fix_to_boundary}, there exist leaves of $\wt \Lambda^+$ and $\wt \Lambda^-$ whose endpoints are, respectively, the sinks and sources of the action of $\wt f^n$.
This shows that $\db(\Lambda^\pm_\HM) \subset \db(\Lambda^\pm)$.
\end{proof}

\Cref{th:lam_relation} gives the precise relation between spA and Handel--Miller representatives. In particular, a spun pseudo-Anosov map $f \colon L \to L$ preserves the pair of singular laminations $\Lambda^\pm$, which according to \Cref{th:lam_relation} are singular versions of $\Lambda^\pm_\HM$ obtained by pinching principal regions to singular leaves (see \Cref{fig:principalregions}). It is in this precise sense that one can consider spA maps as `singular' Handel--Miller maps.

\begin{remark}
Handel--Miller theory applies more generally to endperiodic maps $f \colon L \to L$ for which $\partial L \neq \emptyset$, as long as each noncompact component of $\partial L$ joins a positive end to a negative end and contains no periodic points; see \cite[Hypothesis 3]{CCF19}. It would be interesting to generalize the theory of spA maps to cover this case as well. Potential difficulties include that $(1)$ \Cref{prop:foliation_facts}, which uses work of Fenley \cite{Fen09}, would need to be extended to the case with boundary, and $(2)$ to extend \Cref{th:foliation_cones} one needs a homeomorphism of a surface with boundary that sends every juncture class to its negative.
\end{remark}

\bibliography{vp3.2_spA.bib}
\bibliographystyle{amsalpha}

\end{document}